\documentclass[12pt,british]{article}

\usepackage{babel}
\usepackage{mathrsfs} 
\usepackage{amsmath,amsfonts,amssymb,amsthm,cases}
\usepackage{ucs} % Unicode support
\usepackage[utf8x]{inputenc} % UCS' UTF-8 driver is better than the LaTeX kernel's
\usepackage[T1]{fontenc} % The default font encoding only contains Latin characters

\usepackage{enumerate} %Oscar: Botei porcentagem nesse pacote aqui pra nao gerar erro quando compilo no meu pc

\usepackage{textcomp,url}
\usepackage{eucal}
\usepackage[noBBpl]{mathpazo}
\usepackage{mathabx}%[unsupported]
\usepackage[matrix,arrow]{xy}

\usepackage[%
breaklinks=true,
bookmarksnumbered = true,%     % Signets num√©rot√©s
colorlinks= true,%     % Liens en couleur
urlcolor= green,
anchorcolor = yellow,
citecolor=blue,%  % Couleur des liens externes
pdfborder= {0 0 0}%   % Style de bordure : ici, pas de bordure
]{hyperref}                   % Utilisation de HyperTeX

\makeatletter

\renewcommand{\p@enumii}{}

\def\@enum@{\list{\csname label\@enumctr\endcsname}%
          {\usecounter{\@enumctr}\def\makelabel##1{
\normalfont\ignorespaces\emph{{##1}~}}
\setlength{\labelsep}{3pt}
\setlength{\parsep}{0pt}
\setlength{\itemsep}{0pt}
\setlength{\leftmargin}{0pt}
\setlength{\labelwidth}{0pt}
\setlength{\listparindent}{\parindent}
\setlength{\itemsep}{0pt}
\setlength{\itemindent}{0pt}
\topsep=3pt plus 1pt minus 1 pt}}

\makeatother

\renewcommand{\epsilon}{\ensuremath{\varepsilon}}
\renewcommand{\phi}{\ensuremath{\varphi}}

\renewcommand{\to}{\ensuremath{\longrightarrow}}
\renewcommand{\mapsto}{\ensuremath{\longmapsto}}
\allowdisplaybreaks

\newcommand{\R}{\ensuremath{\mathbb R}}
\newcommand{\N}{\ensuremath{\mathbb N}}
\newcommand{\Z}{\ensuremath{\mathbb Z}}

\newcommand{\St}[1][2]{\ensuremath{\mathbb S}^{#1}}

\newcommand{\sn}[1][n]{\ensuremath{S_{{#1}}}}
\newcommand{\an}[1][n]{\ensuremath{A_{{#1}}}}

\DeclareRobustCommand*{\up}[1]{\textsuperscript{#1}}

\renewcommand{\ker}[1]{\ensuremath{\operatorname{\text{Ker}}\left({#1}\right)}}
\newcommand{\im}[1]{\ensuremath{\operatorname{\text{Im}}\left({#1}\right)}}

\newcommand{\id}{\ensuremath{\operatorname{\text{Id}}}}
\newcommand{\lhra}{\lhook\joinrel\longrightarrow}

\makeatletter

\def\@map#1#2[#3]{\mbox{$#1 \colon\thinspace #2 \to #3$}}
\def\map#1#2{\@ifnextchar [{\@map{#1}{#2}}{\@map{#1}{#2}[#2]}}

\makeatother

\newcommand{\brak}[1]{\ensuremath{\left\{ #1 \right\}}}
\newcommand{\ang}[1]{\ensuremath{\left\langle #1\right\rangle}}

\newcommand{\setangr}[2]{\ensuremath{\ang{#1 \,\left\lvert \, #2 \right.}}}

\newcommand{\setr}[2]{\ensuremath{\brak{#1 \,\left\lvert \, #2 \right.}}}
\newcommand{\setl}[2]{\ensuremath{\brak{\left. #1 \,\right\rvert \, #2}}}

\newtheoremstyle{theoremm}{}{}{\itshape}{}{\scshape}{.}{ }{}

\theoremstyle{theoremm}
\newtheorem{thm}{Theorem}
\newtheorem{lem}[thm]{Lemma}
\newtheorem{prop}[thm]{Proposition}
\newtheorem{cor}[thm]{Corollary}

\newtheoremstyle{reptheorem}{}{}{\itshape}{}{\scshape}{}{ }{\thmname{#1}\mathrm{#3}}
\theoremstyle{reptheorem}

\newtheoremstyle{remark}{}{}{}{}{\scshape}{.}{ }{}
\theoremstyle{remark}

\newtheorem*{defn}{Definition}
\newtheorem{rem}[thm]{Remark}
\newtheorem{rems}[thm]{Remarks}

\newtheoremstyle{comment}{}{}{\bfseries}{}{\bfseries}{:}{ }{}
\theoremstyle{comment}

\newcommand{\reth}[1]{Theorem~\protect\ref{th:#1}}
\newcommand{\relem}[1]{Lemma~\protect\ref{lem:#1}}
\newcommand{\repr}[1]{Proposition~\protect\ref{prop:#1}}
\newcommand{\reco}[1]{Corollary~\protect\ref{cor:#1}}
\newcommand{\resec}[1]{Section~\protect\ref{sec:#1}}
\newcommand{\rerem}[1]{Remark~\protect\ref{rem:#1}}
\newcommand{\rerems}[1]{Remarks~\protect\ref{rem:#1}}
\newcommand{\req}[1]{equation~(\protect\ref{eq:#1})}
\newcommand{\reqref}[1]{(\protect\ref{eq:#1})}

\begin{document}

\title{Quotients of the Artin braid groups\\ and crystallographic groups} 

\author{DACIBERG~LIMA~GON\c{C}ALVES\\
Departamento de Matem\'atica - IME-USP,\\
Caixa Postal~66281~-~Ag.~Cidade de S\~ao Paulo,\\ 
CEP:~05314-970 - S\~ao Paulo - SP - Brazil.\\
e-mail:~\url{dlgoncal@ime.usp.br}\vspace*{4mm}\\
JOHN~GUASCHI\\
Normandie Universit\'e, UNICAEN,\\
Laboratoire de Math\'ematiques Nicolas Oresme UMR CNRS~\textup{6139},\\
14032 Caen Cedex, France.\\
e-mail:~\url{john.guaschi@unicaen.fr}\vspace*{4mm}\\
OSCAR~OCAMPO~\\
Departamento de Matem\'atica - Instituto de Matem\'atica,\\
Universidade Federal da Bahia,\\
CEP:~40170-110 - Salvador - Ba - Brazil.\\
e-mail:~\url{oscaro@ufba.br}
}

%\date{17th May 2014}

\maketitle

\begin{abstract}
\noindent
\emph{Let $n\geq 3$. In this paper, we study the quotient group $B_n/[P_n,P_n]$ of the Artin braid group $B_{n}$ by the commutator subgroup of its pure Artin braid group $P_{n}$. We show that $B_n/[P_n,P_n]$ is a crystallographic group, and in the case $n=3$, we analyse explicitly some of its subgroups. We also prove that $B_n/[P_n,P_n]$ possesses torsion, and we show that there is a one-to-one correspondence between the conjugacy classes of the finite-order elements of $B_n/[P_n,P_n]$ with the conjugacy classes of the elements of odd order of the symmetric group $\sn$, and that the isomorphism class of any Abelian subgroup of odd order of $\sn$ is realised by a subgroup of $B_n/[P_n,P_n]$. Finally, we discuss the realisation of non-Abelian subgroups of $\sn$ of odd order as subgroups of $B_n/[P_n,P_n]$, and we show that the Frobenius group of order $21$, which is the smallest non-Abelian group of odd order, embeds in $B_n/[P_n,P_n]$ for all $n\geq 7$.}
 \end{abstract}

\section{Introduction}\label{sec:intro}

Let $n\in \N$. Quotients of the Artin braid group $B_{n}$ has been studied in various contexts, and may be used to study properties of $B_{n}$ itself. It is well known that one such quotient is the symmetric group $\sn$, which may be expressed in the form $B_{n}/\ang{\sigma_{1}^{2}}^{B_{n}}$, where $\sigma_{1},\ldots, \sigma_{n-1}$ are the standard generators of $B_{n}$ (see \resec{SecCG}), and $\ang{X}^{B_{n}}$ denotes the normal closure subgroup of a subset $X$ of $B_{n}$. Similar quotients of the form $B_{n}/\ang{\sigma_{1}^{m}}^{B_{n}}$, where $m\in \N$, were analysed by Coxeter in~\cite{Coxeter}, who showed that this quotient is finite if and only if $(n,m)\in \brak{(3,3),(3,4), (3,5), (4,3), (5,3)}$, and computed the quotient groups in each case, and by Marin in~\cite{Marin} in the case $(n,m)=(5,3)$ with the aim of studying cubic Hecke algebras. The Brunnian braid groups $\operatorname{Brun}_n$ have been studied in connection with homotopy groups of the $2$-sphere $\St$~\cite{BCWW,LW1,OcampoPhD} by considering quotients of $B_{n}$. For example, for all $n\geq 3$, there exists a subgroup $G_{n}$ of $\operatorname{Brun}_n$ that is normal in the Artin pure braid group $P_{n}$ such that the centre of $P_{n}/G_{n}$ is isomorphic to the direct product $\pi_{n}(\St)\times \Z$ (see~\cite[Theorem~1]{LW1} and~\cite[Theorem~4.3.4]{OcampoPhD}). 

In this paper, we study the quotient $B_n/[P_n,P_n]$ of $B_{n}$ for $n\geq 3$, where $[P_n,P_n]$ is the commutator subgroup of $P_{n}$. Our initial motivation emanates from the observation that $B_{3}/[P_3, P_3]$ is isomorphic to $B_3/\operatorname{Brun}_3$ (see~\cite[Corollary~2.1.4]{OcampoPhD} as well as~\cite[Section~5.2]{OcampoPhD} for other results about $B_3/\operatorname{Brun}_3$, and~\cite[Proposition~3.9]{LW1} and~\cite[Proposition~4.3.10(1)]{OcampoPhD} for a presentation of $B_3/\operatorname{Brun}_3$). The quotient $B_n/[P_n,P_n]$ belongs to a family of groups known as \emph{enhanced symmetric groups} (see~\cite[page~201]{Marin2}) and analysed in~\cite{Tits}. It also arises in the study of pseudo-symmetric braided categories by Panaite and Staic. They consider the quotient, denoted by $PS_n$, of $B_n$ by the normal subgroup generated by the relations $\sigma_i\sigma_{i+1}^{-1}\sigma_i=\sigma_{i+1}\sigma_i^{-1}\sigma_{i+1}$ for $i=1,2,\ldots,n-2$, and they show that it is isomorphic to $B_n/[P_n,P_n]$~\cite{PS}. The results that we obtain in this paper for $B_n/[P_n, P_n]$ are different in nature to those of~\cite{PS}, with the exception of some basic properties.

Crystallographic groups play an important r\^ole in the study of the groups of isometries of Euclidean spaces (see \resec{crystal} for precise definitions, as well as~\cite{Charlap,Dekimpe,Wolf} for more details). As we shall prove in \repr{cryst}, another reason for studying the quotient $B_n/[P_n,P_n]$ is the fact that it is a crystallographic group:
\begin{prop}\label{prop:cryst}
Let $n\geq 2$. There is a short exact sequence:
\begin{equation*}
1 \to \Z^{n(n-1)/2} \to  B_n/[P_n,P_n] \stackrel{\overline{\sigma}}{\to} \sn \to 1,
\end{equation*}
%where the associated integral representation \comj{say what this is before? Or modify the statement, and just say that $B_n/[P_n,P_n]$ is crystallographic?} of $\sn$ in $\operatorname{Aut}(\Z^{n(n-1)/2})$ induced by conjugation is faithful if $n\geq 3$.  In particular, if $n\geq 2$ then the group $B_n/[P_n,P_n]$ is crystallographic.
and the middle group $B_n/[P_n,P_n]$ is a crystallographic group.
\end{prop}

The aim of this paper is to analyse $B_n/[P_n,P_n]$ in more detail, notably its torsion, the conjugacy classes of its finite-order elements, and the realisation of abstract finite groups as subgroups of $B_n/[P_n,P_n]$. Since $B_{1}$ is trivial and $B_{2}$ is isomorphic to $\Z$, in what follows we shall suppose that $n\geq 3$. In \resec{crystal}, we recall the basic definitions and some results about crystallographic and Bieberbach groups. In \resec{SecCG}, we recall some standard information about $B_{n}$ and $P_{n}$, and using the fact that the quotient $B_{n}/P_{n}$ is isomorphic to $\sn$, we shall see that $B_{n}/[P_{n}, P_{n}]$ is an extension of the free Abelian group $P_{n}/[P_{n}, P_{n}]$ by $\sn$, and we shall compute the associated action, which will enable us to prove it is crystallographic. By analysing the action in more detail, we prove that the torsion of $B_n/[P_n, P_n]$ is odd:
\begin{thm}\label{th:2TorCryst}
If $n\geq 3$ then the quotient group $B_n/[P_n, P_n]$ has no finite-order element of even order. 
\end{thm}
By restricting the short exact sequence involving $B_{n}/[P_{n}, P_{n}]$, $P_{n}/[P_{n}, P_{n}]$ and $\sn$ to $2$-subgroups of the latter (see \req{sesholo}), we are able to construct Bieberbach groups of dimension $n(n-1)/2$ (which is the rank of $P_{n}/[P_{n}, P_{n}]$), and show that there exist flat manifolds of the same dimension whose holonomy group is the given $2$-subgroup (see \reth{TeoAKBraids}).

In \resec{SecBCG}, we analyse the torsion of $B_n/[P_n, P_n]$ in more detail. In order to do so, we shall make use of the induced action of certain elements $\alpha_{0,r}$ of $B_n/[P_n, P_n]$, where $2\leq r\leq n$, on the basis $(A_{i,j})_{1\leq i<j\leq n}$ of $P_n/[P_n, P_n]$. The structure of the corresponding orbits is very rigid, and allows us to express the existence of elements of $B_n/[P_n, P_n]$ of order $n$ in terms of the existence of solutions of a certain linear system. It will follow from this that $B_n/[P_n, P_n]$ has infinitely many elements of order $n$ (see \repr{Torcaonimpar}). We then show that if $1\leq n\leq m$, the standard injective homomorphism of $B_{n}$ in $B_{m}$ induces a injective homomorphism of $B_n/[P_n, P_n]$ in $B_m/[P_m, P_m]$:
\begin{thm}\label{th:oddtorsion}
Let  $m$ and $n$ be integers such that $2\leq n\leq m$. 
\begin{enumerate}
\item\label{it:oddtorsiona} Consider the injective homomorphism $\map{\iota}{B_{n}}[B_m]$ defined by $\iota(\sigma_i)=  \sigma_i$ for all $1\leq i\leq n-1$. Then the induced homomorphism $\map{\overline{\iota}}{B_n/[P_n, P_n]}[B_m/[P_m, P_m]]$ of the corresponding quotient groups is injective. 
\item\label{it:oddtorsionb} If  $n\geq 3$ and $n$ is odd then  $B_m/[P_m, P_m]$ possesses elements of order $n$.
Further, there exists such an element whose permutation is an $n$-cycle.
\item\label{it:oddtorsionc} Let $n_1,n_2,\ldots,n_t$ be odd integers greater than or equal to $3$ for which $\sum_{i=1}^{t} \, n_i\leq m$. Then $B_m/[P_m, P_m]$ possesses elements of order $\operatorname{lcm}(n_1,\ldots,n_{t})$. Further, there exists such an element whose cycle type is $(n_1, \ldots, n_t)$.
\end{enumerate}
\end{thm}
Part~(\ref{it:oddtorsionb}) follows from part~(\ref{it:oddtorsiona}) and \repr{Torcaonimpar}. In the course of the proof of \reth{oddtorsion}, we shall see that the direct product of the groups of the form $B_{n_{i}}/[P_{n_{i}}, P_{n_{i}}]$ injects into $B_m/[P_m, P_m]$, which will enable us to prove part~(\ref{it:oddtorsionc}). One consequence of Theorems~\ref{th:2TorCryst} and~\ref{th:oddtorsion} is the characterisation of the torsion of $B_n/[P_n,P_n]$ as that of the odd torsion of the symmetric group $\sn$:
\begin{cor}\label{cor:torsion}
Let $n\geq 3$. The torsion of the quotient $B_n/[P_n, P_n]$ is equal to the odd torsion of the symmetric group $\sn$. Moreover, given an element $\theta\in \sn$ of odd order $r$, there exists $\beta\in B_n/[P_n, P_n]$ of order $r$ such that $\overline{\sigma}(\beta)=\theta$. So given any cyclic subgroup $H$ of $\sn$ of odd order $r$, there exists a finite-order subgroup $\widetilde{H}$ of $B_n/[P_n, P_n]$ such that $\overline{\sigma}(\widetilde{H})=H$.
\end{cor}

In \resec{subgroupsb3}, we focus on the simplest non-trivial case, that of $B_{3}/[P_{3},P_{3}]$, and we describe the structure of the preimages of the subgroups of $\sn[3]$ under the induced homomorphism $B_{3}/[P_{3},P_{3}]\to \sn[3]$. In the cases where these preimages are Bieberbach groups, we describe the corresponding flat $3$-manifold. We also carry out this analysis for the group $B_{3}/[P_{3},P_{3}]$ itself, and identify it in the international tables of crystallographic groups given in~\cite{BBNWS,HL}, as well as for the quotient of $B_3/[P_3, P_3]$ by the subgroup generated by the class of the full-twist braid. 

In \resec{SecBCGconj}, we study the conjugacy classes of the elements and the cyclic subgroups of $B_n/[P_n, P_n]$. This is achieved in Propositions~\ref{prop:i},~\ref{prop:ii} and~\ref{prop:iii} by studying in detail the action of certain elements $\delta_{r,k}$ and $\alpha_{r,k}$ (the latter being a generalisation of $\alpha_{r,0}$) on the basis $(A_{i,j})_{1\leq i<j\leq n}$ of $P_n/[P_n, P_n]$. It is straightforward to see that if any two elements of $B_n/[P_n, P_n]$ are conjugate then their permutations have the same cycle type, and the use of these propositions and a specific product of certain $\delta_{r,k}$ enables us to prove the converse:
\begin{thm}\label{th:classconj}
Let $n\geq 3$, and let $k\geq 3$ be odd. Two elements of $B_n/[P_n, P_n]$ of order $k$ are conjugate if and only if their permutations have the same cycle type. Thus two finite cyclic subgroups of $B_n/[P_n, P_n]$ of order $k$ are conjugate if and only if their images under $\overline{\sigma}$ are conjugate in $\sn$. 
%The number of conjugacy classes of elements of order $k$ in $B_n/[P_n, P_n]$ is equal to the number of conjugacy classes of permutations of order $k$ in $\sn$. In particular, the number of conjugacy classes of cyclic subgroups  of order $k$ in $B_n/[P_n, P_n]$ is equal to the number of conjugacy classes of cyclic subgroups of order $k$ in $\sn$.
\end{thm}
Consequently, given $n\geq 3$, we may determine the number of conjugacy classes of elements of odd order $k$ in $B_n/[P_n, P_n]$.

From \relem{eleconj}, it follows that the set of isomorphism classes of the finite subgroups of $B_n/[P_n, P_n]$ is contained in the corresponding set of  finite subgroups of $\sn$ of odd order. One may ask whether this inclusion is strict or not. As we shall see in \reco{torsion}, any cyclic subgroup of $\sn$ of odd order is realised as a subgroup of $B_n/[P_n, P_n]$. 
%\comj{This added:} 
Combining \reth{oddtorsion}(\ref{it:oddtorsionc}) with a result of~\cite{Ho}, this result may be extended to the Abelian subgroups of $\sn$:

\begin{thm}\label{th:realAbelian}
Let $n\geq 3$. Then there is a a one-to-one correspondence between the isomorphism classes of the finite Abelian subgroups of $B_{n}/[P_{n},P_{n}]$ and the isomorphism classes of the Abelian subgroups of $\sn$ of odd order.
\end{thm}

In \resec{nonAbelian}, we turn our attention to what is probably a more difficult open problem, namely the realisation of finite non-Abelian groups of $\sn$ as subgroups of the group $B_{n}/[P_{n},P_{n}]$. As a initial experiment, we consider the smallest value of $n$, $n=7$, for which $\sn$ possesses a non-Abelian subgroup of odd order. This subgroup is isomorphic to the Frobenius group of order $21$ that we denote by $\mathcal{F}$. We show that $\mathcal{F}$ is indeed realised as a subgroup of $B_7/[P_7, P_7]$.
\begin{thm}\label{th:frob}
The quotient group $B_7/[P_7, P_7]$ possesses a subgroup isomorphic to the Frobenius group $\mathcal{F}$.
\end{thm}
It then follows from \reth{oddtorsion} that $\mathcal{F}$ is realised as a subgroup of $B_{n}/[P_{n},P_{n}]$ for all $n\geq 7$. In \repr{frob}, we prove that $B_7/[P_7, P_7]$ admits a single conjugacy class of subgroups isomorphic to $\mathcal{F}$. We remark that we do not currently know of an example of a subgroup of odd order of $\sn$ whose isomorphism class is not represented by a subgroup of $B_n/[P_n, P_n]$.

%It turns out that the groups  $B_n/[P_n,P_n]$ are crystallographic groups. This leads
% to an interesting connection, once the crystallographic groups remains considered object of main interest and
%continuous to be studied.
%
% The main results of   this work are: 
%
%
%This work is divided into ??? sections

\subsection*{Acknowledgements}

The initial ideas for this paper occurred during the stay of the third author at the Laboratoire de Math\'ematiques Nicolas Oresme from the  5\textsuperscript{th} January to the 5\textsuperscript{th} June 2012, and were developed during the period 2012--2014 when the third author was at the Departamento de Matem\'atica do IME~--~Universidade de S\~ao Paulo and was partially supported by a project grant n\up{o}~2008/58122-6 from FAPESP, and also by a project grant n\up{o}~151161/2013-5 from CNPq. Further work took place during the visit of the first author to the Departamento de Matem\'atica, Universidade Federal de Bahia during the period 16\up{th}--20\up{th}~May 2014 and also during the visit of the second author to the Departamento de Matem\'atica do IME~--~Universidade de S\~ao Paulo during the periods 10\up{th}~July--2\up{nd}~August 2014, and to the Departamento de Matem\'atica do IME~--~Universidade de S\~ao Paulo and the Departamento de Matem\'atica, Universidade Federal de Bahia during the period 1\up{st}--17\up{th}~November 2014, and was partially supported by the international CNRS/FAPESP programme n\up{o}~226555.

\section{Crystallographic and Bieberbach groups}\label{sec:crystal}

In this section, we recall briefly the definition of crystallographic and Bieberbach groups, and we characterise crystallographic groups in terms of a representation that arises from certain group extensions whose kernel is a free Abelian group of finite rank and whose quotient is finite. We also review some results concerning Bieberbach groups and the fundamental groups of flat Riemannian manifolds. For more details, see~\cite[Section~I.1.1]{Charlap},~\cite[Section~2.1]{Dekimpe} or~\cite[Chapter~3]{Wolf}. From now on, we identify $\operatorname{Aut}(\Z^n)$ with $\operatorname{GL}(n,\Z)$. 

\begin{defn}\label{DefCristGeo}
A discrete and uniform subgroup  $\Pi$ of $\R^n\rtimes \operatorname{O}(n,\R)\subseteq \operatorname{Aff}(\R^n)$ is said to be a \emph{crystallographic group of dimension $n$}. If in addition $\Pi$ is torsion free then $\Pi$ is called a \emph{Bieberbach group} of dimension $n$.
\end{defn}

\begin{defn}\label{RepInt}  Let  $\Phi$ be a group.  
An  \emph{integral representation of rank $n$ of  $\Phi$} is defined to be a homomorphism $\map{\Theta}{\Phi}[\operatorname{Aut}(\Z^n)]$.  Two such representations are said to be \emph{equivalent} if their images are conjugate in $\operatorname{Aut}(\Z^n)$. We say that $\Theta$ is a \emph{faithful representation} if it is injective.
\end{defn}

The following characterisation of crystallographic groups seems to be well known to the experts in the field. Since we did not find a suitable reference, we give a short proof. 

\begin{lem}\label{lem:DefC2}
Let $\Pi$ be a group. Then $\Pi$ is a crystallographic group if and only if there exist an integer $n\in \N$ and a short exact sequence 
\begin{equation}\label{eq:SeqCrist2}
\xymatrix{
0 \ar[r] & \Z^n \ar[r] & \Pi \ar[r]^{\zeta} & \Phi \ar[r] & 1}	
\end{equation}
such that:
\begin{enumerate}
\item\label{it:DefC2a} $\Phi$ is finite, and
\item\label{it:DefC2b} the integral representation $\map{\Theta}{\Phi}[\operatorname{Aut}(\Z^n)]$, induced by 
conjugation on $\Z^n$ and defined by $\Theta(\phi)(x)=\pi x \pi^{-1}$, where $x\in \Z^{n}$, $\phi\in \Phi$ and $\pi\in \Pi$ is such that $\zeta(\pi)=\phi$, is faithful. 
\end{enumerate}
\end{lem}

\begin{defn}
If $\Pi$ is a crystallographic group, the integer $n$ that appears in the statement of~\relem{DefC2} is called the \emph{dimension} of $\Pi$, the finite group $\Phi$ is called the \emph{holonomy group} of $\Pi$, and the integral representation $\map{\Theta}{\Phi}[\operatorname{Aut}(\Z^{n})]$ is called the \emph{holonomy representation of  $\Pi$}.
\end{defn}

\begin{proof}[Proof of \relem{DefC2}]
Let $\Phi$ and $\Pi$ be groups, and suppose that there exist $n\in \N$ and a short exact sequence of the form~\reqref{SeqCrist2} such that conditions~(\ref{it:DefC2a}) and~(\ref{it:DefC2b}) hold. Assume on the contrary that $\Pi$ is not crystallographic. The characterisation of~\cite[Theorem~2.1.4]{Dekimpe} implies that $\Z^n$ is not a maximal Abelian subgroup of $\Pi$, in other words, there exists an Abelian group $A$ for which $\Z^n\subsetneqq A\subseteq \Pi$. Let $a\in A\setminus \Z^n$. Then $\zeta(a)\neq 1$ and $\Theta(\zeta(a))(x)=axa^{-1}=x$ for all $x\in \Z^{n}$. Hence $\Theta(\zeta(a))=\id_{\Z^{n}}$,  which contradicts the hypothesis that $\Theta$ is injective. We conclude that $\Pi$ is a crystallographic group of dimension $n$ with holonomy $\Phi$.

The converse follows from the paragraph preceding~\cite[Definition~I.6.2]{Charlap}, since the short exact sequence~\reqref{SeqCrist2} gives rise to an integral representation  
$\map{\Theta}{\Phi}[\operatorname{Aut}(\Z^{n})]$ that is faithful by~\cite[Proposition~I.6.1]{Charlap}. 
\end{proof}

The following lemma will be very useful in what follows.
\begin{lem}\label{lem:eleconj}
Let $G,G'$ be groups, and let $\map{f}{G}[G']$ be a homomorphism whose kernel is torsion free. If $K$ is a finite subgroup of $G$ then the restriction $\map{f\left\lvert_{K}\right.}{K}[f(K)]$ of $f$ to $K$ is an isomorphism. In particular, with the notation of the statement of \relem{DefC2}, if $\Pi$ is a crystallographic group then the restriction $\map{\zeta\left\lvert_{K}\right.}{K}[\zeta(K)]$ of $\zeta$ to any finite subgroup $K$ of $\Pi$ is an isomorphism.
\end{lem}

\begin{proof}
Since $\ker{f}$ is torsion free, the restriction of $f$ to the finite subgroup $K$ is injective, which yields the first part, and the second part then follows directly. 
\end{proof}

\begin{cor}\label{cor:CSG}
Let $\Pi$ be a crystallographic group of dimension $n$ and holonomy group $\Phi$, and let $H$ be a subgroup of $\Phi$.
Then there exists a crystallographic subgroup of $\Pi$ of dimension $n$ with holonomy group $H$.
\end{cor}

\begin{proof}
The result follows by considering the short exact sequence~\reqref{SeqCrist2}, and by applying \relem{DefC2} to the subgroup $\zeta^{-1}(H)$ of $\Pi$.
\end{proof}

\begin{defn}
A Riemannian manifold $M$ is called \emph{flat} if it has zero curvature at every point. 
\end{defn}

As a consequence of the first Bieberbach Theorem, there is a correspondence between Bieberbach groups and fundamental groups of closed flat Riemannian manifolds (see~\cite[Theorem~2.1.1]{Dekimpe} and the paragraph that follows it). We recall that the flat manifold determined by a Bieberbach group $\Pi$ is orientable if and only if the integral representation $\map{\Theta}{\Phi}[\operatorname{GL}(n,\Z)]$ satisfies $\im{\Theta}\subseteq \operatorname{SO}(n,\Z)$. This being the case, we say that $\Pi$ is \emph{an orientable Bieberbach group}. By~\cite[Corollary~3.4.6]{Wolf}, the holonomy group of a flat manifold $M$ is isomorphic to the group $\Phi$.

It is a natural problem to classify the finite groups that are the holonomy group of a flat manifold. The answer was given by L.~Auslander and M.~Kuranishi in 1957.

\begin{thm}[Auslander and Kuranishi~{\cite[Theorem~3.4.8]{Wolf}, \cite[Theorem~III.5.2]{Charlap}}]\label{th:TeoAK}
Any finite group is the holonomy group of some flat manifold.
\end{thm}

%%%%%%%%%%%%%%%%%%%%%%%%%%%%%%%%%%%%%%%% SECAO %%%%%%%%%%%%%%%%%%%%%%%%%%%%%%%%%%%%%%%%

\section{Artin braid groups and crystallographic groups}\label{sec:SecCG}

In this section we prove \repr{cryst} and \reth{2TorCryst}, namely that if $n\geq 3$ then the quotient group $B_n/[P_n,P_n]$ of the Artin braid group $B_{n}$ by the commutator subgroup $[P_{n},P_{n}]$ of its pure braid subgroup $P_{n}$ is crystallographic and does not have $2$-torsion. As we shall see in \resec{SecBCG}, $B_n/[P_n,P_n]$ possesses (odd) torsion. We first recall some facts about the Artin braid group $B_{n}$ on $n$ strings. We refer the reader to~\cite{Ha} for more details. It is well known that $B_{n}$ possesses a presentation with generators $\sigma_{1},\ldots,\sigma_{n-1}$ that are subject to the following relations:
\begin{gather}
\text{$\sigma_{i} \sigma_{j} = \sigma_{j}  \sigma_{i}$ for all $1\leq i<j\leq n-1$}\label{eq:artin1}\\
\text{$\sigma_{i+1} \sigma_{i} \sigma_{i+1} =\sigma_{i} \sigma_{i+1} \sigma_{i}$ for all $1\leq i\leq n-2$.}\label{eq:artin2}
\end{gather}
Let $\map{\sigma}{B_{n}}[\sn]$ be the homomorphism defined on the given generators of $B_{n}$ by $\sigma(\sigma_{i})=(i,i+1)$ for all $1\leq i\leq n-1$. Just as for braids, we read permutations from left to right so that if $\alpha, \beta \in \sn$ then their product is defined by $\alpha\cdot \beta (i)=\beta(\alpha(i))$ for $i=1,2,\ldots, n$. The pure braid group $P_{n}$ on $n$ strings is defined to be the kernel of $\sigma$, from which we obtain the following short exact sequence:
\begin{equation}\label{eq:sespn}
1 \to P_n \to  B_n \stackrel{\sigma}{\to} \sn \to 1.
\end{equation}
A generating set of $P_{n}$ is given by $\brak{A_{i,j}}_{1\leq i<j\leq n}$, where:
\begin{equation}\label{eq:defaij}
A_{i,j}=\sigma_{j-1}\cdots \sigma_{i+1}\sigma_{i}^{2} \sigma_{i+1}^{-1}\cdots \sigma_{j-1}^{-1}.
\end{equation}
Relations~\reqref{artin1} and~\reqref{artin2} may be used to show that:
\begin{equation}\label{eq:defaijalt}
A_{i,j}=\sigma_{i}^{-1}\cdots \sigma_{j-2}^{-1}\sigma_{j-1}^{2} \sigma_{j-2}\cdots \sigma_{i}.
\end{equation}
It follows from the presentation of $P_{n}$ given  in~\cite{Ha} that $P_n/[P_n,P_n]$ is isomorphic to $\Z^{n(n-1)/2}$, and that a basis is given by the $A_{i,j}$, where $1\leq i<j\leq n$, and where by abuse of notation, the $[P_n, P_n]$-coset of $A_{i,j}$ will also be denoted by $A_{i,j}$. Using \req{sespn}, we obtain the following short exact sequence:
\begin{equation}\label{eq:sespnquot}
1 \to  P_n/[P_n,P_n] \to  B_n/[P_n,P_n] \stackrel{\overline{\sigma}}{\to} \sn \to 1,
\end{equation}
where $\map{\overline{\sigma}}{B_n/[P_n,P_n]}[\sn]$ is the homomorphism induced by $\sigma$. This short exact may also be found in~\cite[Proposition~3.2]{PS}.

Since $B_1$ is the trivial group and $B_2/[P_2, P_2]\cong \Z$, we shall suppose in most of this paper that $n\geq 3$. 
%A presentation of $P_{n}$ is given in~\cite{Ha}. A generating set of $P_{n}$ possesses a generating set of the form $\brak{A_{i,j}}_{1\leq i<j\leq n}$, where~\cite{Ha}:
%\begin{equation}\label{eq:defaij}
%A_{i,j}=\sigma_{j-1}\cdots \sigma_{i+1}\sigma_{i}^{2} \sigma_{i+1}^{-1}\cdots \sigma_{j-1}^{-1}.
%\end{equation}
%It follows from the presentation of $P_{n}$ given in~\cite{Ha} that $P_n/[P_n,P_n]\cong \Z^{n(n-1)/2}$, and that a basis of $P_n/[P_n,P_n]$ is given by the cosets of the $A_{i,j}$, where $1\leq i<j\leq n$. By abuse of notation, we denote these cosets by $A_{i,j}$.%\cite[Chapter 3, Theorem 3.8]{KM}. 
We shall be interested in the action by conjugation of $B_n$ on $P_n$ and on $P_n/[P_n, P_n]$. Recall from \cite[Lemma~3.1]{LW1} (see also \cite[Proposition~3.7, Chapter 3]{KM}) that for all $1\leq k\leq n-1$ and for all $1\leq i<j\leq n$,
\begin{equation*}%\label{eq:conjugAij}
\sigma_kA_{i,j}\sigma_k^{-1}=
\begin{cases}
A_{i,j} & \text{if $k\neq i-1, i, j-1, j$}\\
A_{i,k+1} & \text{if $j=k$}\\
A_{i,k+1}^{-1}A_{i,k}A_{i,k+1} & \text{if $j=k+1$ and $i<k$}\\
A_{k,k+1} & \text{if $j=k+1$ and $i=k$}\\
A_{i+1,j} & \text{if $i=k<j-1$}\\
A_{k+1,j}^{-1}A_{k,j}A_{k+1,j} & \text{if $i=k+1$}.
\end{cases}
\end{equation*}
So the induced action of $B_{n}$ on $P_n/[P_n, P_n]$ is given by
\begin{equation}\label{eq:conjugAij2}
\sigma_kA_{i,j}\sigma_k^{-1}=
\begin{cases}
A_{i,j} & \text{if $k\neq i-1, i, j-1, j$}\\
A_{i,k+1} & \text{if $j=k$}\\
A_{i,k} & \text{if $j=k+1$ and $i<k$}\\
A_{k,k+1} & \text{if $j=k+1$ and $i=k$}\\
A_{i+1,j} & \text{if $i=k<j-1$}\\
A_{k,j} & \text{if $i=k+1$}.
\end{cases}
\end{equation}
A study of this action now allows us to prove \repr{cryst}. 

%\begin{prop}\label{prop:cryst}
%Let $n\geq 2$. There is a short exact sequence:
%\begin{equation*}
%1 \to \Z^{n(n-1)/2} \to  B_n/[P_n,P_n] \stackrel{\overline{\sigma}}{\to} \sn \to 1,
%\end{equation*}
%where the associated integral representation of $\sn$ in $\operatorname{Aut}(\Z^{n(n-1)/2})$ induced by conjugation is faithful if $n\geq 3$.  In particular, if $n\geq 2$ then the group $B_n/[P_n,P_n]$ is crystallographic.
%\end{prop}

\begin{proof}[Proof of \repr{cryst}]
Suppose first that $n=2$. Since $B_2=\Z$ and $[P_2, P_2]=1$, we obtain $B_2/[P_2, P_2]=B_2\cong \Z$, and thus the group $B_2/[P_2, P_2]$ is crystallographic. So assume that $n\geq 3$, and consider the short exact sequence~\reqref{sespnquot}. We shall show that the induced action $\map{\phi}{\sn}[\operatorname{Aut}(\Z^{n(n-1)/2})]$ is injective, from which it will follow that the group $B_{n}/[P_{n},P_{n}]$ is crystallographic. Using \req{conjugAij2}, if $1\leq i<j\leq n$, the automorphisms induced by the elements $\sigma_1$ and $\sigma_2$ of $B_{n}$ are given by:
\begin{equation*}
\sigma_1 A_{i,j} \sigma_1^{-1} = 
\begin{cases}
A_{2,j}  & \text{if $i=1$ and $j\geq 3$}\\
A_{1,j}  & \text{if $i=2$ and $j\geq 3$}\\
A_{i,j}  & \text{otherwise}
\end{cases}
\quad\text{and}\quad
\sigma_2 A_{i,j} \sigma_2^{-1} = 
\begin{cases}
A_{3,j}  & \text{if $i=2$ and $j\geq 4$}\\
A_{2,j}  & \text{if $i=3$ and $j\geq 4$}\\
A_{1,3}  & \text{if $i=1$ and $j=2$}\\
A_{1,2}  & \text{if $i=1$ and $ j=3$}\\
A_{i,j}  & \text{otherwise}.
\end{cases}
\end{equation*}
These automorphisms are distinct and non trivial, so the image $\im{\phi}$ of $\phi$ possesses at least three elements. Applying the First Isomorphism Theorem, the kernel $\ker{\phi}$ of $\phi$ is thus a normal subgroup of $\sn$ whose order is bounded above by $n!/3$. If $n\neq 4$, the only normal subgroups of $\sn$ are the trivial subgroup, the alternating subgroup $\an$ and $\sn$ itself, from which we conclude that $\ker{\phi}$ is trivial as required. Now suppose that $n=4$. The same argument applies, but additionally, $\ker{\phi}$ may be isomorphic to the Klein group $\Z_2\oplus \Z_2$, in which case $\im{\phi}$ is of order $6$. But by \req{conjugAij2}, the action of $\sigma_3\sigma_2\sigma_1$ on the basis elements of $P_4 / [P_4, P_4]$ is given by:
\begin{equation*}
\text{$A_{1,2}\mapsto A_{1,4}\mapsto A_{3,4}\mapsto A_{2,3}\mapsto A_{1,2}$ and $A_{1,3}\mapsto A_{2,4}\mapsto A_{1,3}$.}
\end{equation*}
Thus $\phi(\sigma_3\sigma_2\sigma_1)$ is an element of $\im{\phi}$ of order $4$, so $\im{\phi}$ cannot be of order $6$. Once more we see that $\ker{\phi}$ is trivial, and thus the associated integral representation is faithful for all $n\geq 3$.
%By \req{conjugAij2}, the action of $\sigma_3\sigma_2\sigma_1$ on the basis elements of $P_4 / [P_4, P_4]$ is given by:
%\begin{equation*}
%\text{$A_{1,2}\mapsto A_{1,4}\mapsto A_{3,4}\mapsto A_{2,3}\mapsto A_{1,2}$ and $A_{1,3}\mapsto A_{2,4}\mapsto A_{1,3}$,}
%\end{equation*}
%so $\phi(\sigma_3\sigma_2\sigma_1)$ is of order $4$. Now the isomorphism classes of the non-trivial normal subgroups of $\sn[4]$ are $\an[4]$ and $\Z_2\oplus \Z_2$, but the corresponding quotients do not contain an element of order $4$, from which it follows that $\ker{\phi}$ is once again trivial. 
It then follows from \relem{DefC2} that the group $B_n/[P_n,P_n]$ is crystallographic. 
%and so, choosing an appropriate order of the set of generators we have the following matrix representing $\phi(\sigma_3\sigma_2\sigma_1)$
%\begin{equation*}
%\begin{pmatrix}
%0 & 0 & 0 & 1 & 0 & 0\\
%1 & 0 & 0 & 0 & 0 & 0\\
%0 & 1 & 0 & 0 & 0 & 0\\
%0 & 0 & 1 & 0 & 0 & 0\\
%0 & 0 & 0 & 0 & 0 & 1\\
%0 & 0 & 0 & 0 & 1 & 0
%\end{pmatrix}
%\end{equation*}
%Hence, we obtain the desired result. 
\end{proof}

Using~\repr{cryst} and~\reco{CSG}, we may produce other crystallographic groups as follows. 
Let  $H$ be a subgroup of $\sn$, and consider the following short exact sequence: 
\begin{equation}\label{eq:sesholo}
1 \to \frac{P_n}{[P_n, P_n]} \to \widetilde{H}_n \stackrel{\overline{\sigma}}{\to} H \to 1
\end{equation}
induced by that of \req{sespnquot}, where $\widetilde{H}_n$ is defined by: 
\begin{equation}\label{eq:BCG3}
\widetilde{H}_n=\frac{\sigma^{-1}(H)}{[P_n, P_n]}.
\end{equation}
%\begin{equation*}%\label{BCG3}
%\mathcal{BCG}(n,H)=\frac{\sigma^{-1}(H)}{[P_n, P_n]}.
%\end{equation*}

The following corollary is then a consequence of \reco{CSG} and \repr{cryst}.

\begin{cor}\label{cor:CorolBCG3}
Let $n\geq 3$, and let $H$ be a subgroup of $\sn$. Then the group $\widetilde{H}_n$ defined by \req{BCG3}
is a crystallographic group of dimension $n(n-1)/2$ with holonomy group $H$.
\end{cor}

Our next goal is to prove \reth{2TorCryst}, that the quotient groups $B_n/[P_n,P_n]$ do not have $2$-torsion.

%\begin{rem}
%In case of two strings,  since $B_2=\Z$ and $[P_2, P_2]=1$  we have 
%$$
%B_2/[P_2, P_2]=B_2\cong \Z.
%$$
%\noindent Therefore the group $B_2/[P_2, P_2]$ is certainly crystallographic.
%\end{rem}

%\begin{thm}\label{th:2TorCryst}
%If $n\geq 3$ then the quotient group $B_n/[P_n, P_n]$ has no finite-order element of even order. 
%\end{thm}

\begin{proof}[Proof of \reth{2TorCryst}]
Let $n\geq 3$. Suppose on the contrary that there exists $\beta\in B_{n}$ whose $[P_n, P_n]$-coset, which we also denote by $\beta$, is of even order in $B_n/[P_n, P_n]$. By taking a power of $\beta$ if necessary, we may suppose that $\beta$ is of order $2$ in $B_n/[P_n, P_n]$. Since $P_n/[P_n, P_n]$ is torsion free, it follows that $\beta \in B_{n}\setminus P_{n}$. Conjugating $\beta$ by an element of $B_{n}$ if necessary, we may suppose that $\sigma(\beta)=(1,2)(3,4)\cdots(k,k+1)$, where $1\leq k \leq n-1$ and $k$ is odd. Thus $\beta^{2}\in P_{n}$. Let $\alpha=\sigma_1\sigma_3\cdots\sigma_{k-2}\sigma_k$. Then $\sigma(\beta)=\sigma(\alpha)$, thus $N=\beta \alpha^{-1}$ belongs to $P_{n}$, and
%possesses an element $\beta$ of order $2$. Since $P_n/[P_n, P_n]$ is torsion free
%Then the permutation $\sigma(\beta)$ of $\beta$ has the same cyclic structure of the permutation 
%$(1,2)(3,4)\cdots(k,k+1)$, for some $k$ odd, $1\leq k \leq n-1$.
%So the two permutations are conjugated, therefore there exists permutation $\lambda$ such that 
%$\lambda \sigma(\beta) \lambda^{-1}= (1,2)(3,4)\cdots(k,k+1)$. Call $\lambda'$ any element of 
%the group which  projects to $\lambda$. Then follows that $\beta'=\lambda' \beta \lambda'^{-1}$ has permutation associated 
%$(1,2)(3,4)\cdots(k,k+1)$.
%(exist   there exists an element $\beta'$ such that $\beta'$ and $\beta$ are conjugated
%and the permutation associated with $\beta'$ is exactly $(1,2)(3,4)\cdots(k,k+1)$,
%for some $k$ odd, $1\leq k \leq n-1$.
%
%Note that $\beta'$ has order 2 provided $\beta$ has such an order and also note that the braid $\alpha=\sigma_1\sigma_3\cdots\sigma_{k-2}\sigma_k$ projects over the permutation $(1,2)(3,4)\cdots(k,k+1)$. 
%
%Now, since $\beta'\alpha^{-1}$ projects over the trivial (identity) permutation there exists a pure braid $N$ such that $\beta'=N\alpha$ and so $N\alpha$ also has order 2.
so:
\begin{equation}\label{eq:Eqn2Tor}
\beta^{2}= (N \alpha)^{2}=N\ldotp \alpha N\alpha^{-1}\ldotp \alpha^2	
\end{equation}
in $P_n/[P_n, P_n]$ because $\beta^{2}\in P_{n}$. Further, $\alpha^2=A_{1,2}A_{3,4}\cdots A_{k,k+1}$, 
$\alpha A_{1,2} \alpha^{-1}= A_{1,2}$ by \req{conjugAij2}, and if $1\leq r<s\leq n$ then by \req{conjugAij2}, $\alpha A_{r,s} \alpha^{-1} = A_{1,2}$ in $B_n/[P_n, P_n]$ if and only if $(r,s)=(1,2)$. In particular, if we express $N$ (considered as an element of $P_{n}/[P_{n},P_{n}]$) using the basis $\brak{A_{i,j}}_{1\leq i<j\leq n}$, and if $r$ is the coefficient of $A_{1,2}$ in this expression then the coefficient of $A_{1,2}$ in the expression for $\beta^{2}$ in \req{Eqn2Tor} is equal to $2r+1$, which contradicts the fact that $\beta^{2}$ is trivial in $P_{n}/[P_{n},P_{n}]$, and the result follows.
% 
%
%
%, and that for all $1\leq i\leq n-1$, with $i$ odd,
%\begin{equation}\label{eqn:tor2}
%\alpha A_{i,i+1} \alpha^{-1} = A_{i,i+1}
%\end{equation}
%by \req{conjugAij}. 
%
%We will prove that $(N\alpha)^2$ equal to zero in $P_n/[P_n, P_n]$ is a contradiction.
%For all $i$ odd, with $1\leq i\leq n-1$, the conjugacy action of $\alpha$ satisfies the equality
%\begin{equation}\label{eqn:tor2}
%\alpha A_{i,i+1} \alpha^{-1} = A_{i,i+1}
%\end{equation}
%and there is no other different braid $\gamma_i\in \{A_{r,s}\mid 1\leq r < s \leq n \}$ such that $\alpha \gamma_i \alpha^{-1} = A_{i,i+1}$, for all $i$ odd, with $1\leq i\leq n-1$.
%
%
%
%We analyse the elements $A_{1,2}$, $A_{3,4}$, \ldots, $A_{k,k+1}$ in the expression of the pure braid $N\in P_n/[P_n, P_n]$.
%Suppose that each of them appears a number of $r_{1,2}, r_{3,4}, \ldots, r_{k,k+1}$ times in the expression of $N$, respectively. 
%Observe the elementary fact that $r_{1,2}, \ldots, r_{k,k+1}$ are integers, even possibly zero. 
%Then we have from equations \eqref{Eqn2Tor} and \eqref{eqn:tor2} that the set of linear equations, where $r_{i,j}$ are integers, holds
%$$
%\left\{
%\begin{array}{rcl}
%2r_{1,2}+1 & = & 0,\\
%2r_{3,4}+1 & = & 0,\\
% & \vdots & \\
%2r_{k,k+1}+1 & = & 0.
%\end{array}
%\right.
%$$
%But this clearly leads to a contradiction.
%
%So, we conclude that there no exists elements of order  2 in $B_n/[P_n, P_n]$.
\end{proof}

%If $\beta\in B_n/[P_n, P_n]$  satisfies $\sigma(\beta)=(1,2)$ then by \cite[Proposition 3.6]{PS} the element $\beta$ has no order two.  

\begin{rems}\label{rem:restrict}
Let $n\geq 3$.
\begin{enumerate}
\item \reth{2TorCryst} generalises~\cite[Proposition~3.6]{PS}, where it is shown that there is no element $B_n/[P_n, P_n]$ of order two whose image under $\overline{\sigma}$ is the transposition $(1,2)$.
\item \reth{2TorCryst} implies that any finite-order subgroup of $B_n/[P_n, P_n]$ is of odd order.
\item\label{it:restrictc} Applying \repr{cryst}, \reth{2TorCryst} and \relem{eleconj} to the short exact sequence~\reqref{sespnquot}, the restriction $\map{\overline{\sigma}\left\lvert_{K}\right.}{K}[\overline{\sigma}(K)]$ of $\overline{\sigma}$ to any finite subgroup $K$ of $B_n/[P_n, P_n]$ is an isomorphism. In particular, the set of isomorphism classes of the finite subgroups of $B_n/[P_n, P_n]$ is contained in the set of isomorphism classes of the odd-order subgroups of $\sn$.
\end{enumerate}
\end{rems}

%One may use \reco{CorolBCG3} to obtain 
%To close this section, we will show how to obtain 
As we shall now see, by choosing $H$ appropriately, we may use \reco{CorolBCG3} to construct
Bieberbach groups of
%So they are torsion  free  subgroups of  $B_n/[P_n, P_n]$ having 
dimension $n(n-1)/2$ in $B_n/[P_n, P_n]$. In \reth{TeoAKBraids}, we will give a statement for $B_n/[P_n, P_n]$ analogous to that of \reth{TeoAK} in the case that the holonomy group is a finite $2$-group.

%This will lead us to a new proof of the Theorem of Auslander and Kuranishi \ref{TeoAK}, at 
%least for the case where the holonomy is a finite $2$-group.

\begin{lem}\label{lem:LemAK}
Let $n\geq 3$, and let $H$ be a $2$-subgroup of  $\sn$. Then the group $\widetilde{H}_n$ given by \req{BCG3} is a Bieberbach group of dimension $n(n-1)/2$.
%which  is  the dimension  of the  crystallographic group  $B_n/[P_n, P_n]$. 
\end{lem}

\begin{proof}
Let $n\geq 3$, and let $H$ be a $2$-subgroup of $\sn$. Consider the short exact sequences~\reqref{sespnquot} and~\reqref{sesholo}. By \reco{CorolBCG3}, $\widetilde{H}_n$ 
%\begin{equation}\label{BCGn1}
%\mathcal{BCG}(n,H)=\frac{\sigma^{-1}(H)}{[P_n, P_n]}
%\end{equation}
is a crystallographic group of dimension $n(n-1)/2$ with holonomy group $H$. Since the kernel is torsion free, $\overline{\sigma}$ respects the order of the torsion elements of $\widetilde{H}_n$~\cite[Lemma~13]{GG}. In particular, the fact that $H$ is a $2$-group implies that the order of any non-trivial torsion element of $\widetilde{H}_n$ is a positive power of $2$. On the other hand, by \reth{2TorCryst}, the group $B_n/[P_n, P_n]$ has no such torsion elements, so the same is true for $\widetilde{H}_n$. It follows that $\widetilde{H}_n$ is torsion free, hence it is a Bieberbach group of dimension $n(n-1)/2$ because $P_{n}/[P_{n},P_{n}]\cong \Z^{n(n-1)/2}$.
%torsion that the group $\mathcal{BCG}(n,H)$ can have is $2$-torsion. \comj{This has been rewritten a little since before it was stated that a $2$-group can only have $2$-torsion, but in general the torsion consists of powers of $2$. Also, $\mathcal{BCG}(n,H)$ could have other torsion if the kernel were not torsion free.} On the other hand, by \reth{2TorCryst}, the group $B_n/[P_n, P_n]$ has no $2$-torsion, so the same is true for $\mathcal{BCG}(n,H)$. It follows that $\mathcal{BCG}(n,H)$ is a torsion-free crystallographic group, so is a Bieberbach group.
\end{proof}

\begin{rem}
It is not clear to us whether the family of groups that satisfy the conclusions of \reco{CorolBCG3} (resp.\ of \relem{LemAK}) contains all of the isomorphism classes of crystallographic (resp.\ Bieberbach) subgroups of $B_n/[P_n, P_n]$ of dimension $n(n-1)/2$.
\end{rem}

\relem{LemAK} enables us to give an alternative proof of \reth{TeoAK} in the case that the finite group in question is a $2$-group, and to estimate the dimension of the resulting flat manifold.

\begin{thm}\label{th:TeoAKBraids}
Let $H$ be a finite $2$-group. Then $H$ is the holonomy group of some flat manifold $M$. Further, the dimension of $M$ may be chosen to be $n(n-1)/2$, where $n$ is an integer for which $H$ embeds in the symmetric group $\sn$, and the fundamental group of $M$ is isomorphic to a subgroup of $B_n/[P_n, P_n]$.  
\end{thm}

\begin{proof}
Let $H$ be a finite $2$-group. Cayley's Theorem implies that there exists an integer $n\geq 3$ such that $H$ is isomorphic to a subgroup of $\sn$. From \relem{LemAK}, $\widetilde{H}_n$ is a Bieberbach group  of dimension $n(n-1)/2$ with holonomy group $H$ and is a subgroup of $B_n/[P_n, P_n]$. By the first Bieberbach Theorem, there exists a flat manifold $M$ of dimension $n(n-1)/2$ with holonomy group $H$ such that $\pi_1(M)=\widetilde{H}_n$ (see~\cite[Theorem~2.1.1]{Dekimpe} and the paragraph that follows it).
\end{proof}

For a given finite group $H$, it is natural to ask what is the minimal dimension of a flat manifold whose holonomy group is $H$. \reth{TeoAKBraids}
provides an upper bound for this minimal dimension when $H$ is a $2$-group. This upper bound is not sharp in general, for example if $H=\Z_2$.

%%%%%%%%%%%%%%%%%%%%%%%%%%%%%%%%%%%%%%%% SECAO %%%%%%%%%%%%%%%%%%%%%%%%%%%%%%%%%%%%%%%%

\section{The torsion of the group $B_n/[P_n, P_n]$ }\label{sec:SecBCG}

Let $n\geq 3$. In this section we study the torsion elements of the group $B_n/[P_n, P_n]$. The main aim is to show that if $\theta\in \sn$ is of odd order $r$ then there exists $\beta\in B_{n}$ whose $[P_n, P_n]$-coset projects to $\theta$ in $\sn$ and is of order $r$ in $B_n/[P_n, P_n]$  (see \reco{torsion}). We begin by showing that if $r$ is an odd number such that $\sn$ possesses an element of order $r$ then $B_n/[P_n, P_n]$ also has an element of order $r$. By abuse of notation let $\sigma_k=q_n(\sigma_k)$, and let $A_{i,j}=q_n(A_{i,j})$, where 
 $\map{q_n}{B_n}[B_n/[P_n, P_n]]$ is the natural projection. 

\begin{prop}\label{prop:Proalpha0n} Let  $n\geq 3$, let $1\leq i<j\leq r\leq n$, and let $\alpha_{0,r}=\sigma_1\sigma_2\cdots\sigma_{r-1}\in B_n/[P_n, P_n]$.  
The following relations hold in $B_n/[P_n, P_n]$:
%\begin{equation}
\begin{numcases}
{\alpha_{0,r}  A_{i,j} \alpha_{0,r}^{-1}=}
A_{i+1, j+1} & \text{\textnormal{if $j\leq r-1$}}\label{eq:alphaconj1}\\
A_{1, i+1} & \text{\textnormal{if $j=r$.}}\label{eq:alphaconj2}
\end{numcases}
%\end{equation}
\end{prop}

\begin{proof}
We first prove \req{alphaconj1}. If $1\leq k\leq r-2$ then
\begin{align*}
\alpha_{0,r}  \sigma_{k}\alpha_{0,r}^{-1} &= \sigma_1\cdots\sigma_{r-1} \sigma_{k} \sigma_{r-1}^{-1} \cdots \sigma_{1}^{-1}= \sigma_1\cdots\sigma_{k} \sigma_{k+1} \sigma_{k} \sigma_{k+1}^{-1}\sigma_{k}^{-1} \cdots \sigma_{1}^{-1}\\
&= \sigma_1\cdots\sigma_{k-1} \sigma_{k+1} \sigma_{k} \sigma_{k+1}\sigma_{k+1}^{-1}\sigma_{k}^{-1} \sigma_{k-1}^{-1}\cdots \sigma_{1}^{-1}=\sigma_{k+1}.
\end{align*}
So if $1\leq i<j\leq r-1$ then $\alpha_{0,r}  A_{i,j} \alpha_{0,r}^{-1}=A_{i+1, j+1}$ by \req{defaij}, which proves \req{alphaconj1}. So suppose that $j=r$. Let $\gamma=\sigma_{i+1} \cdots \sigma_{r-2} \sigma_{r-1}^{2} \sigma_{r-2}\cdots \sigma_{i+1}$. Since $\gamma\in P_n/[P_n, P_n]$, we have $\alpha_{0,i+1}\gamma\alpha_{0,i+1}^{-1}\in P_n/[P_n, P_n]$, and thus $\alpha_{0,i+1}\gamma\alpha_{0,i+1}^{-1}$ and $\alpha_{0,i+1}\sigma_{i}^{2} \alpha_{0,i+1}^{-1}$ commute pairwise in $P_n/[P_n, P_n]$. Hence:
\begin{align*}
\alpha_{0,r}  A_{i,r} \alpha_{0,r}^{-1}&= \sigma_1\cdots\sigma_{r-1} \sigma_{r-1}\cdots \sigma_{i+1} \sigma_{i}^{2}\sigma_{i+1}^{-1} \cdots \sigma_{r-1}^{-1} \sigma_{r-1}^{-1} \cdots \sigma_{1}^{-1}= \alpha_{0,i+1} \gamma \sigma_{i}^{2}\gamma^{-1}\alpha_{0,i+1}^{-1}\\
&= (\alpha_{0,i+1}\gamma\alpha_{0,i+1}^{-1}) (\alpha_{0,i+1}\sigma_{i}^{2} \alpha_{0,i+1}^{-1})(\alpha_{0,i+1}\gamma^{-1}\alpha_{0,i+1}^{-1})=\alpha_{0,i+1}\sigma_{i}^{2} \alpha_{0,i+1}^{-1}\\
&= \sigma_1\cdots\sigma_{i} \sigma_{i}^{2} \sigma_{i}^{-1}\cdots \sigma_{1}^{-1}\\
&= (\sigma_1\cdots\sigma_{i}\sigma_{i}\cdots \sigma_{1}) (\sigma_1^{-1}\cdots\sigma_{i-1}^{-1} \sigma_{i}^{2} \sigma_{i-1}\cdots \sigma_{1}) (\sigma_1\cdots\sigma_{i}\sigma_{i}\cdots \sigma_{1})^{-1}\\
&= \sigma_1^{-1}\cdots\sigma_{i-1}^{-1} \sigma_{i}^{2} \sigma_{i-1}\cdots \sigma_{1}=A_{1,i+1}
\end{align*}
by \req{defaijalt}, since the three bracketed terms in the penultimate line belong to the quotient $P_n/[P_n, P_n]$ and so commute pairwise. This proves \req{alphaconj2}.
\end{proof}

We now apply \repr{Proalpha0n} to determine the orbits in $P_n/[P_n, P_n]$ for the action of conjugation by the element $\alpha_{0,n}\in B_n/[P_n, P_n]$. If $x\in \R$, $\lfloor x\rfloor$ shall denote the largest integer less than or equal to $x$.

%By abuse of notation we write 
%$$
%\alpha_{0,n}=\Theta_n(\alpha_{0,n})\colon P_n/[P_n, P_n]\to P_n/[P_n, P_n]
%$$ 
%for the next two results. 

\begin{cor}\label{cor:orbits}
Let $n\geq 3$. The set $\setl{A_{i,j}\in P_n/[P_n, P_n]}{1\leq i < j\leq n}$ is invariant under the action of conjugation by the element $\alpha_{0,n}$, and there are $\lfloor\frac{n-1}{2} \rfloor$ orbits each of length $n$ given by:
\begin{multline}\label{eq:orbitN}
A_{1,j+1} \stackrel{\alpha_{0,n}}{\longmapsto} A_{2,j+2} \stackrel{\alpha_{0,n}}{\longmapsto} \cdots \stackrel{\alpha_{0,n}}{\longmapsto} A_{n-j,n} \stackrel{\alpha_{0,n}}{\longmapsto} A_{1,n-j+1} \stackrel{\alpha_{0,n}}{\longmapsto} A_{2,n-j+2} \stackrel{\alpha_{0,n}}{\longmapsto} \cdots\\\stackrel{\alpha_{0,n}}{\longmapsto} A_{j,n} \stackrel{\alpha_{0,n}}{\longmapsto} A_{1,j+1}
\end{multline}
for $j=1,\ldots, \lfloor\frac{n-1}{2} \rfloor$.
%
%\begin{align*}
%& A_{1,2}\stackrel{\alpha_{0,n}}\longmapsto A_{2,3}\stackrel{\alpha_{0,n}}\longmapsto \cdots \stackrel{\alpha_{0,n}}\longmapsto A_{n-2,n-1}\stackrel{\alpha_{0,n}}\longmapsto A_{n-1,n}\stackrel{\alpha_{0,n}}\longmapsto A_{1,n}\stackrel{\alpha_{0,n}}\longmapsto A_{1,2}\\
%& A_{1,3}\stackrel{\alpha_{0,n}}\longmapsto A_{2,4}\stackrel{\alpha_{0,n}}\longmapsto \cdots \stackrel{\alpha_{0,n}}\longmapsto A_{n-2,n}\stackrel{\alpha_{0,n}}\longmapsto A_{1,n-1}\stackrel{\alpha_{0,n}}\longmapsto A_{2,n}\stackrel{\alpha_{0,n}}\longmapsto A_{1,3}\\
%&\vdots\\
%& A_{1,\lfloor \frac{n+1}{2}\rfloor}\stackrel{\alpha_{0,n}}\longmapsto A_{2,\lfloor\frac{n+3}{2}\rfloor}\stackrel{\alpha_{0,n}}\longmapsto \cdots \stackrel{\alpha_{0,n}}\longmapsto A_{\lfloor\frac{n+2}{2}\rfloor,n}\stackrel{\alpha_{0,n}}\longmapsto A_{1,\lfloor\frac{n+4}{2}\rfloor} \stackrel{\alpha_{0,n}}\longmapsto \cdots \stackrel{\alpha_{0,n}}\longmapsto A_{\lfloor\frac{n-1}{2}\rfloor,n}\stackrel{\alpha_{0,n}}\longmapsto A_{1,\lfloor\frac{n+1}{2}\rfloor}.
%\end{align*}
If $n$ is even then there is an additional orbit of length $n/2$ given by:
\begin{equation*}
A_{1,\frac{n+2}{2}}\stackrel{\alpha_{0,n}}\longmapsto A_{2,\frac{n+4}{2}}\stackrel{\alpha_{0,n}}\longmapsto \cdots \stackrel{\alpha_{0,n}}\longmapsto A_{\frac{n}{2},n}\stackrel{\alpha_{0,n}}\longmapsto A_{1,\frac{n+2}{2}}.
\end{equation*}
\end{cor}

\reco{orbits} plays an important rôle in the proof of the following proposition, which states that
if $n$ is odd then the crystallographic group $B_n/[P_n, P_n]$ possesses elements of order $n$.

\begin{prop}\label{prop:Torcaonimpar}
If $n\geq 3$ is odd then $B_n/[P_n, P_n]$ possesses infinitely many elements of order $n$.
\end{prop}

\begin{proof}
Let $n\geq 3$ be odd. For $1\leq i\leq (n-1)/2$ and $1\leq j\leq n$, let
\begin{equation}\label{eq:tableij}
e_{i,j}=\begin{cases}
A_{j,i+j} & \text{if $i+j\leq n$}\\
A_{i+j-n,j} & \text{if $i+j>n$.}
\end{cases}
\end{equation}
By \req{orbitN}, the action by conjugation of $\alpha_{0,n}$ on the $e_{i,j}$ is given by:
\begin{equation}\label{eq:orbeij}
\text{$e_{i,1}\stackrel{\alpha_{0,n}}{\longmapsto} e_{i,2}\stackrel{\alpha_{0,n}}{\longmapsto} \cdots \stackrel{\alpha_{0,n}}{\longmapsto} e_{i,n-1}\stackrel{\alpha_{0,n}}{\longmapsto} e_{i,n}\stackrel{\alpha_{0,n}}{\longmapsto} e_{i,1}$ for $i=1,\ldots, (n-1)/2$.}
\end{equation}
In particular, the set $\brak{e_{i,j}}_{1\leq i\leq (n-1)/2, \; 1\leq j\leq n}$ is a basis for $P_n/[P_n, P_n]$. 
The full-twist braid of $B_{n}$ may be written as $(\sigma_{1}\cdots \sigma_{n-1})^{n}$, or alternatively as the product $\prod_{j=2}^{n} \left( \prod_{i=1}^{j}\, A_{i,j}\right)$. This expression contains each of the $A_{i,j}$ exactly once, and so $\displaystyle\alpha_{0,n}^{n}=\sum_{\substack{1\leq i\leq (n-1)/2 \\ 1\leq j\leq n}}\, e_{i,j}$ in $P_n/[P_n, P_n]$, using additive notation for this group. Let $ N\in P_n/[P_n, P_n]$, and for $1\leq i\leq (n-1)/2$ and $1\leq j\leq n$, let $a_{i,j}\in \Z$ be such that:
\begin{equation}\label{eq:sumN}
N=\sum_{\substack{1\leq i\leq (n-1)/2 \\ 1\leq j\leq n}}\, a_{i,j}e_{i,j}.
\end{equation}
%$$
% N=(a_{1,1},\ldots,a_{1,n};a_{2,1},\ldots,a_{2,n};\cdots ;a_{\frac{n-1}{2},1},\ldots,a_{\frac{n-1}{2},n}).
%$$
It follows from \req{orbeij} that for all $k=0,1,\ldots,n-1$,
\begin{equation*}
\alpha_{0,n}^{k} N \alpha_{0,n}^{-k}=\sum_{\substack{1\leq i\leq (n-1)/2 \\ 1\leq j\leq n}}\, a_{i,j}e_{i,j+k},
\end{equation*}
where the second index of $e_{i,j+k}$ is taken modulo $n$. Hence:
\begin{align*}
( N\cdot\alpha_{0,n})^n & =  N + \alpha_{0,n} N\alpha_{0,n}^{-1} + \alpha_{0,n}^2 N\alpha_{0,n}^{-2} + \cdots + \alpha_{0,n}^{n-1} N\alpha_{0,n}^{1-n} + \alpha_{0,n}^n\\
& = \sum_{i=1}^{n(n-1)/2}\, \left( \sum_{j=1}^{n} a_{i,j}\right) \left( \sum_{j=1}^{n} e_{i,j}\right)+\sum_{\substack{1\leq i\leq (n-1)/2 \\ 1\leq j\leq n}}\, e_{i,j}\\
&= \sum_{i=1}^{n(n-1)/2}\, \left( \left(\sum_{j=1}^{n} a_{i,j}\right)+1\right) \left( \sum_{j=1}^{n} e_{i,j}\right).
\end{align*}
%
% So:
%\begin{align*}
%( N\cdot\alpha_{0,n})^n & =  N + \alpha_{0,n} N\alpha_{0,n}^{-1} + \alpha_{0,n}^2 N\alpha_{0,n}^{-2} + \cdots + \alpha_{0,n}^{n-1} N\alpha_{0,n}^{1-n} + \alpha_{0,n}^n.
%\end{align*}
%It follows 
%
%Therefore, using the previous information we have  that 
%Since  $\alpha_{0,n}=q_n(\Delta_n$) (the class of the  full twist in  $P_n/[P_n, P_n]$) then  $\alpha_{0,n}=(1,1,\ldots,1)$, the element which has 1 in all coordinates.
%
%Notice that 
%
%$$
%\alpha_{0,n} N\alpha_{0,n}^{-1}=(a_{1,n},  a_{1,1},\ldots,a_{1,n-1}; a_{2,n}, a_{2,1},\ldots, a_{2,n-1};\cdots ; a_{\frac{n-1}{2},n}, a_{\frac{n-1}{2},1},\ldots, a_{\frac{n-1}{2},n-1})
%$$
%and so on, permuting cyclically the entries of $N$, for  $\alpha_{0,n}^r N\alpha_{0,n}^{-r}$, with  $2\leq r\leq n-1$. 
%
Thus $(N\cdot\alpha_{0,n})^n$ is equal to the trivial element of $P_n/[P_n, P_n]$ if and only if:
\begin{equation}\label{eq:SistEqns}
\text{$\left(\sum_{j=1}^{n} a_{i,j}\right)+1=0$ for all $i=1,\ldots, (n-1)/2$.}
\end{equation}
%Therefore, we can deduce that the following equality
%$$
%( N\cdot\alpha_{0,n})^n=(0,\ldots,0)\in P_n/[P_n, P_n]
%$$
%is equivalent with the following system of equations:
%
%\begin{equation}
%	\left\{\begin{array}{l}
%	a_{1,1} + a_{1,2} + a_{1,n} + 1 = 0\\
%	a_{2,1} + a_{2,2} + a_{2,n} + 1 = 0\\
%	\ \ \ \ \ \ \ \ \ \ \ \ \ \ \vdots\\
%	a_{\frac{n-1}{2},1} + a_{\frac{n-1}{2},2} + a_{\frac{n-1}{2},n} + 1 = 0.\\
%	\end{array}
%	\right.
%\end{equation}
This system of equations admits infinitely many solutions in $\Z$. For each such solution, $N\alpha_{0,n}$ is of finite order, and its order divides $n$. On the other hand, since $\overline{\sigma}(N\alpha_{0,n})=(1,n,n-1,\ldots,2)$, the order of $N\alpha_{0,n}$ is at least $n$. We thus conclude that for any $N\in P_{n}/[P_{n},P_{n}]$ given by the expression~\reqref{sumN} whose coefficients satisfy the system~\reqref{SistEqns}, the element $N\alpha_{0,n}$ is of order $n$ in $B_{n}/[P_{n},P_{n}]$.
\end{proof}

As we shall now see, \repr{Torcaonimpar} implies part of \reth{oddtorsion}, namely that if $3\leq n\leq m$ and $n$ is odd then $B_m/[P_m, P_m]$ possesses elements of order $n$.

%\begin{thm}\label{th:oddtorsion}
%Let  $m$ and $n$ be integers such that $2\leq n\leq m$. 
%\begin{enumerate}
%\item\label{it:oddtorsiona} Consider the injective homomorphism $\map{\iota}{B_{n}}[B_m]$ defined by $\iota(\sigma_i)=  \sigma_i$ for all $1\leq i\leq n-1$. Then the induced homomorphism $\map{\overline{\iota}}{B_n/[P_n, P_n]}[B_m/[P_m, P_m]]$ of the corresponding quotient groups is injective. 
%\item\label{it:oddtorsionb} If  $n\geq 3$ and $n$ is odd then  $B_m/[P_m, P_m]$ possesses elements of order $n$.
%Further, there exists such an element whose permutation is an $n$-cycle.
%
%\item\label{it:oddtorsionc} Let $n_1,n_2,\ldots,n_t$ be odd integers greater than or equal to $3$ for which $\sum_{i=1}^{t} \, n_i\leq m$. Then $B_m/[P_m, P_m]$ possesses elements of order $\operatorname{lcm}(n_1,\ldots,n_{t})$. Further, there exists such an element whose cycle type is $(n_1, \ldots, n_t)$.
%\end{enumerate}
%\end{thm}

\begin{proof}[Proof of \reth{oddtorsion}]
Let $m$ and $n$ be integers such that $2\leq n\leq m$. 
\begin{enumerate}
\item By \req{defaij}, $\iota$ restricts to an injective homomorphism $\map{\iota\left\lvert_{P_{n}}\right.}{P_{n}}[P_{m}]$ given by $\iota\left\lvert_{P_{n}}\right.(A_{i,j})=A_{i,j}$ for all $1\leq i<j\leq n$. We wish to prove that the induced homomorphism $\map{\overline{\iota}}{B_n/[P_n, P_n]}[B_m/[P_m, P_m]]$ is injective.  Since $\brak{A_{i,j}}_{1\leq i<j\leq n}$ is a subset of the basis $\brak{A_{i,j}}_{1\leq i<j\leq m}$ of $P_m/[P_m, P_m]$, by regarding the $A_{i,j}$ as elements of the quotient $P_n/[P_n, P_n]$, we see that the restriction $\map{\overline{\iota}\left\lvert_{P_n/[P_n, P_n]}\right.}{P_n/[P_n, P_n]}[P_m/[P_m, P_m]]$ is also injective.
%Let us first prove that the homomorphism  
%	$$
%	\overline{\iota}|\colon P_n/[P_n, P_n]\to P_m/[P_m, P_m]
%	$$ 
%	is injective, where  $\iota|\colon P_n\to P_m$ is defined by $A_{i,j}\mapsto A_{i,j}$, for  $1\leq i<j\leq n$. 
%	Recall that  $P_n/[P_n, P_n]\cong \Z^{n(n-1)/2}$ is a free Abelian group 
%of rank  $n(n-1)/2$ generated by the equivalent classes  $\overline{A_{i,j}}$, for  $1\leq i<j \leq n$. 
%	The same happens with the free  Abelian group $P_m/[P_m, P_m]\cong \Z^{m(m-1)/2}$.
%Therefore, using the additive notation  for the groups $P_n/[P_n, P_n]$ and $P_m/[P_m, P_m]$, 
%we have that  
%$$
%\overline{\iota}|\colon P_n/[P_n, P_n]\to P_m/[P_m, P_m]
%$$ 
%is defined by 
%	$$
%	(\underbrace{0,\ldots,0,\overbrace{1}^{k-th\textrm{position}},0,\ldots,0}_{\frac{n(n-1)}{2}})\mapsto (\underbrace{0,\ldots,0,\overbrace{1}^{k-th\textrm{position}},0,\ldots,0}_{\frac{n(n-1)}{2}},0,\ldots,0),
%	$$ 
%	so clearly we can see that  $\overline{\iota}|\colon P_n/[P_n, P_n]\to P_m/[P_m, P_m]$ is injective. 
Using the short exact sequence~\reqref{sespnquot} and the fact that the homomorphism $\iota$ induces an inclusion of $\sn$ in $\sn[m]$, we obtain the following commutative diagram of short exact sequences:
\begin{equation*}%\label{eq:diagr}
\begin{xy}*!C\xybox{%
\xymatrix{%
1 \ar[r] & \displaystyle\frac{P_n}{[P_n, P_n]} \ar[r] \ar[d]^{\overline{\iota}\left\lvert_{P_n/[P_n, P_n]}\right.} & \displaystyle\frac{B_n}{[P_n, P_n]} \ar[r] \ar[d]^{\overline{\iota}} &  \sn \ar[r] \ar@{^{(}->}[d] & 1\\
1 \ar[r] & \displaystyle\frac{P_m}{[P_m, P_m]} \ar[r] & \displaystyle\frac{B_m}{[P_m, P_m]} \ar[r] &  \sn[m] \ar[r] & 1.}}
\end{xy}
\end{equation*}
The injectivity of $\overline{\iota}$ is then a consequence of the $5$-Lemma.

\item Suppose that $3\leq n\leq m$ and that $n$ is odd. \repr{Torcaonimpar} implies that $B_n/[P_n, P_n]$ possesses elements of order $n$. By part~(\ref{it:oddtorsiona}), $\overline{\iota}$ is injective, and so $B_m/[P_m, P_m]$ also has elements of order $n$, which proves the first part of the statement. The second part follows from the construction given in the proof of \repr{Torcaonimpar}.

\item Let $n_{1},\ldots,n_{t}$ be odd integers greater than or equal to $3$ such that $\sum_{i=1}^{t}\, n_{i}\leq m$, let $\map{\sigma}{B_{\sum_{i=1}^{t}\, n_{i}}}[{\sn[\sum_{i=1}^{t}\, n_{i}]}]$ denote the usual homomorphism that to a braid associates its permutation, and let $B_{n_{1},\ldots,n_{t}}$ denote the corresponding mixed braid group, namely the preimage under $\sigma$ of the subgroup $\sn[n_{1}] \times \cdots \times \sn[n_{t}]$ of $\sn[\sum_{i=1}^{t}\, n_{i}]$. We first prove that the group $B_m/[P_m, P_m]$ possesses elements of order $\operatorname{lcm}(n_1,\ldots,n_{t})$. For $1\leq i\leq t$, let $\map{\phi_{i}}{B_{n_{i}}}[B_{n_{1},\ldots,n_{t}}]$ denote the embedding of $B_{n_{i}}$ into the $i\up{th}$ factor of $B_{n_{1},\ldots,n_{t}}$. Since $\phi_{i}([P_{n_{i}}, P_{n_{i}}])\subset \left[P_{\sum_{i=1}^{t}n_i}, P_{\sum_{i=1}^{t}n_i}\right]$, the homomorphism $\phi_{i}$ induces a homomorphism $\map{\overline{\phi_{i}}}{\displaystyle\frac{B_{n_i}}{[P_{n_i},P_{n_i}]}}[\frac{B_{n_1,n_2,\ldots,n_t}}{\left[P_{\sum_{i=1}^{t}n_i},P_{\sum_{i=1}^{t}n_i}\right]}]$. Now let $\map{\psi_{i}}{\displaystyle B_{n_{1},\ldots,n_{t}}}[\frac{B_{n_i}}{[P_{n_i},P_{n_i}]}]$ be the composition of the projection onto the $i\up{th}$ factor of $B_{n_{1},\ldots,n_{t}}$, followed by the canonical projection $\displaystyle B_{n_i}\to\frac{B_{n_i}}{[P_{n_i},P_{n_i}]}$. Under this composition, the normal subgroup $P_{\sum_{i=1}^{t}n_i}$ of $B_{n_{1},\ldots,n_{t}}$ is sent to $\displaystyle\frac{P_{n_i}}{[P_{n_i},P_{n_i}]}$,  hence the normal subgroup $[P_{\sum_{i=1}^{t}n_i},P_{\sum_{i=1}^{t}n_i}]$ of $B_{n_{1},\ldots,n_{t}}$ is sent to the trivial element of $\displaystyle\frac{B_{n_i}}{[P_{n_i},P_{n_i}]}$, from which it follows that $\psi_{i}$ induces a homomorphism $\map{\overline{\psi_{i}}}{\displaystyle \frac{B_{n_{1},\ldots,n_{t}}}{\left[P_{\sum_{i=1}^{t}n_i},P_{\sum_{i=1}^{t}n_i}\right]}}[\frac{B_{n_i}}{[P_{n_i},P_{n_i}]}]$. From the constructions of $\phi_{i}$ and $\psi_{i}$, we see that $\overline{\psi_{i}}\circ \overline{\phi_{i}}=\id_{B_{n_i}/[P_{n_i},P_{n_i}]}$ for all $1\leq i\leq t$, and so the composition
%Since the composition of the natural inclusion \comj{not really an inclusion, otherwise it would automatically be injective} and natural projection 
\begin{equation*}
\frac{B_{n_1}}{[P_{n_1},P_{n_1}]}  \times \cdots \times \frac{B_{n_t}}{[P_{n_t},P_{n_t}]} \!\xrightarrow{\overline{\phi_{1}} \!\times\cdots\times \overline{\phi_{t}}}
\frac{B_{n_1,n_2,\ldots,n_t}}{\Bigl[P_{\sum_{i=1}^{t}n_i},P_{\sum_{i=1}^{t}n_i}\Bigr]} \!\xrightarrow{\overline{\psi_{1}} \! \times\cdots\times \overline{\psi_{t}}} \frac{B_{n_1}}{[P_{n_1},P_{n_1}]}  \times \cdots \times \frac{B_{n_t}}{[P_{n_t},P_{n_t}]}
\end{equation*}
is the identity. Thus $\overline{\phi_{1}}\times\cdots\times\overline{\phi_{t}}$ is injective, and the composition
%\comj{define these homomorphisms and show that they are well defined} is an isomorphism then $\displaystyle{\frac{B_{n_1}}{[P_{n_1},P_{n_1}]} \times \frac{B_{n_2}}{[P_{n_2},P_{n_2}]} \times \cdots \times \frac{B_{n_t}}{[P_{n_t},P_{n_t}]} \to \frac{B_{n_1,n_2,\ldots,n_t}}{[P_{\sum_{i=1}^{t}n_i},P_{\sum_{i=1}^{t}n_i}]}}$ is an injective homomorphism. So, the following compose is an injective homomorphism
\begin{equation*}
\frac{B_{n_1}}{[P_{n_1},P_{n_1}]} \times \cdots \times \frac{B_{n_t}}{[P_{n_t},P_{n_t}]} \!\xrightarrow{\overline{\phi_{1}}\! \times\cdots\times \overline{\phi_{t}}} \frac{B_{n_1,n_2,\ldots,n_t}}{[P_{\sum_{i=1}^{t}n_i},P_{\sum_{i=1}^{t}n_i}]} \!\to \!\frac{B_{\sum_{i=1}^{t}n_i}}{[P_{\sum_{i=1}^{t}n_i},P_{\sum_{i=1}^{t}n_i}]} \!\to\! \frac{B_{m}}{[P_{m},P_{m}]},
\end{equation*}
which we denote by $\Phi$, is injective by part~(\ref{it:oddtorsiona}) and by the injectivity of the homomorphism $B_{n_{1},\ldots,n_{t}}\lhra B_{\sum_{i=1}^{t}\, n_{i}}$. For $1\leq i\leq t$, let $\displaystyle\gamma_i\in \frac{B_{n_i}}{[P_{n_i},P_{n_i}]}$ be an element of order $n_i$ whose permutation is an $n_{i}$-cycle; the existence of $\gamma_{i}$ is guaranteed by part~(\ref{it:oddtorsionb}). Then $\gamma=(\gamma_1, \ldots, \gamma_t)\in \displaystyle{\frac{B_{n_1}}{[P_{n_1},P_{n_1}]} \times \frac{B_{n_2}}{[P_{n_2},P_{n_2}]} \times \cdots \times \frac{B_{n_t}}{[P_{n_t},P_{n_t}]}}$ is of order $\operatorname{lcm}(n_{1},\ldots, n_t)$, and the injectivity of $\Phi$ implies that $\Phi(\gamma)\in \displaystyle\frac{B_{m}}{[P_{m},P_{m}]}$ is also of order $\operatorname{lcm}(n_{1},\ldots, n_t)$. The second part of the statement follows also.\qedhere
%
%
%has order $\operatorname{lcm}(n_k \mid 1\leq k\leq t)$, provided $\Phi\colon \frac{B_{n_1}}{[P_{n_1},P_{n_1}]} \times \frac{B_{n_2}}{[P_{n_2},P_{n_2}]} \times \cdots \times \frac{B_{n_t}}{[P_{n_t},P_{n_t}]}    \to      \frac{B_{n}}{[P_{n},P_{n}]}$ is an injective homomorphism.
\end{enumerate}
\end{proof}

As a consequence, we are able to prove \reco{torsion}, which says that the torsion of $B_n/[P_n, P_n]$ is equal to the odd torsion of the symmetric group $\sn$, and that the map induced by $\overline{\sigma}$ from the set of finite cyclic subgroups of $B_n/[P_n, P_n]$ to the set of cyclic subgroups of $\sn$ of odd order is surjective.

%\begin{cor}\label{cor:torsion}
%Let $n\geq 3$. The torsion of the quotient $B_n/[P_n, P_n]$ is equal to the odd torsion of the symmetric group $\sn$. Moreover, given an element $\theta\in \sn$ of odd order $r$, there exists $\beta\in B_n/[P_n, P_n]$ of order $r$ such that $\overline{\sigma}(\beta)=\theta$. So given any cyclic subgroup $H$ of $\sn$ of odd order $r$, there exists a finite-order subgroup $\widetilde{H}$ of $B_n/[P_n, P_n]$ such that $\overline{\sigma}(\widetilde{H})=H$.
%\end{cor}

\begin{proof}
Let $\beta\in B_n/[P_n, P_n]$ be a non-trivial element of finite order $r$. By \reth{2TorCryst}, $r$ is odd. \relem{eleconj} implies that $\overline{\sigma}(\beta)$ is also of order $r$, so the torsion of $B_n/[P_n, P_n]$ is contained in the odd torsion of the symmetric group $\sn$. Conversely, suppose that $\theta$ is an element of $\sn$ of odd order $r\geq 3$, and let $\theta=\theta_1 \theta_2 \cdots \theta_t$ be a product of disjoint non-trivial cycles, where $\theta_{i}$ is an $n_{i}$-cycle for all $i=1,\ldots,t$. Then $r=\operatorname{lcm}(n_1,\ldots,n_{t})$, the $n_{i}$ are odd and greater than or equal to $3$, and $\sum_{i=1}^{t}\, n_{i}\leq n$ since the $\theta_{i}$ are disjoint. By \reth{oddtorsion}(\ref{it:oddtorsionc}), $B_n/[P_n, P_n]$ possesses an element $\gamma$ of order $r$ whose permutation has cycle type $(n_1, \ldots,n_t)$. So $\overline{\sigma}(\gamma)$ is conjugate to $\theta$, and thus a suitable conjugate of $\gamma$ is an element of order $r$ whose permutation is equal to $\theta$. The last part of the statement follows in a straightforward manner. 
%Let $n\geq 3$, and let $\theta$ be a element of $\sn$ of odd order $r$, which we write as a product $\theta=\theta_1 \theta_2 \cdots \theta_t$ of disjoint non-trivial cycles, where $\theta_{i}$ is an $n_{i}$-cycle for all $i=1,\ldots,t$. Then $r=\operatorname{lcm}(n_1,\ldots,n_{t})$, the $n_{i}$ are odd and greater than or equal to $3$, and $\sum_{i=1}^{t}\, n_{i}\leq n$ since the $\theta_{i}$ are disjoint.
%%There are $\theta_1, \ldots, \theta_s$ disjoint cycles of order $r_1,\ldots, r_s$,
%%respectively, and such that $r=\operatorname{lcm}(r_1,\ldots,r_{s})$ and $\theta$  have an expression 
%%given by the product $\theta=\theta_1 \theta_2 \cdots \theta_s$.
%Let $\gamma\in \sn$ be such that $\gamma \theta_{i} \gamma^{-1}$ is the $n_{i}$-cycle $(\sum_{j=1}^{i} n_{j}, \sum_{j=1}^{i} n_{j}-1,\ldots, \sum_{j=1}^{i-1} n_{j}+1)$, and let $\widetilde{\gamma}\in B_{n}/[P_{n},P_{n}]$ be such that $\overline{\sigma}(\widetilde{\gamma})=\gamma$. \comj{To be completed\ldots\ Also add a line to explain the second part of the statement.}
\end{proof}

%\begin{rem}
%As we shall see in \reth{??}, \comj{check this, or mention \rerem{abelian}} the last part of the statement of \reco{torsion} may be generalised to any finite Abelian group of $\sn$ of odd order.
%\end{rem}

\begin{rem}
In order to study the conjugacy classes of finite-order elements of the group $B_n/[P_n, P_n]$, we will describe some of these elements in more detail in \resec{SecBCGconj}.
\end{rem}

\section{A study of some crystallographic subgroups of dimension $3$ of $B_3/[P_3, P_3]$ }\label{sec:subgroupsb3}

As we saw in \resec{SecCG}, the group $B_3/[P_3, P_3]$ is crystallographic and has no $2$-torsion. In this section, we further analyse this quotient and we study some of the crystallographic subgroups of $B_3/[P_3, P_3]$ of dimension $3$, of the form $\overline{\sigma}^{-1}(H)$, where $H$ is a subgroup of $S_3$. In order to study these subgroups, it suffices to consider a representative of each conjugacy class of subgroups of $\sn[3]$. We shall also comment on some other subgroups of $B_3/[P_3, P_3]$.  

\begin{prop}\label{prop:apres}
Let $H$ be a subgroup of $\sn[3]$, and let $\widetilde{H}_3$ be given by \req{BCG3}.
\begin{enumerate}
\item\label{it:apresa} Let $H=\brak{1}$. The crystallographic group $\widetilde{H}_3$ admits a presentation whose generators are $ A_{1,2}, A_{1,3}, A_{2,3}$, with defining 
relations $[A_{1,2}, A_{1,3}]=1$, $[A_{1,2}, A_{2,3}]=1$ and $[A_{1,2}, A_{2,3}]=1$.
\item\label{it:apresb} Let $H=\ang{\left(1,3,2\right)}$. The crystallographic group $\widetilde{H}_3$ is normal in $B_3/[P_3,P_3]$ and admits a presentation given by:
\begin{itemize}
\item generators: $A_{1,2}, A_{2,3}, A_{1,3}, \alpha_{0,3}$, where $\alpha_{0,3}=\sigma_1\sigma_2\in B_3/[P_3,P_3]$.
\item relations: 
\begin{enumerate}
\item $[A_{1,2}, A_{1,3}]=1$, $[A_{1,2}, A_{2,3}]=1$, $[A_{1,3}, A_{1,3}]=1$.
\item $\alpha_{0,3}^3=\Delta_3^{2}=A_{1,2}A_{1,3}A_{2,3}$ ($\Delta_3^{2}$ is the class of the full-twist braid in $P_3/[P_3,P_3]$).
\item $\alpha_{0,3}A_{1,2}\alpha_{0,3}^{-1}=A_{2,3}$, $\alpha_{0,3}A_{1,3}\alpha_{0,3}^{-1}=A_{1,2}$, $\alpha_{0,3}A_{2,3}\alpha_{0,3}^{-1}=A_{1,3}$.
\end{enumerate}
\end{itemize}
The Abelianisation $\left(\widetilde{H}_3\right)_{\text{Ab}}$ of $\widetilde{H}_3$ is given by:
\begin{equation*}
\left(\widetilde{H}_3\right)_{\text{Ab}}=\setangr{A_{1,2}, \alpha_{0,3}}{[A_{1,2}, \alpha_{0,3}]=1, \alpha_{0,3}^3=A_{1,2}^3},
\end{equation*}
and is isomorphic to $\Z\oplus \Z_3$, where the factors are generated by $A_{1,2}$ and $A_{1,2}\alpha_{0,3}^{-1}$.	

\item\label{it:apresc} Let $H=\ang{\left(1, 2\right)}$. The crystallographic group $\widetilde{H}_3$ admits a presentation given by:
\begin{itemize}
\item generators: $A_{1,2}, A_{2,3}, A_{1,3}, \sigma_1$.
\item relations: 
\begin{enumerate}
\item $[A_{1,2}, A_{1,3}]=1$, $[A_{1,2}, A_{2,3}]=1$, $[A_{1,3}, A_{1,3}]=1$.
\item $\sigma_1^2=A_{1,2}$.
\item $\sigma_1 A_{1,2}\sigma_1^{-1}=A_{1,2}$, $\sigma_1 A_{1,3}\sigma_1^{-1}=A_{2,3}$, $\sigma_1 A_{2,3}\sigma_1^{-1}=A_{1,3}$.
\end{enumerate}
\end{itemize}
We have:
\begin{equation*}
\left(\widetilde{H}_3\right)_{\text{Ab}}= \setangr{A_{1,2}, A_{1,3}, \sigma_1}{[A_{1,2}, \sigma_1]=1, [A_{1,3}, \sigma_1]=1,  \sigma_1^2=A_{1,2}}\cong \Z\oplus \Z.
\end{equation*}

\item\label{it:apresd} Let $H_3=\sn[3]$. The crystallographic group $\widetilde{H}_3=B_3/[P_3, P_3]$ admits a presentation whose generators are $\sigma_1, \sigma_2$, with defining relations $\sigma_1\sigma_2\sigma_1=\sigma_2\sigma_1\sigma_2$ and $(\sigma_1^{-1}\sigma_2)^3=1$. We have
$\left(\widetilde{H}_3\right)_{\text{Ab}}=\ang{\sigma_1}\cong \Z$.
\end{enumerate}
\end{prop}

\begin{proof}
Part~(\ref{it:apresa}) follows from the fact that $\widetilde{H}_3=P_{3}/[P_3,P_3]$ if $H$ is trivial. The presentations given in parts~(\ref{it:apresb}) and~(\ref{it:apresc}) may be obtained by applying the method of presentations of group extensions given in~\cite[Section 10.2]{Johnson}. In part~(\ref{it:apresb}), the normality of $\widetilde{H}_3$ follows from that of $\Z_{3}$ in $\sn[3]$.

By~\cite[Lemma 4.3.9]{OcampoPhD}, the commutator subgroup $[P_3,P_3]$ is equal to the normal closure of $[A_{1,2},A_{2,3}]$ in $B_{3}$. Since $[A_{1,2},A_{2,3}]=(\sigma_1^{-1}\sigma_2)^3$ and $B_3=\setangr{\sigma_1, \sigma_2}{\sigma_1\sigma_2\sigma_1=\sigma_2\sigma_1\sigma_2}$, we thus obtain the presentation given in part~(\ref{it:apresd}). In each case, $\left(\widetilde{H}_3\right)_{\text{Ab}}$ is obtained in a straightforward manner from the presentation of $\widetilde{H}_3$.
\end{proof}

\begin{rem}
 The presentation of $B_{3}/[P_{3},P_{3}]$ of \repr{apres}(\ref{it:apresd}) also appeared in~\cite[Proposition~4.3.10]{OcampoPhD} and in~\cite[Proposition~3.9]{LW1}.
\end{rem}

\begin{thm}\label{th:XHW}
Let $H$ be a subgroup of $\sn[3]$, and let $\widetilde{H}_3$ be given by \req{BCG3}.
\begin{enumerate}
\item\label{it:xhwa} Let $H=\brak{1}$. Then $\widetilde{H}_3$ is isomorphic to the quotient $P_3/[P_3, P_3]$, which is isomorphic to $\Z^3$. The corresponding flat manifold is the $3$-torus.
	
\item\label{it:xhwb} Let $H=\ang{\left(1,2\right)}$. Then $\widetilde{H}_3$ is a Bieberbach group of dimension $3$ with holonomy group $\Z_2$. The corresponding flat Riemannian manifold is diffeomorphic to the non-orientable manifold $\mathscr{B}_2$ that appears in the classification of flat Riemannian $3$-manifolds given in~\cite[Corollary~3.5.10]{Wolf}.

\item\label{it:xhwc} Let $H=\ang{\left(1,3,2\right)}$. Then $\widetilde{H}_3$ is isomorphic to the semi-direct product $\Z^{3}\rtimes \Z_3$, where the action is given by the matrix 
 $\left(\begin{smallmatrix}
0 & 0 & 1 \\
1 & 0 & 0 \\
0 & 1 & 0
\end{smallmatrix}\right)$ with respect to the basis $(A_{1,2}, A_{2,3}, A_{1,3})$ of $P_{3}/[P_{3},P_{3}]$, this quotient being identified with $\Z^{3}$.
	
%\item\label{it:xhwd} Let $H=\sn[3]$. \comj{Daciberg: check this!} Then $\widetilde{H}_3$ may be identified in terms of the classification of~\cite{CDHT}  {\bf !!!D!!!} and of that of three-dimensional 
%crystallographic groups of~\cite{BBNWS} where it appears on page~71 as \texttt{5/4/1:SPGR:02}, which corresponds to \texttt{IT~161; OBT~1} in the international table.
\end{enumerate}
\end{thm}

\begin{proof}
Part~(\ref{it:xhwa}) is clear,
% For part~(\ref{it:xhwd}), note that the subgroup of $B_3/[P_3, P_3]$ generated by the class of the full-twist braid $A_{1,2}A_{1,3}A_{2,3}$, given by $(1,1,1)$ in terms of the given basis of $P_3/[P_3, P_3]$, is a normal subgroup of $B_3/[P_3, P_3]$. 
%The associated quotient group admits a presentation that may be obtained from a presentation of $B_3/[P_3, P_3]$ as follows. If in this quotient we denote the class of $A_{1,2}$ (resp.\ $A_{2,3}^{-1}$, $\sigma_{1}$, $\sigma_{1} \sigma_{2}$) by $\alpha$ (resp.\ $\beta$, $\rho$, $\sigma$) then a presentation of this group is given by:
%\begin{equation*}
%\setangr{\alpha, \beta, \rho, \sigma}{\rho \alpha \rho^{-1}=\alpha,  \rho \beta \rho^{-1}=\alpha\beta^{-1}, \sigma\alpha\sigma^{-1}=\alpha^{-1}\beta, \sigma\beta\sigma^{-1}=\alpha^{-1}, \rho^2=\alpha, \sigma \rho=\rho^{-1}\sigma}.
%\end{equation*}
%It is not difficult to see that this is another presentation of the group $G_3^1$ given in the first theorem of~\cite[page~73]{Ly}. \comj{to be completed:} The identification of the group with the group of the table claimed  on  the Theorem ?? follows %by ....??????
so let us prove part~(\ref{it:xhwb}). \reth{2TorCryst} implies that $B_3/[P_3, P_3]$ has no $2$-torsion, and so the subgroup $\widetilde{H}_3$ is a Bieberbach group of dimension $3$ by~\relem{LemAK}. Let $X$ be the flat Riemannian manifold uniquely determined by $\widetilde{H}_3$, so that $\pi_1(X)=\widetilde{H}_3$. The holonomy representation of $\widetilde{H}_3$ is a homomorphism of the form $\Z_2 \to \operatorname{Aut}(\Z^{3})$, where we identify $P_{3}/[P_{3},P_{3}]$ with $\Z^{3}$. Relative to the basis $(A_{1,2}, A_{1,3}, A_{2,3})$ of $P_{3}/[P_{3},P_{3}]$, by \req{conjugAij2}, the image of the generator of $\Z_{2}$ by this homomorphism is given by the matrix
$\left(\begin{smallmatrix}
1 & 0 & 0 \\
0 & 0 & 1 \\
0 & 1 & 0
\end{smallmatrix}\right)$ whose determinant is equal to $-1$ . Thus $X$ is a non-orientable flat Riemannian $3$-manifold with holonomy group $\Z_2$. Up to affine diffeomorphism, $X$ is one of the two manifolds $\mathscr{B}_1$ or $\mathscr{B}_2$ described in~\cite[Theorem~3.5.9]{Wolf}. Using the presentation of $\widetilde{H}_3$ given in \repr{apres}(\ref{it:apresb}) we have $H_1(X; \Z)\cong \Z^{2}$, %\comj{This comes from?} 
and from the table in~\cite[Corollary~3.5.10]{Wolf}, we conclude that $X=\mathscr{B}_2$. 

Finally we prove part~(\ref{it:xhwc}). The following short exact sequence:
\begin{equation*}
1\to P_{3}/[P_{3},P_{3}] \to \widetilde{H}_3 \to H \to 1
\end{equation*}
admits a section given by sending the generator $\left(1,3,2\right)$ of $H$ to the element $\sigma_1^{-1}\sigma_2$ of $\widetilde{H}_3$, and so $\widetilde{H}_3$ is isomorphic to the semi-direct product of the form $\Z^{3}\rtimes \Z_3$. Relative to the basis $(A_{1,2}, A_{2,3}, A_{1,3})$ of $P_{3}/[P_{3},P_{3}]$, the matrix of the associated action is equal to 
 $\left(\begin{smallmatrix}
0 & 0 & 1 \\
1 & 0 & 0 \\
0 & 1 & 0
\end{smallmatrix}\right)$.
\end{proof}

\begin{rems}\mbox{}
\begin{enumerate}
\item The subgroup of $B_3/[P_3, P_3]$ generated by the class of the full-twist braid $A_{1,2}A_{1,3}A_{2,3}$, given by $(1,1,1)$ in terms of the basis $(A_{1,2}, A_{1,3}, A_{2,3})$ of $P_3/[P_3, P_3]$, is a normal subgroup of $B_3/[P_3, P_3]$. The associated quotient group admits the following presentation that is obtained from a presentation of $B_3/[P_3, P_3]$:
\begin{equation*}
\setangr{\sigma_1, \sigma_2}{\sigma_2\sigma_1 \sigma_2= \sigma_1\sigma_2\sigma_1,\,   
(\sigma_1^{-1}\sigma_2)^3=1,\,  A_{1,2}A_{1,3}A_{2,3}=1}.
\end{equation*}
The group $G_3^1$ given in the first theorem of~\cite[page~73]{Ly} is generated by the set $\brak{\alpha, \beta, \sigma, \rho}$, with relations: 
\begin{equation*}
[\alpha,\beta]=[\rho,\alpha]=\sigma^3=\rho^2=(\sigma\rho)^2=1, \ \sigma\alpha\sigma^{-1}=\alpha^{-1}\beta, \ \sigma\beta\sigma^{-1}=\alpha^{-1},\, \rho\beta\rho^{-1}=\alpha\beta^{-1}.
\end{equation*}
A routine calculation shows that the map that sends $A_{1,2}$ (resp.\ $A_{1,3}^{-1}$, $A_{1,3}\sigma_1$, $\sigma_{1}^{-1} \sigma_{2}$) to $\alpha$ (resp.\ $\beta$, $\rho$, $\sigma$) extends to an isomorphism of the two groups. 
%\begin{equation*}
%\setangr{\alpha, \beta, \rho, \sigma}{\rho \alpha \rho^{-1}=\alpha,  \rho \beta \rho^{-1}=\alpha\beta^{-1}, \sigma\alpha\sigma^{-1}=\alpha^{-1}\beta, \sigma\beta\sigma^{-1}=\alpha^{-1}, \rho^2=\alpha, \sigma \rho=\rho^{-1}\sigma}.
%\end{equation*}
\item   The group $B_3/[P_3, P_3]$ is the three-dimensional 
crystallographic group that appears as \texttt{5/4/1:SPGR:02} of~\cite[page~71]{BBNWS}, and that corresponds to \texttt{IT~161; OBT~1} in the international table~\cite{HL}.

\item Let $L$ be a crystallographic subgroup of $B_3/[P_3, P_3]$ of dimension $3$, and consider the subgroup $\overline{\sigma}(L)$ of $\sn[3]$. If $\overline{\sigma}(L)=\brak{\id}$ then clearly $L$ is isomorphic to $\Z^3$. If $\overline{\sigma}(L)=\ang{(1,2)}$ then $L$ is a Bieberbach group. If $\overline{\sigma}(L)=\ang{(1, 3, 2)}$ then the group $L$ may be Bieberbach or not, with holonomy $\Z_3$. For example, if $L$ is the subgroup generated by $\brak{\sigma_1^{-1}\sigma_2, A_{1,2}^{2}, A_{1,3}^{2}, A_{2,3}^{2}}$ then $L$ is a proper crystallographic subgroup of $\overline{\sigma}^{-1}(\ang{(1, 3, 2)})$ of dimension $3$ with holonomy $\Z_3$ and that admits torsion elements, $\sigma_1^{-1}\sigma_2$ for example. On the other hand, if $L$ is the subgroup generated by $\brak{A_{1,2}\sigma_1^{-1}\sigma_2, A_{1,2}^{3}, A_{1,3}^{3}, A_{2,3}^{3}}$ then $L$ is a proper subgroup of $\overline{\sigma}^{-1}(\ang{(1, 3, 2)})$, and is a Bieberbach group of dimension $3$ with holonomy $\Z_3$. To see this, let $L_{1}=L\cap \ker{\overline{\sigma}}= L\cap \ang{A_{1,2},A_{1,3}, A_{2,3}}$. Clearly $L_{1}$ is a free Abelian group, so is torsion free. Using \req{conjugAij2}, we see that $(A_{1,2}\sigma_1^{-1}\sigma_2)^{3}=A_{1,2} A_{1,3} A_{2,3}$, and since $\brak{(A_{1,2},A_{1,3}, A_{2,3})^{j}}_{j\in \brak{0,1,2}}$ is a set of coset representatives of $L_{1}$ in $L$, it follows that $L_{1}$ is generated by $\brak{A_{1,2} A_{1,3} A_{2,3}, A_{1,2}^{3}, A_{1,3}^{3}, A_{2,3}^{3}}$. Note then that $\brak{A_{1,2} A_{1,3} A_{2,3}, A_{1,3}^{3}, A_{2,3}^{3}}$ is a basis of $L_{1}$. Suppose that $w$ is a non-trivial torsion element of $L$. By \relem{eleconj}, $w$ must be of order $3$. Now $w\notin L_{1}$, so there exist $\theta\in L_{1}$ and $j\in \brak{1,2}$ such that $w=\theta (A_{1,2}\sigma_1^{-1}\sigma_2)^{j}$. Since $\theta \in L_{1}$, there exist $\lambda_{1}, \lambda_{2},\lambda_{3}\in \Z$ such that $\theta=(A_{1,2} A_{1,3} A_{2,3})^{\lambda_{1}} A_{1,3}^{3\lambda_{2}} A_{2,3}^{3\lambda_{3}}$, and hence:
\begin{equation*}
1=w^{3}= \theta (A_{1,2}\sigma_1^{-1}\sigma_2)^{j} \theta (A_{1,2}\sigma_1^{-1}\sigma_2)^{-j} \ldotp (A_{1,2}\sigma_1^{-1}\sigma_2)^{2j}  \theta (A_{1,2}\sigma_1^{-1}\sigma_2)^{-2j} \ldotp (A_{1,2}\sigma_1^{-1}\sigma_2)^{3j}. 
\end{equation*}
Using once more \req{conjugAij2}, the relation $(A_{1,2}\sigma_1^{-1}\sigma_2)^{3}=A_{1,2} A_{1,3} A_{2,3}$, and comparing the coefficients of $A_{1,2},A_{1,3}$ and $A_{2,3}$, we obtain the equality $3(\lambda_{1}+\lambda_{2}+\lambda_{3})+j=0$, which has no solution in $\Z$. It follows that $L$ is torsion free, and so is a Bieberbach group of dimension $3$ with holonomy $\Z_3$.

\item There is no Bieberbach subgroup of $B_3/[P_3, P_3]$ of dimension $3$ that projects to $\sn[3]$, since none of the ten flat Riemannian $3$-manifolds have fundamental group with holo\-nomy $\sn[3]$ (see~\cite[Theorems~3.5.5 and~3.5.9]{Wolf}).
\end{enumerate}
\end{rems}

%%%%%%%%%%%%%%%%%%%%%%%%%%%%%%%%%%%%%%%% SECAO %%%%%%%%%%%%%%%%%%%%%%%%%%%%%%%%%%%%%%%%

\section{Conjugacy classes of finite-order elements of $B_n/[P_n, P_n]$}\label{sec:SecBCGconj}

In this section, we study the conjugacy classes of finite-order elements of $B_n/[P_n, P_n]$. The aim is to prove \reth{classconj}, which states that there is a bijection between the conjugacy classes of cyclic subgroups of odd order of $B_n/[P_n, P_n]$ and the set of conjugacy classes of cyclic subgroups of odd order of the symmetric group $\sn$.

We begin with an elementary fact about conjugacy classes that will help to simplify the study of our problem.   

%The conjugacy classes of elements of $\sn$ of odd order are closely related to the conjugacy classes of the elements of finite odd order in $B_n/[P_n, P_n]$.

\begin{lem}\label{lem:conj}
Let $\alpha, \beta\in B_n/[P_n, P_n]$ be two conjugate elements of finite order. Then $\overline{\sigma}(\alpha)$ and $\overline{\sigma}(\beta)$ are permutations of odd order and have the same cycle type.
\end{lem}

\begin{proof}
Since $\alpha, \beta$ are of finite order, their common order is odd by \reth{2TorCryst}. The fact that $\alpha$ and $\beta$ are conjugate in $B_n/[P_n, P_n]$ implies that the permutations $\overline{\sigma}(\alpha)$ and $\overline{\sigma}(\beta)$ are conjugate in $\sn$. The result then follows since two permutations are conjugate in $\sn$ if and only they have the same cycle type.
\end{proof}

In order to analyse the conjugacy classes of elements of finite order, \relem{conj} implies that it suffices to choose a single representative permutation for each conjugacy class of $\sn$ of odd order and to study the conjugacy classes of elements of $B_n/[P_n, P_n]$ of finite order that project to the chosen permutation.

% The above result shows  that to study the conjugacy classes of elements of finite order, it suffices 
%to consider one permutation for each conjugacy class of permutations of odd order and study the  conjugacy classes of elements of $G$ of finite order which projects to the chosen permutation. We will use that on the calculation below. 

%%%%%%%%%%%%%%%%%%%%%%% Classes de conjugacao de elementos de ordem finita em B_n/[P_n, P_n] 

%Let  $\sn$ be the symmetric groups regarded as the set of permutations of the numbers $1,2, \ldots, n$. 
Let us consider the action by conjugation of certain elements of $B_n/[P_n, P_n]$ on the group $P_n/[P_n, P_n]$. If $k,n\geq 3$ and $r\geq 0$ are integers such that $r+k\leq n$, define $\delta_{r,k},\alpha_{r,k}\in B_n/[P_n, P_n]$ by:
\begin{equation}\label{eq:defdelta}
\text{$\delta_{r,k}=\sigma_{r+k-1}\cdots \sigma_{r+\frac{k+1}{2}} \sigma_{r+\frac{k-1}{2}}^{-1}\cdots \sigma_{r+1}^{-1}$ and $\alpha_{r,k}=\sigma_{r+1}\cdots \sigma_{r+k-1}$.}
\end{equation}

\begin{lem}\label{lem:0}
Let $n, k\geq 3$ and $r\geq 0$ be integers such that $k$ is odd and $r+k\leq n$. Then $\delta_{r,k}$ is of order $k$ in $B_n/[P_n, P_n]$, and satisfies:
\begin{equation}\label{eq:deltarkdef}
\delta_{r,k}=\left(A_{r+\frac{k+1}{2}, r+k} A_{r+\frac{k+3}{2}, r+k} \cdots A_{r+k-1, r+k}\right) \alpha_{r,k}^{-1}.
\end{equation}
Furthermore, the action of conjugation by $\alpha_{r,k}$ on the basis elements $A_{i,j}$ of $P_n/[P_n, P_n]$ is given by:
\begin{equation}\label{eq:conjalpha}
\alpha_{r,k}A_{i,j}\alpha_{r,k}^{-1}=
\begin{cases}
A_{i,j} & \text{if $i,j\notin \brak{r+1, \ldots, r+k}$}\\
A_{i+1,j+1} & \text{if $r+1 \leq i < j \leq r+k-1$}\\
A_{r+1,i+1} & \text{if $r+1 \leq i < j = r+k$}\\
A_{i,j+1} & \text{if $i < r+1 \leq j \leq r+k-1$}\\
A_{i,r+1} & \text{if $i < r+1$ and $j=r+k$}\\
A_{i+1,j} & \text{if $r+1 \leq i \leq r+k-1$ and $r+k < j \leq n$}\\
A_{r+1,j} & \text{if $i=r+k < j \leq n$,}
\end{cases}
%\left\{
%\begin{array}{lcl}
%A_{i,j} & \mbox{if} & i,j\notin \{r+1, \ldots, r+k \},\\
%A_{i+1,j+1} & \mbox{if} & r+1 \leq i < j \leq r+k-1,\\
%A_{r+1,i+1} & \mbox{if} & r+1 \leq i < j = r+k,\\
%A_{i,j+1} & \mbox{if} & i < r+1 \leq j \leq r+k-1,\\
%A_{i,r+1} & \mbox{if} & i < r+1 \textrm{ and } j=r+k,\\
%A_{i+1,j} & \mbox{if} & r+1 \leq i \leq r+k-1 \textrm{ and } r+k < j \leq n,\\
%A_{r+1,j} & \mbox{if} & i=r+k < j \leq n,
%\end{array}
%\right.
\end{equation}
and the action of conjugation by $\delta_{r,k}$ is the inverse action of $\alpha_{r,k}$ and is given by:
\begin{equation}\label{eq:conjdelta}
\delta_{r,k}A_{i,j}\delta_{r,k}^{-1}=
\begin{cases}
A_{i,j} & \text{if $i,j\notin \brak{r+1, \ldots, r+k}$}\\
A_{i-1,j-1} & \text{if $r+2 \leq i < j \leq r+k$}\\
A_{j-1,r+k} & \text{if $r+1 = i < j \leq r+k$}\\
A_{i,j-1} & \text{if $i < r+1 < j \leq r+k$}\\
A_{i,r+k} & \text{if $i < r+1$ and $j=r+1$}\\
A_{i-1,j} & \text{if $r+1 < i \leq r+k$ and $r+k < j \leq n$}\\
A_{r+k,j} & \text{if $r+1=i$ and $r+k < j \leq n$.}
\end{cases}
%\left\{
%\begin{array}{lcl}
%A_{i,j} & \mbox{if} & i,j\notin \{r+1, \ldots, r+k \},\\
%A_{i-1,j-1} & \mbox{if} & r+2 \leq i < j \leq r+k,\\
%A_{j-1,r+k} & \mbox{if} & r+1 = i < j \leq r+k,\\
%A_{i,j-1} & \mbox{if} & i < r+1 < j \leq r+k,\\
%A_{i,r+k} & \mbox{if} & i < r+1 \textrm{ and } j=r+1,\\
%A_{i-1,j} & \mbox{if} & r+1 < i \leq r+k \textrm{ and } r+k < j \leq n,\\
%A_{r+k,j} & \mbox{if} & i=r+1 < j \leq n.
%\end{array}
%\right.
\end{equation}
\end{lem}

\begin{proof}
Let $n\geq 3$, $k\geq 3$ and $r\geq 0$ be integers such that $k$ is odd and $r+k\leq n$. We start by proving~\req{deltarkdef} and by showing that $\delta_{r,k}\in B_n/[P_n, P_n]$ is of order $k$. First let $r=0$. Then:
\begin{align}
\delta_{0,k}\alpha_{0,k} & = \sigma_{k-1}\cdots \sigma_{\frac{k+1}{2}} \sigma_{\frac{k-1}{2}}^{-1}\cdots \sigma_{1}^{-1}\sigma_{1}\cdots \sigma_{\frac{k-1}{2}} \sigma_{\frac{k+1}{2}} \cdots \sigma_{k-1}\notag\\
 & = \sigma_{k-1}\cdots \sigma_{\frac{k+3}{2}} \sigma_{\frac{k+1}{2}}^2 \sigma_{\frac{k+3}{2}} \cdots \sigma_{k-1} = A_{\frac{k+1}{2}, k} A_{\frac{k+3}{2}, k} \cdots A_{k-1, k},\label{eq:deltaalpha}
\end{align}
%
%We note that 
%$$
%\begin{array}{rcl}
%\delta_{0,k}\alpha_{0,k} & = &\sigma_{k-1}\cdots \sigma_{\frac{k+1}{2}} \sigma_{\frac{k-1}{2}}^{-1}\cdots \sigma_{1}^{-1}\sigma_{1}\cdots \sigma_{\frac{k-1}{2}} \sigma_{\frac{k+1}{2}} \cdots \sigma_{k-1}\\
% & = & \sigma_{k-1}\cdots \sigma_{\frac{k+3}{2}} \sigma_{\frac{k+1}{2}}^2 \sigma_{\frac{k+3}{2}} \cdots \sigma_{k-1}\\
% & = & A_{\frac{k+1}{2}, k} A_{\frac{k+3}{2}, k} \cdots A_{k-1, k}.
%\end{array}
%$$
which yields the equality~\reqref{deltarkdef}. Set $N=(A_{\frac{k+1}{2}, k} A_{\frac{k+3}{2}, k} \cdots A_{k-1, k})^{-1}\in P_{n}/[P_{n},P_{n}]$. Then $\alpha_{0,k}^{-1} \delta_{0,k}^{-1} \alpha_{0,k}=N\alpha_{0,k}$ by \req{deltaalpha}, and so $\delta_{0,k}$ and $N\alpha_{0,k}$ are of the same order. Considering $N$ and $\alpha_{0,k}$ to be elements of $B_k/[P_k, P_k]$ for a moment, and using~\reqref{tableij} and~\reqref{sumN}, we have $\displaystyle N=-\sum_{i=1}^{\frac{k-1}{2}}\, e_{i,k-i}$, and so $N$ satisfies the system of equations~\reqref{SistEqns} (taking $n=k$ in that system). It follows from the proof of \repr{Torcaonimpar} that $N\alpha_{0,k}$ is of order $k$ in $B_k/[P_k, P_k]$, and so $\delta_{0,k}$ is of order $k$ in $B_k/[P_k, P_k]$. Since $k\leq n$, we deduce from \reth{oddtorsion}(\ref{it:oddtorsiona}) that $\delta_{0,k}$, considered as an element of $B_n/[P_n, P_n]$, is also of order $k$.

Now assume that $r\geq 1$. Let $\psi$ denote the composition of the following homomorphisms:
\begin{equation*}%\label{eq:inclus}
\frac{B_k}{[P_k, P_k]} \to \frac{B_{r,k,n-r-k}}{[P_n, P_n]} \to \frac{B_n}{[P_n, P_n]},
\end{equation*}
where the first homomorphism is induced by the inclusion $B_{k}\lhra B_{r,k,n-r-k}$ of $B_{k}$ in the middle block of the mixed braid group $B_{r,k,n-r-k}$, and the second homomorphism is induced by the inclusion $B_{r,k,n-r-k}\lhra B_{n}$. In a manner similar to that for $\overline{\iota}$ in the proof of \reth{oddtorsion}(\ref{it:oddtorsiona}), the homomorphism $\psi$ may be seen to be injective. For all $1\leq i\leq n-1$, $\psi(\sigma_i)=\sigma_{r+i}$, hence $\psi(\delta_{0,k})=\delta_{r,k}$ and $\psi(\alpha_{0,k})=\alpha_{r,k}$ by \req{defdelta}. The injectivity of $\psi$ implies that $\delta_{r,k}$ is of order $k$ in $\displaystyle\frac{B_n}{[P_n, P_n]}$. Moreover, for all $1\leq i < j\leq n$, $\psi(A_{i,j})=A_{r+i, r+j}$ by \req{defaij}, thus:
\begin{align*}
\delta_{r,k} &= \psi(\delta_{0,k}) = \psi\left( \left(A_{\frac{k+1}{2}, k} A_{\frac{k+3}{2}, k} \cdots A_{k-1, k}\right) \alpha_{0,k}^{-1} \right)\quad\text{by \req{deltaalpha}}\\
&= \left(A_{r+\frac{k+1}{2}, r+k} A_{r+\frac{k+3}{2}, r+k} \cdots A_{r+k-1, r+k}\right) \alpha_{r,k}^{-1},
\end{align*}
which is \req{deltarkdef}. This proves the first part of the statement. It remains to establish equations~\reqref{conjalpha} and~\reqref{conjdelta}. The first relation of~\reqref{conjalpha} holds clearly. Applying $\psi$ to both sides of equations~\reqref{alphaconj1} and~\reqref{alphaconj2} (and taking $r=k$)
%\begin{equation}\label{eq:conjalpha0}
%\alpha_{0,k}A_{i,j}\alpha_{0,k}^{-1}=
%\begin{cases}
%A_{i+1,j+1} & \text{if $1 \leq i < j \leq k-1$}\\
%A_{1,i+1} & \text{if $1 \leq i < j = k$.}
%\end{cases}
%\end{equation}
gives rise to the second and third relations of~\reqref{conjalpha}. 
%\begin{equation}\label{eq:conjalpha2}
%\alpha_{r,k}A_{i,j}\alpha_{r,k}^{-1}=
%\begin{cases}
%A_{i+1,j+1} & \text{if $r+1 \leq i < j \leq r+k-1$}\\
%A_{r+1,i+1} & \text{if $r+1 \leq i < j = r+k$,}
%\end{cases}
%\end{equation}
%which yields the second and third relations of~\reqref{conjalpha}. 
%
%
%%It is clear that 
%%\begin{equation}\label{eq:conjalpha1}
%%\alpha_{r,k}A_{i,j}\alpha_{r,k}^{-1}=A_{i,j},  \textrm{ for } i,j\notin \brak{r+1, \ldots, r+k}.
%%\end{equation}
%
%From \reco{orbits} we have
%\begin{equation}\label{eq:conjalpha0}
%\alpha_{0,k}A_{i,j}\alpha_{0,k}^{-1}=
%\begin{cases}
%A_{i+1,j+1} & \text{if $1 \leq i < j \leq k-1$}\\
%A_{1,i+1} & \text{if $1 \leq i < j = k$.}
%\end{cases}
%\end{equation}
%
%Observe that applying $\psi$ in \eqref{eq:conjalpha0} we obtain
%\begin{equation}\label{eq:conjalpha2}
%\alpha_{r,k}A_{i,j}\alpha_{r,k}^{-1}=
%\begin{cases}
%A_{i+1,j+1} & \text{if $r+1 \leq i < j \leq r+k-1$}\\
%A_{r+1,i+1} & \text{if $r+1 \leq i < j = r+k$.}
%\end{cases}
%\end{equation}
Finally, \req{conjugAij2} yields the four remaining relations of~\reqref{conjalpha}. To obtain~\reqref{conjdelta}, by \req{deltarkdef}, conjugation by $\alpha_{r,k}\delta_{r,k}$ in $P_{n}/[P_{n},P_{n}]$ is conjugation by an element of $P_{n}/[P_{n},P_{n}]$, which gives rise to the trivial action. So the actions  
by conjugation of $\alpha_{r,k}$ and $\delta_{r,k}$ on $P_{n}/[P_{n},P_{n}]$ are mutual inverses. Equation~\reqref{conjalpha} then implies \req{conjdelta}.
%
%\begin{equation}\label{eq:conjalpha3}
%\alpha_{r,k}A_{i,j}\alpha_{r,k}^{-1}=
%\begin{cases}
%A_{i,j+1} & \text{if $i < r+1 \leq j \leq r+k-1$}\\
%A_{i,r+1} & \text{if $i < r+1$ and $j=r+k$}\\
%A_{i+1,j} & \text{if $r+1 \leq i \leq r+k-1$ and $r+k < j \leq n$}\\
%A_{r+1,j} & \text{if $i=r+k < j \leq n$.}
%\end{cases}
%\end{equation}
%Hence, from \reqref{conjalpha1}, \reqref{conjalpha2} and \reqref{conjalpha3} follows \req{conjalpha}
%\begin{equation*}
%\alpha_{r,k}A_{i,j}\alpha_{r,k}^{-1}=
%\begin{cases}
%A_{i,j} & \text{if $i,j\notin \brak{r+1, \ldots, r+k}$}\\
%A_{i+1,j+1} & \text{if $r+1 \leq i < j \leq r+k-1$}\\
%A_{r+1,i+1} & \text{if $r+1 \leq i < j = r+k$}\\
%A_{i,j+1} & \text{if $i < r+1 \leq j \leq r+k-1$}\\
%A_{i,r+1} & \text{if $i < r+1$ and $j=r+k$}\\
%A_{i+1,j} & \text{if $r+1 \leq i \leq r+k-1$ and $r+k < j \leq n$}\\
%A_{r+1,j} & \text{if $i=r+k < j \leq n$.}
%\end{cases}
%\end{equation*}
\end{proof}

%The braid $\delta_{r,k}$ has order $k$ and projects over the cyclic permutation $\theta=(r+1,r+2,\cdots,r+k)$ which belongs to  the permutation group $\sn$. 

%\comj{This added, as well as the proof:} 
Another corollary of \reth{oddtorsion}(\ref{it:oddtorsionc}), which we now prove, is \reth{realAbelian}, which states that there is a one-to-one correspondence between the finite Abelian subgroups of $B_{n}/[P_{n},P_{n}]$ and the Abelian subgroups of $\sn$ of odd order.

\begin{proof}[Proof of \reth{realAbelian}]
First, it follows from \rerems{restrict}(\ref{it:restrictc}) that the isomorphism class of a finite Abelian subgroup of $B_{n}/[P_{n},P_{n}]$ is realised by a subgroup of $\sn$ (of odd order). Conversely, let $H$ be an Abelian subgroup of $\sn$ of odd order. Then $H$ is isomorphic to a direct product of the form $\Z_{k_{1}}\times \cdots \times \Z_{k_{r}}$, where for $i=1,\ldots,r$, $k_{i}$ is a power of an odd prime number. By~\cite{Ho}, $\sum_{i=1}^{r} k_{i} \leq n$. Let $k_{0}=0$. Then for $l=1,\ldots,r$, the element $\delta_{\sum_{j=0}^{l-1} \,k_{j}, k_{l}}$ belongs to $B_{n}/[P_{n},P_{n}]$ and is of order $k_{l}$ by \relem{0}. By construction, the $\delta_{\sum_{j=1}^{l-1} \,k_{j}, k_{l}}$ commute pairwise. The subgroup $\ang{\delta_{0,k_{1}}, \ldots, \delta_{\sum_{j=1}^{r-1} \,k_{j}, k_{r}}}$ is then isomorphic to $H$ since $\overline{\sigma}\left(\delta_{\sum_{j=1}^{l-1} \,k_{j}, k_{l}}\right)$ is a $k_{l}$-cycle in $\sn$, and the supports of such cycles are pairwise disjoint.
\end{proof}

The following two propositions are immediate consequences of \relem{0}.

\begin{prop}\label{prop:i}
Let $n, k\geq 3$ and $r\geq 0$ be integers such that $k$ is odd, and suppose that $3\leq r+k\leq n$. Then the action of conjugation by $\delta_{r,k}$ on $P_n/[P_n, P_n]$ restricts to an action on the set $A=\setl{A_{i,j}}{r+1\leq i < j\leq r+k}$. The orbits of this action partition the set $A$ 
%Then, there is a partition of the  set of pure braids
%$A=\{ A_{i,j} \mid r+1\leq i < j\leq r+k \}\subset P_n/[P_n, P_n]$ 
into $\frac{k-1}{2}$ orbits each of length $k$ and given by:
%, where the action of $\delta_{r,k}$ restricted to 
%each orbit  is a cyclic permutation of order $k$ of the elements  of the orbit. 
%Furthermore, the orbits  are given by
%\begin{multline*}
%A_{r+1,r+i+1} \stackrel{\alpha_{r,k}}{\longmapsto} A_{r+2,r+i+2} \stackrel{\alpha_{r,k}}{\longmapsto} \cdots \stackrel{\alpha_{r,k}}{\longmapsto} A_{r+k-i,r+k} \stackrel{\alpha_{r,k}}{\longmapsto} A_{r+1,r+k-i+1} \stackrel{\alpha_{r,k}}{\longmapsto}\\ A_{r+2,r+k-i+2} \stackrel{\alpha_{r,k}}{\longmapsto} \cdots\stackrel{\alpha_{r,k}}{\longmapsto} A_{r+i,r+k} \stackrel{\alpha_{r,k}}{\longmapsto} A_{r+1, r+i+1}
%\end{multline*}
\begin{multline*}
A_{r+1, r+i+1} \stackrel{\delta_{r,k}}{\longmapsto} A_{r+i,r+k} \stackrel{\delta_{r,k}}{\longmapsto} A_{r+i-1,r+k-1} \stackrel{\delta_{r,k}}{\longmapsto} \cdots \stackrel{\delta_{r,k}}{\longmapsto} A_{r+2,r+k-i+2} \stackrel{\delta_{r,k}}{\longmapsto}\\
A_{r+1,r+k-i+1} \stackrel{\delta_{r,k}}{\longmapsto} A_{r+k-i,r+k} \stackrel{\delta_{r,k}}{\longmapsto}\cdots \stackrel{\delta_{r,k}}{\longmapsto} A_{r+2,r+i+2} \stackrel{\delta_{r,k}}{\longmapsto} A_{r+1,r+i+1} \quad \text{for $i=1,\ldots, \frac{k-1}{2}$.}   
\end{multline*}
%\begin{itemize}
%	\item $A_{r+1,r+k}\stackrel{\delta_{r,k}}\longmapsto A_{r+k-1, r+k}\stackrel{\delta_{r,k}}\longmapsto \cdots \stackrel{\delta_{r,k}}\longmapsto A_{r+2,r+3}\stackrel{\delta_{r,k}}\longmapsto A_{r+1,r+2} \stackrel{\delta_{r,k}}\longmapsto A_{r+1,r+k}$.
%
%	\item $A_{r+2,r+k}\stackrel{\delta_{r,k}}\longmapsto A_{r+1, r+k-1}\stackrel{\delta_{r,k}}\longmapsto A_{r+k-2, r+k}\stackrel{\delta_{r,k}}\longmapsto \cdots \stackrel{\delta_{r,k}}\longmapsto A_{r+2,r+4}\stackrel{\delta_{r,k}}\longmapsto A_{r+1,r+3} \stackrel{\delta_{r,k}}\longmapsto A_{r+2,r+k}$
%
%	\item $A_{r+3,r+k}\stackrel{\delta_{r,k}}\longmapsto A_{r+2, r+k-1}\stackrel{\delta_{r,k}}\longmapsto A_{r+1, r+k-2}\stackrel{\delta_{r,k}}\longmapsto A_{r+k-3, r+k}\stackrel{\delta_{r,k}}\longmapsto \cdots \stackrel{\delta_{r,k}}\longmapsto A_{r+2,r+5}\stackrel{\delta_{r,k}}\longmapsto A_{r+1,r+4} \stackrel{\delta_{r,k}}\longmapsto A_{r+3,r+k}$.
%
%$$\vdots$$
%
%	\item $A_{r+\frac{k-1}{2},r+k}\stackrel{\delta_{r,k}}\longmapsto A_{r+\frac{k-3}{2}, r+k-1}\stackrel{\delta_{r,k}}\longmapsto \cdots \stackrel{\delta_{r,k}}\longmapsto A_{r+1, r+\frac{k+3}{2}}\stackrel{\delta_{r,k}}\longmapsto A_{r+\frac{k+1}{2}, r+k}\stackrel{\delta_{r,k}}\longmapsto \cdots \stackrel{\delta_{r,k}}\longmapsto A_{r+2, r+\frac{k+3}{2}}\stackrel{\delta_{r,k}}\longmapsto A_{r+1, r+\frac{k+1}{2}} \stackrel{\delta_{r,k}}\longmapsto A_{r+\frac{k-1}{2},r+k}$.
%\end{itemize}
\end{prop}

\begin{proof}
The result follows from the second and third lines of~\req{conjdelta}.
\end{proof}

\begin{prop}\label{prop:ii}
Let $n, k\geq 3$ and $r\geq 0$ be integers such that $k$ is odd, and suppose that $3 \leq r+k\leq n$. Then for each  $j>r+k$, the action of conjugation by $\delta_{r,k}$ on $P_n/[P_n, P_n]$ restricts to a transitive action on the set $\setl{A_{i,j}}{r+1\leq i \leq r+k}$, whose orbit of length $k$ is given by:
%which is an orbit of the action 
% cyclic permutation 
%of order $k$.
%Furthermore, as a permutation is a cycle given by
\begin{equation}\label{eq:actAj}
A_{r+k, j}\stackrel{\delta_{r,k}}\longmapsto A_{r+k-1,j} \stackrel{\delta_{r,k}}\longmapsto  \cdots \stackrel{\delta_{r,k}}\longmapsto A_{r+2,j} \stackrel{\delta_{r,k}}\longmapsto A_{r+1,j} \stackrel{\delta_{r,k}}\longmapsto A_{r+k,j}.
\end{equation}
Similarly, if $i<r+1$, the action of conjugation by $\delta_{r,k}$ on $P_n/[P_n, P_n]$ restricts to a transitive action on the set $A=\setl{A_{i,j}}{r+1\leq j \leq  r+k}$, whose orbit of length $k$ is given by:
\begin{equation}\label{eq:actAi}
A_{i, r+k}\stackrel{\delta_{r,k}}\longmapsto A_{i,r+k-1} \stackrel{\delta_{r,k}}\longmapsto  \cdots \stackrel{\delta_{r,k}}\longmapsto A_{i,r+2} \stackrel{\delta_{r,k}}\longmapsto A_{i,r+1} \stackrel{\delta_{r,k}}\longmapsto A_{i,r+k}.
\end{equation}
\end{prop}

\begin{proof}
Equation~\reqref{actAj} (resp.\ \req{actAi}) follows from the $4\up{th}$ and $5\up{th}$ lines (resp.\ the $6\up{th}$ and $7\up{th}$ lines) of \req{conjdelta}.
%
%This result follows from the action of $\delta_{r,k}$ given in (\ref{eq:conjdelta}), see lines four and five on the right hand of \eqref{eq:conjdelta} for the first statement and lines six and seven for the second one.
\end{proof}

%For \comj{conditions on $r,k,s,l$?} Let $\theta_1=(r+1,r+2,\ldots,r+k)$ and $\theta_2=(s+1,s+2,\ldots,s+l)$ be the cycles in the permutation group $\sn$. 

\begin{prop}\label{prop:iii}
Let $n, k, l\geq 3$ and $r, s\geq 0$ be integers such that $k$ and $l$ are odd, $3 \leq r+k < s+1$ and $s+l\leq n$, and let $\ell_0=\operatorname{lcm}(k,l)$. The action of conjugation by $\delta_{r,k}\delta_{s,l}$ on $P_n/[P_n, P_n]$ restricts to an action on the set $\setl{A_{i,j}}{\text{$r+1 \leq i \leq r+k$ and $s+1 \leq j \leq s+l$}}$, given by:
\begin{equation*}
\delta_{r,k}\delta_{s,l} A_{i,j} (\delta_{r,k}\delta_{s,l})^{-1} = 
\begin{cases}
A_{i-1,j-1} & \text{if $r+1<i\leq r+k$ and $s+1<j\leq s+l$}\\
A_{r+k,j-1} & \text{if $i=r+1$ and $s+1<j\leq s+l$}\\
A_{i-1,s+l} & \text{if $r+1<i\leq r+k$ and $j=s+1$}\\
A_{r+k,s+l} & \text{if $i=r+1$ and $j=s+1$.}
\end{cases}
\end{equation*}
The orbits of the action partition this set into $kl/\ell_0$ orbits of length $\ell_0$ given by combining~\reqref{actAj} and~\reqref{actAi}.
%Further,  the action splits the set $A$ 
%%Then, there is a partition of the  set of pure braids
%%$A=\{ A_{i,j} \mid r+1\leq i < j\leq r+k \}\subset P_n/[P_n, P_n]$ 
%%into   $\frac{k-1}{2}$ orbits of length $k$, where     the action of $\delta_{r,k}$ restricted to 
%%each orbit  is a cyclic permutation of order $k$ of the elements  of the orbit. 
%%Furthermore, the orbits  are given by
%into $kl/\ell_0$  orbits  of length $\ell_0$, where each orbit  is a cyclic permutation of order $\ell_0$.
%%such that the action of  $\delta_{r,k}\delta_{s,l}$ restricted to each
%%subset is a cyclic permutation of order $\ell_0$.
%Furthermore, the cycles are given by mixing (\ref{eq:actAj}) and (\ref{eq:actAi}).
\end{prop}

\begin{proof}
The result follows by applying Propositions~\ref{prop:i} and~\ref{prop:ii} to the action of $\delta_{r,k}$ and $\delta_{s,l}$ on the elements of the set $\setl{A_{i,j}}{\text{$r+1 \leq i \leq r+k$ and $s+1 \leq j \leq s+l$}}$.
%\begin{align*}
%\delta_{r,k}A_{i,j}\delta_{r,k}^{-1} & =\text{$A_{i-1, j-1}$ if $r+1 < i \leq r+k$ and  $s+1 < j \leq s+l$}\\
%\delta_{r,k}A_{r+1,j}\delta_{r,k}^{-1}&= \text{$A_{r+k, j-1}$ if $s+1 < j \leq s+l$}\\
%\delta_{r,k}A_{i,s+1}\delta_{r,k}^{-1}&=\text{$A_{i-1, r+k}$ if $r+1 < i \leq r+k$.}
%\end{align*}
%The result follows from this.
\end{proof}

%The proof of the following proposition is straightforward and is omitted.

\begin{prop}\label{prop:iv}
Let $n\geq 3$, and let $\beta$ be an element of $B_n/[P_n, P_n]$. If $m\geq 0$ is such that $\overline{\sigma}(\beta)$ belongs to the subgroup of $\sn$ isomorphic to $\sn[q]$ on the symbols $\brak{m+1, m+2, \ldots, m+q}$ then the action of $\beta$ on the set $\setl{A_{i,j}}{\text{$1\leq i < j \leq m$ or $m+q+1\leq i < j \leq n$}}$ is trivial.
%$ whose image under $\map{\overline{\sigma}}{B_n/[P_n, P_n]}[\sn]$ belongs to the subgroup $\sn[q]$ of $\sn$ on $\brak{1, 2, \ldots, q}$ for some $q<n-1$. Then the action of $\beta$ by conjugation on the set $\setl{A_{i,j}}{q+1\leq i < j \leq n}$ is trivial.
\end{prop}

\begin{proof} 
We just prove the claim for $m=0$ since the remaining cases are similar. 
Let $1\leq q <n-1$. By \reth{oddtorsion}(\ref{it:oddtorsiona}), the inclusion $\map{\iota}{B_{q}}[B_{n}]$ induces an injective homomorphism $\map{\overline{\iota}}{B_{q}/[P_q, P_q]}[B_n/[P_n, P_n]]$. Since $\overline{\sigma}(\beta)\in \sn[q]$, there exists $\tau\in B_{q}/[P_q, P_q]$ such that $\overline{\sigma}(\beta)=\overline{\sigma}(\overline{\iota}(\tau))$ is the identity permutation, and so there exists $\beta'\in P_{n}/[P_n, P_n]$ such that 
$\beta=\overline{\iota}(\tau) \beta'$. The result follows from \req{conjugAij2} and the fact that $\beta'$ is central in $P_{n}/[P_n, P_n]$.
\end{proof}

%Let $r\geq 0$, let $k\geq 3$ be odd, and suppose that $r+k\leq n$. By \relem{0}, the element $\delta_{r,k}$ defined by \req{defdelta}
%\begin{equation}\label{eq:deltark}
%\delta_{r,k}=\sigma_{r+k-1}\cdots \sigma_{r + \frac{k+1}{2}} \sigma^{-1}_{r + \frac{k-1}{2}} \cdots \sigma^{-1}_{r+1}
%\end{equation} 
%is of order $k$. 
%Note that $\overline{\sigma}(\delta_{r,k})=(r+1,r+2,\ldots,r+k)$. The elements $N\in P_n/[P_n, P_n]$ for which $N\delta_{r,k}$ is of order $k$ were characterised in \repr{Torcaonimpar}. 

Let $n\geq 3$, let $k_{0}=0$, let $3\leq k_1\leq k_2 \leq \ldots \leq k_s$ be odd, and suppose that $\sum_{j=1}^{s}\, k_{j}\leq n$. We define:
%
%Let $\underline{k}=(k_{1},\ldots, k_{s})$ and
%\comj{%perhaps write  and $\delta$ as $\delta_{\underline{k}}$ since $\delta$ depends on the choice of the $k_{i}$. 
%It might also be convenient to introduce some notation for $\sum_{j=1}^{l} \,k_{j}$ to avoid having too many subscripts.}
%
%\comdo{We started to use the notation $k^{l}=\sum_{j=1}^{l} \,k_{j}$ but later we noted that maybe will not be efficient, in part because the reader should forget what it is. So we are not using it.}
\begin{equation}\label{eq:delta}
\delta=\delta(k_{0},\ldots,k_{s})=\delta_{k_{0},k_1}\delta_{k_1,k_2}\delta_{k_1+k_2, k_3}\cdots\delta_{\sum_{j=1}^{s-1} \,k_{j}, k_s},
\end{equation}
where for all $0\leq l\leq s-1$, the element $\delta_{\sum_{j=1}^{l} \,k_{j}, k_{l+1}}$ is given by \req{defdelta}. Since the $\delta_{\sum_{j=1}^{l} \,k_{j}, k_{l+1}}$ commute pairwise and $\delta_{\sum_{j=1}^{l} \,k_{j}, k_{l+1}}$ is of order $k_{l+1}$ by \relem{0}, it follows that $\delta$ is of order $\operatorname{lcm}(k_1,\ldots,k_s)$ in $B_n/[P_n, P_n]$ and:
\begin{equation}\label{eq:theta}
\theta=\overline{\sigma}(\delta)=\theta_1 \cdots \theta_s, 
\end{equation}
where for $i=1,\ldots,s$, $\theta_{i}$ is the $k_{i}$-cycle defined by:
\begin{equation}\label{eq:thetai}
\theta_{i}=\left(\sum_{j=1}^{i-1} \,k_{j}+1, \sum_{j=1}^{i-1} \,k_{j}+2, \ldots, \sum_{j=1}^{i} \,k_{j}\right).
\end{equation}
%\begin{equation*}
%\theta_r=\left((k_1 + \cdots + k_{r-1}+1),(k_1 + \cdots + k_{r-1}+2), \cdots, (k_1 + \cdots + k_{r-1} + k_r)\right),
%\end{equation*}
%for all $1 \leq r \leq s$. 
The order of the permutation $\theta$ is also equal to $\operatorname{lcm}(k_1, k_2, \ldots, k_s)$. Using Propositions~\ref{prop:i}, \ref{prop:ii},~\ref{prop:iii} and~\ref{prop:iv}, we shall now describe the orbits given by the action of conjugation by $\delta$ on the basis $\setl{A_{i,j}}{1\leq i < j \leq n}$ of $P_n/[P_n, P_n]$. The associated partition will be useful when it comes to proving \reth{classconj}.

\begin{thm}\label{th:gener}
Let $n\geq 3$, let $k_{0}=0$, let $3\leq k_1\leq k_2 \leq \ldots \leq k_s$ be odd such that $\sum_{j=1}^{s}\, k_{j}\leq n$, and let $\delta \in B_{n}/[P_{n},P_{n}]$ be defined by \req{delta}. The following sets are disjoint and invariant under the action by conjugation of $\delta$ on the set of basis elements $\setl{A_{i,j}}{1\leq i < j \leq n}$ of $P_n/[P_n, P_n]$: 
% in four invariant sets with respect to  the conjugacy action of $\delta$ and is given below. 
\begin{enumerate}[(a)]
\item\label{it:genera} $\setr{A_{i,j}}{\text{$\sum_{l=1}^{r-1}\, k_{l}+1 \leq i < j \leq \sum_{l=1}^{r}\, k_{l}$}}$, where $1 \leq r \leq s$. Under the given action, the orbits of this set are obtained from the relations $e_{r,h,t} \stackrel{\delta}{\mapsto} e_{r,h,t+1}$, where $1\leq h\leq \frac{k_r-1}{2}$, the index $t$ is taken modulo $k_{r}$, and
\begin{equation*}
e_{r,h,t}=\begin{cases}
A_{\sum_{l=1}^{r-1}\, k_{l}+h-t+1, \sum_{l=1}^{r}\, k_{l}-t+1} & \text{if $t\in \brak{1,\ldots,h}$}\\
A_{\sum_{l=1}^{r}\, k_{l}-t+1, \sum_{l=1}^{r}\, k_{l}-t+1+h} & \text{if $t\in \brak{h+1,\ldots,k_{r}}$.}
\end{cases}
\end{equation*}

\item\label{it:generb} $\setr{A_{i,j}}{\text{$\sum_{l=1}^{r-1}\, k_{l}+1 \leq i \leq \sum_{l=1}^{r}\, k_{l}$ and $\sum_{l=1}^{s}\, k_{l} < j \leq n$}}$, where $1 \leq r \leq s$. Under the given action, the orbits of this set are obtained from the relations $e_{r,j,t} \stackrel{\delta}{\mapsto} e_{r,j,t+1}$, where the index $t$ is taken modulo $k_{r}$, and $e_{r,j,t}=A_{\sum_{l=1}^{r-1}\, k_{l}+t,j}$.
%\comj{change the notation $e_{r,j,t}$, since it isn't coherent with that of the previous item?}

\item\label{it:generc} $\setr{A_{i,j}}{\text{$\sum_{l=1}^{p-1}\, k_{l} +1 \leq i \leq \sum_{l=1}^{p}\, k_{l}$ and $\sum_{l=1}^{q-1}\, k_{l}+1 \leq j \leq \sum_{l=1}^{q}\, k_{l}$}}$, where $1 \leq p<q \leq s$. Under the given action, the orbits of this set are obtained from the relations $e_{p,q,v,t} \stackrel{\delta}{\mapsto} e_{p,q,v,t+1}$, where  $\displaystyle 1\leq v \leq \frac{k_p\cdot k_q}{\operatorname{lcm}(k_p,k_q)}$, $1\leq t \leq \operatorname{lcm}(k_p,k_q)$, and
\begin{equation*}
e_{p,q,v,t}=A_{\sum_{l=1}^{p-1}\, k_{l} +[2 - t]_{k_p},\sum_{l=1}^{q-1}\, k_{l}+[1 -t +v]_{k_q}},
\end{equation*}
where the notation $[x]_n$ means the positive integer between $1$ and $n$ that is congruent to $x$ modulo $n$.

\item\label{it:generd} $\setl{A_{i,j}}{\text{$\sum_{l=1}^{s}\, k_{l} < i < j \leq n$}}$. Under the given action, each $A_{i,j}$ is fixed.
\end{enumerate}
%\begin{enumerate}[(a)]
%\item $A_{i,j}$ such that $k_1 + \cdots + k_{r-1}+1 \leq i < j \leq k_1 + \cdots + k_{r-1} + k_r$, for all $1 \leq r \leq s$.
%\item $A_{i,j}$ such that $k_1 + \cdots + k_{r-1}+1 \leq i \leq k_1 + \cdots + k_{r-1} + k_r$, for all $1 \leq r \leq s$ and $k_1+k_2+\cdots k_s < j \leq n$. 
%\item $A_{i,j}$ with $i<j$ such that $k_1 + \cdots + k_{r-1}+1 \leq i \leq k_1 + \cdots + k_{r-1} + k_r$, for all $1 \leq r \leq s$, and $k_1 + \cdots + k_{t-1}+1 \leq j \leq k_1 + \cdots + k_{t-1} + k_t$, for all $1 \leq t \leq s$. 
%\item $A_{i,j}$ such that $k_1 + \cdots + k_{s-1} + k_s < i < j \leq n$. 
%\hfill $\Box$
%\end{enumerate}
\end{thm}

\begin{proof}
Parts~(\ref{it:genera})--(\ref{it:generd}) follow from Propositions~\ref{prop:i},~\ref{prop:ii},~\ref{prop:iii} and~\ref{prop:iv} respectively.
\end{proof}

Let $k_{1},\ldots,k_{s}$ be as in the statement of \reth{gener}, and let $\mathcal{B}$ denote the basis of $P_{n}/[P_{n},P_{n}]$ that consists of the following elements:
\begin{enumerate}
\item\label{it:basea} $e_{r,h,t}$, where $1\leq r\leq s$, $\displaystyle 1\leq h\leq \frac{k_r-1}{2}$ and $1\leq t\leq k_{r}$.
\item\label{it:baseb} $e_{r,j,t}$, where $1\leq r\leq s$, $\sum_{l=1}^{s}\, k_{l}<j<n$ and $1\leq t\leq k_{r}$.
\item\label{it:basec} $e_{p, q, v, t}$, where $1\leq p<q\leq s$, $\displaystyle 1\leq v \leq \frac{k_p\cdot k_q}{\operatorname{lcm}(k_p,k_q)}$ and $1\leq t\leq \operatorname{lcm}(k_p,k_q)$.
\item\label{it:based} $A_{i,j}$, where $\sum_{l=1}^{s}\, k_{l}<i<j\leq n$.
\end{enumerate}
An element of $\mathcal{B}$ will then be said to be of type~(a),~(b)~(c) or~(d) respectively.
If $A \in P_n/[P_n, P_n]$, it may thus be written uniquely in the following form:
\begin{equation}\label{eq:coeffA}
A=\prod_{\substack{1\leq r\leq s \\ 1\leq h\leq \frac{k_r-1}{2} \\ 1\leq t \leq k_r}}   e_{r,h,t}^{m_{r,h,t}} 
\prod_{\substack{1\leq r\leq s\\ \sum_{j=1}^{s} \,k_{j} +1 \leq j\leq n\\ 1\leq t \leq k_r}}  e_{r, j, t}^{m_{r, j, t}} 
\prod_{\substack{1\leq p < q \leq s \\ 1\leq v \leq \frac{k_p\cdot k_q}{\operatorname{lcm}(k_p,k_q)} \\ 1\leq t\leq \operatorname{lcm}(k_p,k_q)   }}   e_{p, q, v, t}^{m_{p, q, v, t}}
 \prod_{\sum_{j=1}^{s} \,k_{j} < i < j \leq n}   A_{i,j}^{m_{i,j}}.
\end{equation}
The following proposition allows us to decide whether $B_{n}/[P_{n},P_{n}]$ possesses elements of order $\operatorname{lcm}(k_1,\ldots,k_s)$.

\begin{prop}\label{prop:adelta}
Let $n\geq 3$, let $k_{0}=0$, let $3\leq k_1\leq k_2 \leq \ldots \leq k_s$ be odd such that $\sum_{j=1}^{s}\, k_{j}\leq n$, let $\delta \in B_{n}/[P_{n},P_{n}]$ be defined by \req{delta}, and let $A\in P_n/[P_n, P_n]$ be given by \req{coeffA}. Then the element $A\delta$ is of order $\operatorname{lcm}(k_1,\ldots,k_s)$ if and only if the following system of equations is satisfied:
\begin{equation}\label{eq:sysAdelta}
\begin{cases}
\displaystyle\sum_{\substack{1\leq t\leq k_r}} m_{r,h,t}  = 0 & \text{for all  $1\leq r\leq s$ and $\displaystyle 1\leq h\leq \frac{k_r-1}{2}$}\\
\displaystyle\sum_{\substack{1\leq t\leq k_r}} m_{r,j,t} = 0 & \text{for all $1\leq r\leq s$ and $\displaystyle\sum_{j=1}^{s} \,k_{j} +1 \leq j\leq n$}\\
\displaystyle\sum_{\substack{1\leq t\leq \operatorname{lcm}(k_p,k_q)}} m_{p, q,v, t} = 0 & \text{for all $1\leq p < q \leq s$ and $\displaystyle 1\leq v \leq \frac{k_p\cdot k_q}{\operatorname{lcm}(k_p,k_q)}$}\\
m_{i,j} = 0 & \text{for all $\displaystyle\sum_{j=1}^{s} \,k_{j} < i < j \leq n$}.
\end{cases}
\end{equation}
%$$
%\left\{
%\begin{array}{rcll}
%\displaystyle\sum_{\substack{1\leq t\leq k_r}} m_{r,h,t} & = & 0 & \textrm{for } 1\leq r\leq s \textrm{ and } 1\leq h\leq \displaystyle\frac{k_r(k_r-1)}{2},\\
%\displaystyle\sum_{\substack{1\leq t\leq k_r}} m_{r,j,t} & = & 0 & \textrm{for } 1\leq r\leq s \textrm{ and } k_1 + \cdots + k_s +1 \leq j\leq n,\\
%\displaystyle\sum_{\substack{1\leq t < q \leq \operatorname{lcm}(k_p,k_q)}} m_{k_p, k_q,v, t} & = & 0 & \textrm{for } k_p < k_q\in \{k_1, \ldots, k_s\} \textrm{ and } 1\leq v \leq \frac{k_p\cdot k_q}{\operatorname{lcm}(k_p,k_q)},\\
%m_{4,i,j} & = & 0 & \textrm{for } k_1 + \cdots + k_{s-1} + k_s < i < j \leq n.
%\end{array}
%\right.
%$$
\end{prop}

\begin{proof} 
%\comj{This has been rewritten somewhat and some details added.}
The argument is similar to that of the proof of \repr{Torcaonimpar}.
Let $A$ be written in the form of \req{coeffA}, and let $\ell=\operatorname{lcm}(k_1,\ldots,k_s)$. Since $A\in P_{n}/[P_{n},P_{n}]$, $\overline{\sigma}(A\delta)=\overline{\sigma}(\delta)=\theta$, where $\theta$ is as defined in \req{theta}. The fact that $\theta$ is of order $\ell$ implies that the order of $A\delta$, if it is finite, cannot be less than $\ell$. Since $\delta$ is of order $\ell$ by \relem{0}, it follows that:
\begin{equation}\label{eq:Adelta}
(A\delta)^{\ell}= \prod_{j=0}^{\ell-1}\, \delta^{j} A\delta^{-j}.
\end{equation}
%Let $w_{1}$ denote one of the basis elements of $P_{n}/[P_{n},P_{n}]$ given in \req{coeffA}, so:
%\begin{equation}\label{eq:widef}
%w_{1}=\begin{cases}
%e_{r,h,t} & \text{for some $1\leq r\leq s$, $\displaystyle 1\leq h\leq \frac{k_r-1}{2}$ and $1\leq t\leq k_{r}$}\\
%e_{r,j,t} & \text{for some $1\leq r\leq s$, $\sum_{l=1}^{s}\, k_{l}<j<n$ and $1\leq t\leq k_{r}$}\\
%e_{k_p, k_q, v, t} & \text{for some $1\leq p<q\leq s$, $\displaystyle 1\leq v \leq \frac{k_p\cdot k_q}{\operatorname{lcm}(k_p,k_q)}$ and $1\leq t\leq \operatorname{lcm}(k_p,k_q)$}\\
%e_{i,j} & \text{for some $\sum_{l=1}^{s}\, k_{l}<i<j\leq n$.}
%\end{cases}
%\end{equation}
Let $w_{1}\in \mathcal{B}$, and let $q$ denote the length of the orbit of $w_{1}$ under the action of conjugation by $\delta$. By \reth{gener}, $q=k_{r}$ if $w_{1}$ is of type~(a) or~(b), 
$q=\operatorname{lcm}(k_p,k_q)$ if $w_{1}$ is of type~(c), and $q=1$ if $w_{1}$ is of type~(d). For $i=1,\ldots,q$, let $w_{i}=\delta^{i-1}w_{1}\delta^{-(i-1)}$ be the (distinct) elements of the orbit of $w_{1}$. So $\delta^{q}w_{i}\delta^{-q}=w_{i}$, and since $q$ divides $\ell$, we have:
\begin{equation*}
\prod_{j=0}^{\ell-1} \delta^{j} w_{i} \delta^{-j}= \prod_{\substack{0\leq j\leq q-1\\ 0\leq k\leq \frac{\ell}{q}-1}} \delta^{kq+j} w_{i} \delta^{-(kq+j)}=\left(\prod_{0\leq j\leq q-1} \delta^{j} w_{i} \delta^{-j}\right)^{\ell/q}=(w_{1}\cdots w_{q})^{\ell/q}.
\end{equation*}
If $m_{i}\in \Z$ then for $i=1,\ldots,q$, we have:
\begin{equation}\label{eq:prodorbit}
\left(\prod_{i=1}^{q} w_{i}^{m_{i}} \delta\right)^{\ell}=\prod_{j=0}^{\ell-1} \delta^{j} \left( \prod_{i=1}^{q} w_{i}^{m_{i}}\right) \delta^{-j}= \prod_{i=1}^{q} (w_{1}\cdots w_{q})^{m_{i}\ell/q}= (w_{1}\cdots w_{q})^{\frac{\ell}{q}\sum_{i=1}^{q} m_{i}}.
\end{equation}
Combining equations~\reqref{Adelta} and~\reqref{prodorbit} and using the fact that the orbits of the elements of $\mathcal{B}$ are invariant under conjugation by $\delta$, it follows that $(A\delta)^{\ell}=1$ if and only if
\begin{equation}\label{eq:summi}
\text{$\sum_{i=1}^{q} m_{i}=0$ for all $w_{1}\in \mathcal{B}$,}
\end{equation}
where for $i=1,\ldots,q$, $m_{i}$ is the coefficient of $w_{i}$ that appears in \req{coeffA}. Taking $w_{1}$ to be successively the element $e_{r,h,1}$ of type~(a), the element $e_{r,j,1}$ of type~(b), the element $e_{p, q, v, 1}$ of type~(c), and the element $A_{i,j}$ of type~(d), we conclude that $(A\delta)^{\ell}=1$ if and only if the system of equations~\reqref{sysAdelta} is satisfied, and this completes the proof of the proposition.
\end{proof}

We now prove \reth{classconj} that concerns the conjugacy classes of finite-order elements of $B_n/[P_n, P_n]$, and which is the main result of this section.

%\begin{thm}\label{th:classconj}
%Let $n\geq 3$, and let $k\geq 3$ be odd. Two elements of $B_n/[P_n, P_n]$ of order $k$ are conjugate if and only if their permutations have the same cycle type. Thus two finite cyclic subgroups of $B_n/[P_n, P_n]$ of order $k$ are conjugate if and only if their images under $\overline{\sigma}$ are conjugate in $\sn$. 
%%The number of conjugacy classes of elements of order $k$ in $B_n/[P_n, P_n]$ is equal to the number of conjugacy classes of permutations of order $k$ in $\sn$. In particular, the number of conjugacy classes of cyclic subgroups  of order $k$ in $B_n/[P_n, P_n]$ is equal to the number of conjugacy classes of cyclic subgroups of order $k$ in $\sn$.
%\end{thm}

%\comj{Maybe this result could stated in another way such as: `Two elements of $B_n/[P_n, P_n]$ of order $k$ are conjugate if and only if their permutations have the same cycle type.'}

\begin{proof}[Proof of \reth{classconj}]
Let $\theta\in \sn[n]$ be of order $k$. Conjugating $\theta$ if necessary, we may suppose that there exist odd integers $3\leq k_{1}\leq \ldots\leq k_{s}$ such that $\sum_{i=1}^{s} \, k_{i}\leq n$ and $k=\operatorname{lcm}(k_1, \ldots, k_s)$ for which $\theta$ is of the form given by \req{theta}, and where the elements $\theta_{i}$ of that equation are defined by \req{thetai}. Let $\delta\in B_{n}/[P_{n},P_{n}]$ be defined by \req{delta}, which we know to be of order $k$ using \relem{0}. Now let $\beta\in B_{n}/[P_{n},P_{n}]$ be an element of finite order such that $\overline{\sigma}(\beta)=\theta$. By \relem{eleconj}, $\beta$ is of order $k$. To prove \reth{classconj}, it suffices to show that $\beta$ and $\delta$ are conjugate. Since they have the same permutation, there exists $A\in B_{n}/[P_n, P_n]$ such that $\beta=A \delta$, and we may write $A$ in the form of \req{coeffA}. With the notation of the proof of \repr{adelta}, \req{summi} holds by that proposition because $A\delta$ is of order $\operatorname{lcm}(k_1, \ldots, k_s)$. 
To prove the theorem, it suffices to show that $A\delta$ and $\delta$ are conjugate. To do so, we will exhibit $X\in P_n/ [P_n, P_n]$ for which $XA\delta X^{-1} = \delta$. This is equivalent to the following relation:
\begin{equation}\label{eq:adelta}
\text{$XA\delta X^{-1} \delta^{-1}=1$ in $P_n/ [P_n, P_n]$.}
\end{equation}
We start by writing $X$ in the form of \req{coeffA} as follows:
\begin{equation}\label{eq:coeffX}
X=\prod_{\substack{1\leq r\leq s \\ 1\leq h\leq \frac{k_r-1}{2} \\ 1\leq t \leq k_r}}   e_{r,h,t}^{x_{r,h,t}} 
\prod_{\substack{1\leq r\leq s\\ \sum_{j=1}^{s} \,k_{j} +1 \leq j\leq n\\ 1\leq t \leq k_r}}  e_{r, j, t}^{x_{r, j, t}} 
\prod_{\substack{1\leq p< q\leq s \\ 1\leq v \leq \frac{k_p\cdot k_q}{\operatorname{lcm}(k_p,k_q)} \\ 1\leq t\leq \operatorname{lcm}(k_p,k_q)   }}   e_{p, q, v, t}^{x_{p, q, v, t}}
 \prod_{\sum_{j=1}^{s} \,k_{j} < i < j \leq n}   A_{i,j}^{x_{i,j}},
\end{equation}
where the exponents are the coefficients of the elements of $\mathcal{B}$.
%\begin{equation*}
%X=\prod_{\substack{1\leq r\leq s \\ 1\leq h\leq \frac{k_r(k_r-1)}{2} \\ 1\leq t \leq k_r}}   e_{r,h,t}^{x_{r,h,t}} 
%\prod_{\substack{1\leq r\leq s\\ k_1 + \cdots + k_s +1 \leq j\leq n\\ 1\leq t \leq k_r}}  e_{r, j, t}^{x_{r, j, t}} 
%\prod_{\substack{k_p < k_q \\ k_p, k_q\in \{k_1, \ldots, k_s\} \\ 1\leq v \leq \frac{k_p\cdot k_q}{\operatorname{lcm}(k_p,k_q)} \\ 1\leq t\leq \operatorname{lcm}(k_p,k_q)   }}   e_{k_p, k_q, v, t}^{x_{k_p, k_q, v, t}}
% \prod_{k_1 + \cdots + k_{s-1} + k_s < i < j \leq n}   e_{i,j}^{x_{i,j}}.
%\end{equation*}
%We will now determine \comj{necessary and sufficient?} conditions to obtain the equality $XA\delta X^{-1} = \delta$, i.e. 
%\begin{equation}\label{eq:adelta}
%XA\delta X^{-1} \delta^{-1} = 0 \in P_n/[P_n, P_n].
%\end{equation}
As we saw in the proof of \repr{adelta}, it suffices to study the subsystems obtained from \req{adelta} that correspond to the orbits of the action of conjugation by $\delta$. In particular, if $w_{1}\in \mathcal{B}$ and $w_{i}=\delta^{i-1}w_{1}\delta^{-(i-1)}$ are the elements of the orbit of $w_{1}$, where $i=1,\ldots,q$, then it follows from equations~\reqref{coeffA},~\reqref{adelta} and~\reqref{coeffX} that:
\begin{equation*}
\text{$\left( \prod_{i=1}^{q} w_{i}^{x_{i}} \right) \left( \prod_{i=1}^{q} w_{i}^{m_{i}} \right) \left( \prod_{i=1}^{q} w_{i}^{-x_{i-1}} \right)=1$ in $P_n/ [P_n, P_n]$,}
\end{equation*}
where $m_{i}$ (resp.\ $x_{i}$) is the coefficient of $w_{i}$ appearing in \req{coeffA} (resp.\ in \req{coeffX}), and $x_{0}=x_{q}$. We conclude that:
\begin{equation}\label{eq:xisystem}
\text{$x_{i-1}-x_{i}=m_{i}$ for all $i=1,\ldots,q$ and for all choices of $w_{1}\in \mathcal{B}$.}
\end{equation}
Choosing $x_{q}\in \Z$ arbitrarily, the solution of the subsystem of equations obtained by taking $i=2,\ldots,q$ in \req{xisystem} is given by $x_{i-1}=x_{q}+\sum_{j=i}^{q}\, m_{j}$, which determines $x_{l}$ for all $l=1,\ldots,q$. The remaining equation, corresponding to $i=1$, is satisfied, because:
\begin{equation*}
\text{$x_{q}-x_{1}=-\sum_{j=2}^{q}\, m_{j}=m_{1}$ by \req{summi}.} 
\end{equation*}
Hence the system of equations~\reqref{xisystem} possesses solutions for all choices of $w_{1}\in \mathcal{B}$, and so \req{adelta} admits solutions, from which it follows that $A\delta$ is conjugate to $\delta$ by an element of $P_{n}/[P_{n},P_{n}]$. This proves the first part of the statement. The second part is then a direct consequence.
\end{proof}

\begin{rems}\mbox{} \label{rem:abelian}%\comj{This remark added (there was something similar before, but in the proof). There was also a remark saying the same thing at the end of the section, that I have deleted.}
\begin{enumerate}[(a)]
\item The number of conjugacy classes of permutations of order $k$ in $\sn$ is equal to the number of partitions $(n_{1},\ldots,n_{r})$ of $n$, where $n_{i}\in \N$, $n_{1}\leq n_{2}\leq \ldots \leq n_{r}$, $\sum_{i=1}^{r} n_{i}=n$ and $\operatorname{lcm}(n_{1}, \ldots, n_{r})=k$.
	
\item It follows from \reco{torsion} and \reth{classconj} that if $k$ is odd, $\overline{\sigma}$ induces a bijection between the set of conjugacy classes of elements of order $k$ in $B_n/[P_n, P_n]$ and the set of conjugacy classes of elements of order $k$ in $\sn$. The same result also holds for finite cyclic subgroups.
\item\label{it:abelian3} %\comj{This modified:} 
Given an Abelian subgroup $H$ of finite odd order of $\sn$, we saw in \reth{realAbelian} that $B_n/[P_n,P_n]$ contains a subgroup $G$ isomorphic to $H$. An open and more difficult question is whether $B_n/[P_n,P_n]$ contains a subgroup $G$ such that $\overline{\sigma}(G)=H$.
\end{enumerate}
\end{rems}

%\begin{rem} Recall that the number of conjugacy classes of elements of  order $k$   in the 
%symmetric group $\sn$  is given by the number of distinct cyclic structure of elements of $\sn$ 
%of order $k$. 
%\end{rem}

\section{Finite non-Abelian subgroups of $B_n/[P_n, P_n]$}\label{sec:nonAbelian}

As we saw in \reth{2TorCryst} and \relem{eleconj}, any finite subgroup of $B_n/[P_n, P_n]$ is of odd order, and embeds in $\sn$. Following the discussion of the previous sections, it is natural to try to characterise the isomorphism classes of the finite subgroups of  $B_n/[P_n, P_n]$ as well as their conjugacy classes. 
%\comj{This modified:} 
For the question of isomorphism classes, this was achieved for finite Abelian subgroups in \reth{realAbelian}, and for that of conjugacy classes, was solved in \reth{classconj} and \reco{torsion} for cyclic groups. Going a step further, we may also ask whether $B_n/[P_n, P_n]$ possesses finite non-Abelian subgroups. Since any group of order $9$ or $15$ is Abelian, the smallest non-Abelian group of odd order is the \emph{Frobenius group} of order $21$, which we denote by $\mathcal{F}$. It admits the following presentation:
\begin{equation}\label{eq:deffrob}
\mathcal{F}=\setangr{s,t}{\text{$s^3=t^7=1$, $sts^{-1}=t^2$}}.
\end{equation}
%\comj{A few more details added.} 
The group $\mathcal{F}$ is thus a semi-direct product of the form $\Z_{7}\rtimes \Z_{3}$, and it possesses six (resp.\ fourteen) elements of order $7$ (resp.\ of order $3$). As we shall see in \relem{frob}, $\mathcal{F}$ embeds in $\sn[7]$, and as a first step in deciding whether $B_n/[P_n, P_n]$ possesses finite non-Abelian subgroups, one may ask whether $\mathcal{F}$ embeds in $B_7/[P_7, P_7]$. The main result of this section, \reth{frob}, shows that the answer is positive. \reth{oddtorsion}(\ref{it:oddtorsiona}) then implies that $\mathcal{F}$ embeds in $B_n/[P_n, P_n]$ for all $n\geq 7$.  In \reth{uniqueconj}, we show that in $B_7/[P_7, P_7]$, there is a single conjugacy class of subgroups isomorphic to $\mathcal{F}$. The general questions regarding the embedding in $B_n/[P_n, P_n]$ of an arbitrary finite non-Abelian group of odd order (for large enough $n$) and the number of its conjugacy classes remain open.

%It is a natural question try to classify all  possible finite subgroups of the quotient group  
% $B_n/[P_n, P_n]$ and also the conjugacy classes of such finite subgroups. 
%From  the previous sections  follows the  classification of all finite cyclic  subgroups.
%It is not difficult to see that if $H$ is a finite subgroup of $B_n/[P_n, P_n]$
%then $H$ embeds as a subgroup of $\sn$.

%In this section we show that the Frobenius group, a non-Abelian group which embeds as a subgroup of $ \sn[7]$,
%does not embeds in $B_7/[P_7, P_7]$. This group has 
%the minimum cardinality among all finite    non-Abelian groups of odd order, which is 21. 
%
%Recall that a group of order $p^2$, where $p$ is a prime, is Abelian and also a group of order 15 is Abelian.
%Let us denote the Frobenius group   by   $\cal F$ which admits a presentation 
%   $\left\langle x, y \mid x^3=y^7=1, xyx^{-1}=y^2 \right\rangle$. 
%We can view $\cal F$ as a subgroup of the symmetric group $\sn[7]$, see Lemma \ref{lem:frob}.

We first exhibit a subgroup $\mathcal{F}_{0}$ of $\sn[7]$ that is isomorphic to $\mathcal{F}$. We shall see later in \repr{frob} that any subgroup of $\sn[7]$ that is isomorphic to $\mathcal{F}$ is conjugate to $\mathcal{F}_{0}$. In what follows, we consider the following elements of $\sn[7]$: %\comj{This notation defined because we use these elements a lot. I think that this saves space and also makes the reading easier.}
\begin{equation}\label{eq:alphabeta}
\text{$\alpha=(1,3,4,2,5,6,7)$ and $\beta=(1,2,3)(4,5,6)$.}
\end{equation}
Let $\mathcal{F}_{0}$ denote the subgroup of $\sn[7]$ generated by $\brak{\alpha,\beta}$. As noted previously, we read our permutations from left to right, to coincide with our convention for the composition of braids.

\begin{lem}\label{lem:frob}
The subgroup $\mathcal{F}_{0}$ of $\sn[7]$ is isomorphic to $\mathcal{F}$. Further, if $G$ is a subgroup of $\sn[7]$ that is isomorphic to $\mathcal{F}$ then $G$ is generated by two elements $\alpha'$ and $\beta'$, where $\alpha'$ is a $7$-cycle, the cycle type of $\beta'$ is $(3,3,1)$, and $\beta' \alpha'  \beta'^{-1}=\alpha'^2$. 
%$\cal F$.  Also given  a subgroup $G$ of $\sn[7]$ which is isomorphic to $\cal F$ then the group $G$ is generated  by a set 
%of two permutations $\{\alpha, \beta\}$ where  $\alpha$ has cyclic structure  $7$  and $\beta$ has cycle structure $3, 3$.
\end{lem}

\begin{proof}
The first part of the statement is obtained from a straightforward computation using equations~\reqref{deffrob} and~\reqref{alphabeta}. For the second part, if $G$ is a subgroup of $\sn[7]$ that is isomorphic to $\mathcal{F}$ then it possesses a generating set $\brak{\alpha',\beta'}$, where $\alpha'$ is a $7$-cycle, $\beta'$ is of order $3$, and $\beta' \alpha'  \beta'^{-1}=\alpha'^2$. The cycle type of $\beta'$ is either $(3,3,1)$ or $(3,1,1,1,1)$. Suppose that we are in the second case. Then $\beta'=(n_1,n_2,n_3)$ where $n_1,n_{2}$ and $n_{3}$ are distinct elements of $\brak{1,\ldots,7}$. Hence the remaining four elements $m_1, m_2, m_3$ and $m_4$ of $\brak{1,2,3,4,5,6,7} \setminus \brak{n_1,n_2,n_3}$ are fixed by $\beta'$. So there are two consecutive elements of the $7$-cycle $\alpha'$, denoted by $m_j, m_k$, that belong to $\brak{m_1, m_2, m_3, m_4}$. Since $\beta' \alpha'  \beta'^{-1}=\alpha'^2$, we have 
%\comj{the computation is done from left to right} 
$\beta \alpha  \beta^{-1} (m_j)=\alpha \beta^{-1}(m_j)=\beta^{-1}  (m_k)=m_k$, but this is different from $\alpha^2(m_j)=\alpha(m_k)$. This yields a contradiction, and shows that the cycle type of $\beta'$ is $(3,3,1)$.
%Now if you apply the relation $\beta \alpha  \beta^{-1}=\alpha^2$ to the 
% element $m_j$ we get  $\beta \alpha  \beta^{-1} (m_j)=\beta \alpha (m_j)=\beta  (m_k)=m_k$ which is 
% different from $\alpha^2(m_j)=\alpha(m_k)\ne m_k$. So we get a contradiction and this shows that 
% the cyclic structure of $\beta$ is of the form $3,3$.
\end{proof}

%\comdo{We modify this paragraph.}
%The two results above should be useful to study the conjugacy classes of subgroups isomorphic to the Frobenius group.
The rest of this section is devoted to proving that $\mathcal{F}$ embeds in $B_7/[P_7, P_7]$ and to showing that in $B_7/[P_7, P_7]$, there is a single conjugacy class of subgroups isomorphic to $\mathcal{F}$. In this quotient, we define:
\begin{equation}\label{eq:defxy}
\text{$x=\sigma_2\sigma_1^{-1}\sigma_5\sigma_{4}^{-1}$ and $y=\sigma_2\sigma_3 \sigma_6 \sigma_5\sigma_{4}\sigma_3^{-1} \sigma_2^{-1}\sigma_1^{-1} \sigma_3^{-1} \sigma_2^{-1}$.}
\end{equation}
Then $\overline{\sigma}(y)=\alpha$ and using the notation of \req{defdelta}, $y=\sigma_2\sigma_3\delta_{0,7} \sigma_3^{-1}\sigma_2^{-1}$. So $y$ is of order $7$ by \relem{0}. Similarly, $\overline{\sigma}(x)=\beta$, $x=\delta_{0,3}\delta_{3,3}$, and $x$ is of order $3$ (see the discussion on page~\pageref{eq:delta} just after \req{delta}). We now prove \reth{frob} that asserts the existence of a subgroup of $B_7/[P_7, P_7]$ isomorphic to the Frobenius group $\mathcal{F}$.

%\begin{thm}\label{th:frob}
%The quotient group $B_7/[P_7, P_7]$ possesses a subgroup isomorphic to the Frobenius group $\mathcal{F}$.
%\end{thm}

\begin{proof}[Proof of \reth{frob}]
%\comj{This has been rewritten taking into account \relem{frob1}.}
%\comdo{The proof was entirely modified.}
Consider the subgroup $H$ of $B_7/[P_7, P_7]$ generated by $\brak{x,y}$. By the above comments, we know that $\overline{\sigma}(y)=\alpha$ and $\overline{\sigma}(x)=\beta$, therefore $\overline{\sigma}(H)=\mathcal{F}_0$. 
Let
\begin{equation*}
\gamma=\sigma_2^{-1}\sigma_4^{-1} \sigma_5\sigma_4\sigma_6^{-1} \sigma_5^{-1} \sigma_1\sigma_2^{-1}.
\end{equation*}
Using the Artin relations~\reqref{artin1} and~\reqref{artin2}, and \req{conjugAij2}, we have:
\begin{align*}
\gamma xyx^{-1}y^{-2} \gamma^{-1}=&\sigma_2^{-1}\sigma_4^{-1} \sigma_5\sigma_4\sigma_6^{-1} \sigma_5^{-1} \sigma_1\sigma_2^{-1}\ldotp  \sigma_2\sigma_1^{-1}\sigma_5\sigma_{4}^{-1}\ldotp  \sigma_2\sigma_3 \sigma_6 \sigma_5\sigma_{4}\sigma_3^{-1} \sigma_2^{-1} \sigma_1^{-1} \sigma_3^{-1} \sigma_2^{-1}\ldotp\\
&\sigma_{4}\sigma_5^{-1}\sigma_1\sigma_2^{-1} \ldotp \sigma_2\sigma_3 \sigma_1 \sigma_2\sigma_{3}\sigma_4^{-1} \sigma_5^{-1}\sigma_6^{-1} \sigma_1 \sigma_2\sigma_{3}\sigma_4^{-1} \sigma_5^{-1}\sigma_6^{-1} \sigma_3^{-1}\sigma_2^{-1} \ldotp\\
& \sigma_2 \sigma_1^{-1} \sigma_5\sigma_6 \sigma_{4}^{-1}\sigma_5^{-1}\sigma_4 \sigma_2\\
=& \sigma_4^{-1} \sigma_5^{2} \sigma_4^{-1} \sigma_3  \sigma_{4} \sigma_2^{-1} \sigma_1^{-1} \sigma_3^{-1} \sigma_2^{-1} \sigma_{4}\sigma_5^{-1}\sigma_1^{2} \sigma_2 \sigma_3\sigma_{2}\sigma_4^{-1} \sigma_5^{-1} \sigma_1 \sigma_2\sigma_{3} \sigma_4^{-1} \sigma_5^{-1}\ldotp\\
&\sigma_6^{-1} \sigma_5^{-1} \sigma_3^{-1} \sigma_1^{-1} \sigma_5\sigma_6 \sigma_5\sigma_{4}^{-1}\sigma_5^{-1} \sigma_2\\
=& A_{4,6} A_{4,5}^{-1}\ldotp A_{4,5}\ldotp \sigma_3  \sigma_{4} \sigma_2^{-1} \sigma_1^{-1} \sigma_{4} \sigma_3 \sigma_4^{-1} \sigma_5^{-1}  \sigma_{2} \sigma_4^{-1} \sigma_5^{-1} \sigma_1 \sigma_2 \sigma_4^{-1} \sigma_3^{-1} \sigma_{4}\sigma_1^{-1} \ldotp\\
& \sigma_{4}^{-1}\sigma_5^{-1} \sigma_2\\
=& A_{4,6} \ldotp \sigma_3 \sigma_{4}^{2} \sigma_2^{-1} \sigma_1^{-1} \sigma_3 \sigma_4^{-1} \sigma_5^{-1}  \sigma_{2} \sigma_4^{-1} \sigma_5^{-1} \sigma_1 \sigma_2 \sigma_4^{-1} \sigma_3^{-1} \sigma_1^{-1}  \sigma_5^{-1} \sigma_2\\
=& A_{4,6} A_{3,5}\ldotp \sigma_3 \sigma_2^{-1} \sigma_1^{-1} \sigma_3 \sigma_4^{-2} \sigma_5^{-1} \sigma_4^{-1} \sigma_{1}  \sigma_2 \sigma_1 \sigma_4^{-1} \sigma_3^{-1} \sigma_1^{-1}  \sigma_5^{-1} \sigma_2\\
=& A_{4,6} A_{3,5} A_{2,5}^{-1} \ldotp \sigma_3 \sigma_2^{-1} \sigma_3 \sigma_5^{-1} \sigma_4^{-2} \sigma_2 \sigma_3^{-1} \sigma_5^{-1} \sigma_2\\
=& A_{4,6} A_{3,5} A_{2,5}^{-1} A_{2,6}^{-1} \ldotp \sigma_3 \sigma_2^{-1} \sigma_3 \sigma_5^{-2} \sigma_2 \sigma_3^{-1} \sigma_2\\
=& A_{4,6} A_{3,5} A_{2,5}^{-1} A_{2,6}^{-1} A_{5,6}^{-1} \ldotp \sigma_3 \sigma_2^{-1} \sigma_3 \sigma_2 \sigma_3^{-1} \sigma_2\\
=& A_{4,6} A_{3,5} A_{2,5}^{-1} A_{2,6}^{-1} A_{5,6}^{-1} A_{3,4} A_{2,3} A_{2,4}^{-1}.
\end{align*}
Using \req{conjugAij2} once more, we obtain:
\begin{align}\label{eq:frob1}
xyx^{-1}y^{-2} &= \gamma^{-1} A_{4,6} A_{3,5} A_{2,5}^{-1} A_{2,6}^{-1} A_{5,6}^{-1} A_{3,4} A_{2,3} A_{2,4}^{-1} \gamma\notag\\
&= A_{4,7} A_{1,7} A_{1,6}A_{2,7}^{-1}A_{2,6}^{-1}A_{2,4}^{-1}A_{1,2}A_{4,6}^{-1}.
\end{align}
Note that this shows that $xyx^{-1}y^{-2}$ is non trivial in the free Abelian group $P_7/[P_7,P_7]$, which implies that $H$ is not isomorphic to $\mathcal{F}$. We now look for an element $N\in P_7/[P_7, P_7]$ such that if $v=Ny$ then the subgroup $\ang{x,v}$ is isomorphic to $\mathcal{F}$, where $v$ is of order $7$, $\overline{\sigma}(v)=\alpha$ and $xvx^{-1}=v^{2}$. This last equality gives rise to the following equivalences:
\begin{equation}\label{eq:xvx}
xvx^{-1}=v^{2} \Longleftrightarrow xNy x^{-1}=Ny Ny \Longleftrightarrow xyx^{-1}y^{-2}= xN^{-1}x^{-1} \ldotp N\ldotp yNy^{-1}.
\end{equation}
We seek solutions $N\in P_7/[P_7, P_7]$ of \req{xvx} taking into account \req{frob1} and the fact that $Ny$ is of order $7$. In order to do so, we use additive notation, and we write $N$ in terms of the basis $\brak{A_{i,j}}_{1\leq i<j\leq 7}$ of $P_7/[P_7, P_7]$ as follows: 
%\comj{This notation is more in the spirit of \repr{Torcaonimpar}}
\begin{equation}\label{eq:defN}
N=\sum_{1\leq i<j\leq 7} p_{i,j} A_{i,j},
\end{equation}
where $p_{i,j}\in \Z$ for all $1\leq i<j\leq 7$. 
%\comj{All of what follows may be made a little more elegant by choosing a basis (in a similar manner to that of \repr{Torcaonimpar}) that is partitioned so as to respect the $y$-orbits (in a similar spirit to that of the original proof). One nice property that seems to hold here is that conjugation by $x$ sends each of the three $y$-orbits to another $y$-orbit, and this enables us to write a general formula in this case, without having to write down the specific values of the indices. However, if it cannot be generalised for other purposes, I don't know whether it is worth pursuing this.}
% 
%\comdo{Since we did not see quickly how to choose a basis in a simmilar matter to that of \repr{Torcaonimpar} we share the opinion whether is worth pursuing this.}
By \req{conjugAij2} and \repr{i}, we see that the orbits under the action of conjugation by $y$ are of the form:
\begin{equation}\label{eq:frob4}
\begin{cases}
A_{1,2} \mapsto A_{4,7} \mapsto A_{3,6} \mapsto A_{1,5} \mapsto A_{2,7} \mapsto A_{4,6} \mapsto A_{3,5} \mapsto A_{1,2}\\
A_{1,3} \mapsto A_{1,7} \mapsto A_{6,7} \mapsto A_{5,6} \mapsto A_{2,5} \mapsto A_{2,4} \mapsto A_{3,4} \mapsto A_{1,3}\\
A_{1,4} \mapsto A_{3,7} \mapsto A_{1,6} \mapsto A_{5,7} \mapsto A_{2,6} \mapsto A_{4,5} \mapsto A_{2,3} \mapsto A_{1,4},
\end{cases}
\end{equation}
and the orbits under the action of conjugation by $x$ are of the form:
\begin{equation}\label{eq:frob3}
\left\{\begin{aligned}
& A_{1,2} \mapsto A_{1,3} \mapsto A_{2,3} \mapsto A_{1,2} &&
A_{4,6} \mapsto A_{5,6} \mapsto A_{4,5} \mapsto A_{4,6}\\
& A_{2,7} \mapsto A_{1,7} \mapsto A_{3,7} \mapsto A_{2,7} &&
A_{4,7} \mapsto A_{6,7} \mapsto A_{5,7} \mapsto A_{4,7}\\
& A_{3,6} \mapsto A_{2,5} \mapsto A_{1,4} \mapsto A_{3,6}&&
A_{1,5} \mapsto A_{3,4} \mapsto A_{2,6} \mapsto A_{1,5}\\
& A_{3,5} \mapsto A_{2,4} \mapsto A_{1,6} \mapsto A_{3,5}.&&
\end{aligned}\right.
\end{equation}
The first (resp.\ second) line of \reqref{frob3} may be obtained by applying \repr{i} (resp.\ \repr{ii}) and \repr{iv}, and the last two lines follow from \repr{iii}. Arguing in a manner similar to that of the proof of \repr{Torcaonimpar}, and using the fact that $y$ is of order $7$, we obtain:
\begin{equation*}
0 =(Ny)^{7}=\sum_{k=0}^{6} y^{k}Ny^{-k}=\sum_{k=0}^{6} y^{k}\left(\sum_{1\leq i<j\leq 7} p_{i,j} A_{i,j}\right)y^{-k}=\sum_{1\leq i<j\leq 7} p_{i,j} \left( \sum_{k=0}^{6} y^{k}A_{i,j} y^{-k} \right),
\end{equation*}
from which it follows that the sum of the coefficients corresponding to the elements of each of the orbits given in~\reqref{frob4} is zero:
\begin{equation}\label{eq:sumxyz}
\left\{\begin{aligned}
p_{1,2}+p_{4,7}+p_{3,6}+p_{1,5}+p_{2,7}+p_{4,6}+p_{3,5}&=0\\
p_{1,3}+p_{1,7}+p_{6,7}+p_{5,6}+p_{2,5}+p_{2,4}+p_{3,4}&=0\\
p_{1,4}+p_{3,7}+p_{1,6}+p_{5,7}+p_{2,6}+p_{4,5}+p_{2,3}&=0.
\end{aligned}\right.
\end{equation}
Using equations~\reqref{frob4} and~\reqref{frob3} to compute first $xNx^{-1}$ and $yNy^{-1}$ and then equations~\reqref{frob1}, \reqref{xvx} and~\reqref{defN}, we obtain the following systems of equations:
\begin{equation}\label{eq:frob5}
\left\{
\begin{aligned}
A_{1,2}\colon & & -p_{2,3}+p_{1,2}+p_{3,5} &= 1 &
A_{1,3}\colon & & -p_{1,2}+p_{1,3}+p_{3,4} &= 0\\
A_{4,7}\colon & & -p_{5,7}+ p_{4,7} +p_{1,2} &= 1 &
A_{1,7}\colon & & -p_{2,7}+ p_{1,7} +p_{1,3} &= 1\\
A_{3,6}\colon & & -p_{1,4} +p_{3,6}+p_{4,7} &= 0 &
A_{6,7}\colon & & -p_{4,7} +p_{6,7}+p_{1,7} &= 0\\
A_{1,5}\colon & & -p_{2,6}+p_{1,5}+p_{3,6} &= 0 &
A_{5,6}\colon & & -p_{4,6}+p_{5,6}+p_{6,7} &= 0\\
A_{2,7}\colon & & -p_{3,7}+p_{2,7}+ p_{1,5} & = -1 &
A_{2,5}\colon & & -p_{3,6}+p_{2,5}+ p_{5,6} & = 0\\
A_{4,6}\colon & & -p_{4,5}+p_{4,6}+ p_{2,7} &= -1 &
A_{2,4}\colon & & -p_{3,5}+p_{2,4}+ p_{2,5} &= -1\\
A_{3,5}\colon & & -p_{1,6}+p_{3,5}+ p_{4,6} &= 0 &
A_{3,4}\colon & & -p_{1,5}+p_{3,4}+ p_{2,4} &= 0\\
A_{1,4}\colon & & -p_{2,5}+p_{1,4}+p_{2,3} &= 0 &
A_{2,6}\colon & & -p_{3,4}+p_{2,6}+ p_{5,7} & = -1\\
A_{3,7}\colon & & -p_{1,7}+ p_{3,7} +p_{1,4} &= 0 &
A_{4,5}\colon & & -p_{5,6}+p_{4,5}+ p_{2,6} &= 0\\
A_{1,6}\colon & & -p_{2,4} +p_{1,6}+p_{3,7} &= 1 &
A_{2,3}\colon & & -p_{1,3}+p_{2,3}+ p_{4,5} &= 0\\
A_{5,7}\colon & & -p_{6,7}+p_{5,7}+p_{1,6} &= 0. & & & &
\end{aligned}\right.
\end{equation}
One may check that the systems~\req{sumxyz} and~\reqref{frob5} together admit a solution, taking for example all of the coefficients to be zero, with the exception of:
\begin{equation*}
\text{$p_{2,7}=p_{5,7}=-1$ and $p_{3,5}=p_{1,6}=1$.}
\end{equation*}
%
%Together with the system of , these systems admit a solution, by taking for example:
%\begin{equation*}
%p_{1,3}=p_{1,7}=p_{6,7}=p_{5,6}=p_{2,5}=p_{2,4}=p_{3,4}=0.
%\end{equation*}
%%\comdo{Note that these elements corresponds to the midle column in the second system of equations.}
%All of the other coefficients are then seen to be zero, with the exception of: 
For these values of $p_{i,j}$, we have $N=A_{3,5}+A_{1,6}-A_{2,7}-A_{5,7}$, and it follows from above that the subgroup $\ang{x,v}$ of $B_7/[P_7,P_7]$ is isomorphic to $\mathcal{F}$, which completes the proof of the theorem.
\end{proof}

We now analyse the conjugacy classes of subgroups isomorphic to $\mathcal{F}$ in $B_{7}/[P_{7},P_{7}]$. We first show that $\sn[7]$ possesses a single such conjugacy class. 

\begin{prop}\label{prop:frob}
Any two subgroups of $\sn[7]$ isomorphic to $\mathcal{F}$ are conjugate.
\end{prop}

\begin{proof}
Let $G$ be a subgroup of $\sn[7]$ isomorphic to $\mathcal{F}$. It suffices to show that $G$ is conjugate to $\mathcal{F}_{0}$. By \relem{frob}, $G$ is generated by two elements $\alpha'$ and $\beta'$, where $\alpha'$ is a $7$-cycle, the cycle type of $\beta'$ is $(3,3,1)$, and $\beta' \alpha  \beta'^{-1}=\alpha^2$. Conjugating $G$ if necessary, we may suppose that $\alpha'=\alpha$.
%the Frobenius group is conjugated to the 
%subgroup generated by   $(1,2,3)(4,5,6)$ and   $(1,3,4,2,5,6,7)$. 
%Let us consider a set of generators $\alpha ; \beta$ of $G$ which
%have cyclic  decomposition of the form $7$;  $3,3$, respectively. Since two elements which have 
%cyclic decomposition $7$ are conjugated we can assume without loss of generality that $G$ admits a set of 
%generators  of the form $(1,3,4,2,5,6,7)  , \beta$, where $\beta$ has cyclic structure $3,3$. 
Now $\beta' \alpha  \beta'^{-1}=\alpha^2$ in $G$ and $\beta\alpha  \beta^{-1}=\alpha^2$ in $\mathcal{F}_{0}$, from which it follows that $\beta^{-1}\beta'$ belongs to the centraliser of $\alpha$. But since $\alpha$ is a complete cycle in $\sn[7]$, its centraliser is equal to $\ang{\alpha}$. So there exists $k\in \brak{0,1,\ldots,6}$ such that $\beta'=\beta \alpha^{k}$, and hence $G=\ang{\alpha,\beta'}=\ang{\alpha,\beta}=\mathcal{F}_{0}$ as required.
%
% Because  $\beta  (1,3,4,2,5,6,7)\beta^{-1}=  (1,3,4,2,5,6,7)^2$,
% which in turn is 
%\begin{equation*}
%(1,2,3)(4,5,6)  (1,3,4,2,5,6,7)((1,2,3)(4,5,6))^{-1},
%\end{equation*}
%it follows that  
% $\beta ((1,2,3)(4,5,6) )^{-1}$ belong to the centraliser of the element  
%$(1,3,4,2,5,6,7)$.
%
%Since this is a full cycle, it is known that the centraliser of this element is the subgroup 
%generated by  $(1,3,4,2,5,6,7)$. So  $\beta=(1,2,3)(4,5,6)(1,3,4,2,5,6,7)^k$. Therefore the 
%subgroup $G=\cal F$ and the result follows.
\end{proof}

%The element $x$ projects to the permutation   $(1,2,3)(4,5,6)$ and it is equal to 
%$\delta_{0,3}\delta_{3,3}$ where the $\delta$'s  are given by Lemma \ref{lem:0}. So $x$ has order $3$ since 
%$\delta_{0,3}$ and $\delta_{3,3}$ commute and each of them has order $3$ by Lemma \ref{lem:0}. 
%Similarly the element $y$  projects to the permutation  $(1,3,4,2,5,6,7)$ and it 
%it is equal to $\sigma_2\sigma_3\delta_{0,7}\sigma_3^{-1}\sigma_2^{-1}$ where by the
%Lemma \ref{lem:0} $\delta_{0,7}$ has order $7$. So $y$ has order $7$.
%
%Let $\mathcal{F}_0$ denote the subgroup of $\sn[7]$ generated by the  
%permutations   $(1,2,3)(4,5,6)$ and   $(1,3,4,2,5,6,7)$, which from 
%the proof of Lemma \ref{lem:frob} is a Frobenius group. 

\begin{rem}\label{rem:conjalpha}
%\comj{This remark has been modified. The reason is that there was a small error in \repr{frob1} where it was written that `the centraliser of $\beta$ in $\sn[7]$, which is equal to $\ang{(1,2,3),(4,5,6)}$', but this is not true: the centraliser contains a square root of $\beta$, and is of order $18$, not $9$.}
For the purposes of the proof of \repr{frob1}, we shall study the elements of the form $\epsilon \alpha \epsilon^{-1}$, where $\epsilon$ belongs to the centraliser of $\beta$ in $\sn[7]$. This centraliser may be seen to be of order $18$, and consists of the elements of the form $\tau^{i}(123)^{j}$, where $\tau=(1,4,2,5,3,6)$, $0\leq i\leq 5$ and $0\leq j\leq 2$. Let: 
\begin{equation}\label{eq:alpha12}
\begin{cases}
\alpha_{1}&=(1,3,2) \alpha (1,3,2)^{-1}=(1,4,3,5,6,7,2), \quad\text{and}\\
\alpha_{2}&=(4,6,5) \alpha (4,6,5)^{-1} =(1,3,5,2,6,4,7).
\end{cases}
\end{equation}
A straightforward computation shows that:
\begin{equation*}
(1,2,3)^{j} \alpha (1,2,3)^{-j}= \begin{cases}
\alpha & \text{if $j=0$}\\
\alpha_{2}^{2} & \text{if $j=1$}\\
\alpha_{1} & \text{if $j=2$,}
\end{cases}
\end{equation*}
and $\tau \alpha \tau^{-1}=\alpha_{2}^{-2}$, $\tau \alpha_{2} \tau^{-1}=\alpha^{-1}$ and $\tau \alpha_{1} \tau^{-1}=\alpha_{1}^{3}$ for all $0\leq i\leq 5$ and $0\leq j\leq 2$. It then follows that for all $0\leq i\leq 5$ and $0\leq j\leq 2$, there exists $z\in \brak{\alpha,\alpha_{1}, \alpha_{2}}$ such that $\tau^{i}(123)^{j} \alpha (123)^{-j}\tau^{-i}$ is a generator of $\ang{z}$.
%
%In what follows, we shall consider the nine elements of the form $\gamma \alpha \gamma^{-1}$, where $\gamma\in \ang{(1,2,3),(4,5,6)}$. Set:
%
%Besides $\alpha$, $\alpha_{1}$ and $\alpha_{2}$, a computation shows that these elements are given by:
%\begin{align*}
%\alpha^{2}&=(1,2,3)(4,5,6) \alpha ((1,2,3)(4,5,6))^{-1} & \alpha^{4}&=(1,3,2)(4,6,5) \alpha ((1,3,2)(4,6,5))^{-1}\\
%\alpha_{1}^{2}&=(4,5,6) \alpha (4,5,6)^{-1} & \alpha_{1}^{4}&=(1,2,3)(4,6,5) \alpha ((1,2,3)(4,6,5))^{-1}\\
%\alpha_{2}^{2}&=(1,2,3) \alpha (1,2,3)^{-1} & \alpha_{2}^{4}&=(1,3,2)(4,5,6) \alpha ((1,3,2)(4,5,6))^{-1}.
%\end{align*}
%In particular, each conjugate $\gamma \alpha \gamma^{-1}$ is of the form $z^{m}$, where $z\in \brak{\alpha,\alpha_{1},\alpha_{2}}$ and $m\in \brak{1,2,4}$.
%Further, if 
%\begin{align*}
%y_1 &=\sigma_1\sigma_2^{-1} y \sigma_2\sigma_1^{-1} = \sigma_1\sigma_2^{-1}\sigma_2\sigma_3\sigma_6\sigma_5\sigma_{4}\sigma_3^{-1}\sigma_2^{-1}\sigma_1^{-1}\sigma_3^{-1}\sigma_2^{-1}\sigma_2\sigma_1^{-1} \quad \text{and}\\
%y_2 &=\sigma_{4}\sigma_5^{-1} y \sigma_5\sigma_{4}^{-1} = \sigma_{4}\sigma_5^{-1} \sigma_2\sigma_3\sigma_6\sigma_5\sigma_{4}\sigma_3^{-1}\sigma_2^{-1}\sigma_1^{-1}\sigma_3^{-1}\sigma_2^{-1}\sigma_5\sigma_{4}^{-1}
%\end{align*}
%then $\overline{\sigma}(y_{i})=\alpha_{i}$ for $i\in \brak{1,2}$.
\end{rem}

\begin{prop}\label{prop:frob1}
Suppose that $H$ is a subgroup of $B_7/[P_7, P_7]$ isomorphic to $\mathcal{F}$. Then $H$ is conjugate to a subgroup of the form $\ang{x,v}$, where $x$ is given by \req{defxy} and $\overline{\sigma}(v)=\alpha$.
%If there is a subgroup of    $B_7/[P_7, P_7]$  isomorphic to the Frobenius group
%then there is a subgroup isomorphic to the Frobenius group generated by two elements $u,v $ such that we can assume that $u=x$ and $v$ projects to one of the permutations
%$$
%(1,3,4,2,5,6,7),\textrm{ } (1,4,3,5,6,7,2) \textrm{ or } (1,3,5,2,6,4,7).
%$$ 
\end{prop}

\begin{proof}  
Let $H$ be a subgroup of $B_7/[P_7, P_7]$ isomorphic to $\mathcal{F}$. Since $\overline{\sigma}(H)$ is a subgroup of $\sn[7]$ isomorphic to $\mathcal{F}$ by \relem{eleconj}, it follows from \repr{frob} that there exists $\rho\in \sn[7]$ such that $\mathcal{F}_0=\rho \overline{\sigma}(H)\rho^{-1}$. So if $\widehat{\rho}\in B_7/[P_7, P_7]$ is such that $\overline{\sigma}(\widehat{\rho})=\rho$ then $H_1=\widehat{\rho} H \widehat{\rho}^{-1}$ satisfies $\overline{\sigma}(H_{1})=\mathcal{F}_0$. Let $\widetilde{x},\widetilde{y} \in H_{1}$ be such that $\overline{\sigma}(\widetilde{x})=\beta$ and $\overline{\sigma}(\widetilde{y})=\alpha$, where $\alpha$ and $\beta$ are given by \req{alphabeta}. Now $\beta=\overline{\sigma}(x)$, and since $x$ and $\widetilde{x}$ are of order $3$ and have the same permutation, \reth{classconj} implies that they are conjugate. So there exists $\lambda_{1} \in B_7/[P_7, P_7]$ such that $\lambda_{1} \widetilde{x} \lambda_{1}^{-1}=x$. Hence $\overline{\sigma}(\lambda_{1}) \overline{\sigma}(\widetilde{x}) \overline{\sigma}(\lambda_{1})^{-1}=\overline{\sigma}(x)$, and since $\overline{\sigma}(\widetilde{x})=\overline{\sigma}(x)=\beta$, we conclude that $\overline{\sigma}(\lambda_{1})$ belongs to the centraliser of $\beta$ in $\sn[7]$. 
%\comj{the following has been rewritten a little, since there was a mistake in the centraliser of $\boldmath{\beta}$} 
By \rerem{conjalpha}, this centraliser is equal to $\ang{\tau,(1,2,3)}$, and the fact that $\overline{\sigma}(\widetilde{y})=\alpha$ implies that there exists $z\in \brak{\alpha,\alpha_{1}, \alpha_{2}}$ such that $\overline{\sigma}(\lambda_{1}\widetilde{y} \lambda_{1}^{-1})$ is a generator of $\ang{z}$.
%which as we saw in \rerem{conjalpha}, is equal to $\ang{\tau,(1,2,3)}$. Now $\overline{\sigma}(\widetilde{y})=\alpha$, and it follows from that remark that there exists $z\in \brak{\alpha,\alpha_{1}, \alpha_{2}}$ such that $\overline{\sigma}(\lambda_{1}\widetilde{y} \lambda_{1}^{-1})$ is a generator of $\ang{z}$.
%$\overline{\sigma}(\lambda_{1}\widetilde{y} \lambda_{1}^{-1})=z^{m}$.
%$z\in \brak{\alpha,\alpha_{1}, \alpha_{2}}$ such that $\tau^{i}(123)^{j} \alpha (123)^{-j}\tau^{-i}\in \ang{z}\setminus \brak{\id}$.
%\rerem{conjalpha} and the fact that $\overline{\sigma}(\widetilde{y})=\alpha$ that there exist $z\in \brak{\alpha,\alpha_{1},\alpha_{2}}$ and $m\in \brak{1,2,4}$ such that $\overline{\sigma}(\lambda_{1}\widetilde{y} \lambda_{1}^{-1})=z^{m}$. 
Let:
\begin{equation*}
\lambda_{2}=\begin{cases}
e & \text{if $z=\alpha$}\\
\sigma_{1}\sigma_{2}^{-1} & \text{if $z=\alpha_{1}$}\\
\sigma_{4}\sigma_{5}^{-1} & \text{if $z=\alpha_{2}$.}
\end{cases}
\end{equation*}
Note that $\lambda_{2}$ commutes with $x$, and by \req{alpha12}, $\overline{\sigma}(\lambda_{2}^{-1} \lambda_{1}\widetilde{y} \lambda_{1}^{-1}\lambda_{2})$ is a generator of $\ang{\alpha}$. Taking $v$ to be the element of $\lambda_{2}^{-1}\lambda_{1}\ang{\widetilde{y}} \lambda_{1}^{-1}\lambda_{2}$ for which $\overline{\sigma}(v)=\alpha$, the subgroup $\lambda_{2}^{-1}\lambda_{1}\widehat{\rho} H (\lambda_{2}^{-1}\lambda_{1}\widehat{\rho})^{-1}$ is then seen to be equal to $\ang{x,v}$, which proves the proposition.
\end{proof}

\begin{thm}\label{th:uniqueconj}
The group $B_{7}/[P_{7},P_{7}]$ possesses a unique conjugacy class of subgroups isomorphic to $\mathcal{F}$.
\end{thm}

\begin{proof}
From the proof of \reth{frob}, $B_{7}/[P_{7},P_{7}]$ possesses a subgroup $H_{0}=\ang{x,v_{0}}$ isomorphic to $\mathcal{F}$, where $v_{0}=N_{0}y$, and $N_{0}=A_{1,6} A_{3,5} A_{2,7}^{-1} A_{5,7}^{-1}$. Let $H$ be a subgroup of $B_{7}/[P_{7},P_{7}]$ isomorphic to $\mathcal{F}$. By \repr{frob1}, up to conjugacy, we may suppose that $H=\ang{x,v}$, where $\overline{\sigma}(v)=\alpha =\overline{\sigma}(y)= \overline{\sigma}(v_{0})$. Thus $v=Ny$, where $N\in P_{7}/[P_{7},P_{7}]$. Again from the proof of \reth{frob}, the coefficients $p_{i,j}$ of $N$ given by \req{defN} satisfy the systems of equations~\reqref{sumxyz} and~\reqref{frob5}, and one may check that the general solution of these two systems is of rank $6$, and is given by:
\begin{equation}\label{eq:solutionN}
\left\{
\begin{aligned}
p_{1,2} &= -r_{6} +r_{4}+r_{3}-r_{2}+1 & 
p_{1,3} &= -r_{6}-r_{2}\\
p_{4,7} &= r_{6}-r_{3}+r_{2}&
p_{1,7} &= r_{6} -r_{5}-r_{4} -r_{3}+r_{2}\\
p_{3,6} &= r_{3}-r_{2}&
p_{6,7} &= r_{5}+r_{4}\\
p_{1,5} &= r_{2}&
p_{5,6} &= -r_{6} +r_{3}-r_{2}-r_{1}\\
p_{2,7} &= -r_{5}-r_{4}-r_{3}-1&
p_{2,5} &= r_{6}+r_{1}\\
p_{4,6} &= -r_{6} +r_{5} +r_{4} +r_{3}-r_{2}-r_{1}&
p_{2,4} &= -r_{4}-r_{3}+r_{2}-1\\
p_{3,5} &= r_{6} -r_{4} -r_{3}+r_{2}+r_{1}&
p_{3,4} &= r_{4}+r_{3}+1\\
p_{1,4} &= r_{6}& 
p_{2,6} &= r_{3}\\
p_{3,7} &= -r_{5} -r_{4} -r_{3}+r_{2}& 
p_{4,5} &= -r_{6}-r_{2}-r_{1}\\
p_{1,6} &= r_{5}& p_{2,3} &= r_{1}\\
p_{5,7} &= r_{4},& &
\end{aligned}\right.
\end{equation}
where $r_{1},\ldots,r_{6}\in \Z$ are arbitrary. So choose the values of the $r_{l}$ so that $v=Ny$. We claim that there exists $\Theta\in P_{7}/[P_{7},P_{7}]$ such that:
\begin{align}
& \Theta x\Theta^{-1}=x, \quad \text{and}\label{eq:conjfrob1}\\
& \Theta v_{0} \Theta^{-1}=v.\label{eq:conjfrob2}
\end{align}
This being the case, we have $H=\ang{x,v} = \Theta \ang{x,v_{0}} \Theta^{-1}=\Theta H_{0} \Theta^{-1}$, in particular $H$ and $H_{0}$ are conjugate in $B_{7}/[P_{7},P_{7}]$, which proves the statement of the theorem. To prove the claim, let $\displaystyle \Theta=\sum_{1\leq i<j\leq 7} \theta_{i,j}A_{i,j}$. We must determine the coefficients $\theta_{i,j}$ of $\Theta$ that satisfy equations~\reqref{conjfrob1} and~\reqref{conjfrob2}. By \req{frob3}, \req{conjfrob1} holds if and only if there exist $s_{1},\ldots,s_{7}\in \Z$ such that:
\begin{equation}\label{eq:definesk}
\left\{
\begin{aligned}
s_{1}&=\theta_{1,2}=\theta_{1,3}=\theta_{2,3} & 
s_{2}&= \theta_{2,7}=\theta_{1,7}=\theta_{3,7}\\
s_{3}&=\theta_{3,6}=\theta_{2,5}= \theta_{1,4} &
s_{4}&=\theta_{3,5}=\theta_{2,4}=\theta_{1,6}\\
s_{5}&=\theta_{4,6}=\theta_{5,6}=\theta_{4,5} &
s_{6}&=\theta_{4,7}=\theta_{6,7}=\theta_{5,7}\\
s_{7}&=\theta_{1,5}=\theta_{3,4}=\theta_{2,6}.& &
\end{aligned}\right.
\end{equation}
Equation~\reqref{conjfrob2} may be written in the form $\Theta \ldotp N_{0} \ldotp y \Theta^{-1} y^{-1}=N$. Using \req{frob3}, we obtain the following system of equations:
\begin{equation}\label{eq:syspthetas}
\left\{
\begin{aligned}
p_{1,2} &= \theta_{1,2}-\theta_{3,5}=s_{1}-s_{4}&
p_{1,3} &= \theta_{1,3}-\theta_{2,3}=s_{1}-s_{7}\\
p_{4,7} &= \theta_{4,7}-\theta_{1,2}=s_{6}-s_{1}&
p_{1,7} &= \theta_{1,7}-\theta_{1,3}=s_{2}-s_{1}\\
p_{3,6} &= \theta_{3,6}-\theta_{4,7}=s_{3}-s_{6}&
p_{6,7} &= \theta_{6,7}-\theta_{1,7}=s_{6}-s_{2}\\
p_{1,5} &= \theta_{1,5}-\theta_{3,6}=s_{7}-s_{3}&
p_{5,6} &= \theta_{5,6}-\theta_{6,7}=s_{5}-s_{6}\\
p_{2,7} &= \theta_{2,7}-\theta_{1,5}-1=s_{2}-s_{7}-1&
p_{2,5} &= \theta_{2,5}-\theta_{5,6}=s_{3}-s_{5}\\
p_{4,6} &= \theta_{4,6}-\theta_{2,7}=s_{5}-s_{2}&
p_{2,4} &= \theta_{2,4}-\theta_{2,5}=s_{4}-s_{3}\\
p_{3,5} &= \theta_{3,5}-\theta_{4,6}+1=s_{4}-s_{5}+1&
p_{3,4} &= \theta_{3,4}-\theta_{2,4}=s_{7}-s_{4}\\
p_{1,4} &= \theta_{1,4}-\theta_{3,5}=s_{3}-s_{1}&
p_{2,6} &= \theta_{2,6}-\theta_{5,7}=s_{7}-s_{6}\\
p_{3,7} &= \theta_{3,7}-\theta_{1,4}=s_{2}-s_{3}&
p_{4,5} &= \theta_{4,5}-\theta_{2,6}=s_{5}-s_{7}\\
p_{1,6} &= \theta_{1,6}-\theta_{3,7}+1=s_{4}-s_{2}+1&
p_{2,3} &= \theta_{2,3}-\theta_{4,5}=s_{1}-s_{5}\\
p_{5,7} &= \theta_{5,7}-\theta_{1,6}-1=s_{6}-s_{4}-1. &&
\end{aligned}\right.
\end{equation}
It remains to show that by choosing the $s_{k}$ appropriately, we obtain a system of coefficients that satisfy the equations of system~\reqref{syspthetas}. Consider the system:
\begin{equation}\label{eq:basicS}
\left\{\begin{aligned}
s_{1}-s_{4}&= -r_{6} +r_{4}+r_{3}-r_{2}+1\\
s_{6}-s_{1}&=r_{6}-r_{3}+r_{2}\\
s_{3}-s_{6}&=r_{3}-r_{2}\\
s_{7}-s_{3}&=r_{2}\\
s_{2}-s_{7}&=-r_{5}-r_{4}-r_{3}\\
s_{5}-s_{2}&=-r_{6} +r_{5} +r_{4} +r_{3}-r_{2}-r_{1}.
\end{aligned}\right.
\end{equation}
 This system clearly possesses solutions in the $s_{k}$ in terms of the $r_{l}$, obtained for example by taking $s_{4}$ to be an arbitrary integer, and by rewriting the other $s_{k}$ in terms of $s_{4}$ and the $r_{l}$. For such a solution, the first six equations of the first column of~\reqref{syspthetas} are satisfied using \req{solutionN}. Using just~\reqref{solutionN} and~\reqref{basicS}, we now verify the remaining equations of~\reqref{syspthetas}. For example:
\begin{align*}
s_{4}-s_{5}+1=& -((s_{1}-s_{4})+(s_{6}-s_{1})+(s_{3}-s_{6})+ (s_{7}-s_{3})+ (s_{2}-s_{7})+(s_{5}-s_{2}))+1\\
=& -((-r_{6} +r_{4}+r_{3}-r_{2}+1)+(r_{6}-r_{3}+r_{2})+(r_{3}-r_{2})+r_{2}+\\
&(-r_{5}-r_{4}-r_{3})+ (-r_{6} +r_{5} +r_{4} +r_{3}-r_{2}-r_{1}))+1\\
=&r_{6}-r_{4}-r_{3}+r_{2}+r_{1}=p_{3,5}.%\\
\end{align*}
In a similar manner, one may check that the right-hand side of each of the equations of the system~\reqref{syspthetas} is equal to the left-hand side, using first~\reqref{basicS} to express the $s_{k}$ in terms of the $r_{l}$, and then using~\reqref{solutionN} to obtain the corresponding $p_{i,j}$. The straightforward computations are left to the reader.
So with this choice of $s_{k}$, we obtain values of the $\theta_{i,j}$ using \req{definesk} for which equations~\reqref{conjfrob1} and~\reqref{conjfrob2} are satisfied. Conversely, given arbitrary $r_{1},\ldots,r_{6}\in \Z$ and $s_{1},\ldots,s_{6}$ satisfying \req{basicS}, we see that if the $p_{i,j}$ are given by \req{syspthetas} and the $\theta_{i,j}$ are given by \req{definesk} then equations~\reqref{conjfrob1} and~\reqref{conjfrob2} are satisfied, and this completes the proof of the theorem. 
\end{proof}

%\begin{prop}
%There is only one conjugacy class of subgroups of $B_7/[P_7,P_7]$ isomorphic to the Frobenius group.
%\end{prop}

%\begin{proof}
%From Lemma ??? we can assume that the subgroup is of the form $\brak{x,v}$ where $v=Ny$ is such that $N$ satisfies the system of equations ??? . Any solution of that system of equations has the property that the sum of coefficients of one $y$-orbit is zero. Let $\mathcal{F}_0$ the group constructed in the proof of Thm ??? 
%\end{proof}

\begin{rem}
We saw in \reth{frob} that the Frobenius group $\mathcal{F}$ embeds in $B_7/[P_7, P_7]$. It is the only finite non-Abelian subgroup of $\sn[7]$ of odd order . To see this, besides $3\times 7$, which is the order of $\mathcal{F}$, the possible orders of non-Abelian subgroups of odd order of $\sn[7]$ are $3^2\times 5$, $3^2\times 5$, $3^2\times 7$, $3\times 5\times 7$ and $3^2\times 5\times 7$. Further, if $H$ is a subgroup of $\sn$ of odd order then it is necessarily a subgroup of $\an$. Indeed, any element $h\in H$ may be decomposed  as a product of disjoint cycles each of which is of odd length, and so it follows that $h\in \an$. From the table of maximal subgroups of $\an[7]$ given in~\cite[page 10]{atlas}, we see that $\sn[7]$ has no subgroup of order $3^2\times 5$, $3^2\times 7$, $3\times 5\times 7$ or $3^2\times 5\times 7$, and that if $\sn[7]$ possesses a subgroup $K$ of order $3^2\times 5$ then $K$ is a subgroup of $\an[6]$. It follows from the corresponding table for $\an[6]$ that there is no such subgroup  (see \cite[page 4]{atlas}).
%\comj{If $H$ is a subgroup of $\sn$ of odd order then it is necessarily a subgroup of $\an$. Indeed, any element $h\in H$ may be decomposed  as a product of disjoint cycles each of which is of odd length, and so it follows that $h\in \an$. Perhaps this should be taken into account in the following remark.} We saw in \reth{frob} that the Frobenius group embeds in $B_7/[P_7, P_7]$. It is the only finite non-Abelian group of odd order in $\sn[7]$. To see this, besides $3\times 7$, which is the order of the Frobenius group, we note that the possible orders of a non-abelian subgroup of odd order of $\sn[7]$ are $3^2\times 5$, $3^2\times 7$, $3\times 5\times 7$ and $3^2\times 5\times 7$. From the table which gives the order of the maximal subgroups of $\sn[7]$ if a subgroup has one of the orders $3^2\times 7$, $3\times 5\times 7$ and $3^2\times 5\times 7$ it must be a subgroup of $\an[7]$.  From the table which gives the order of the maximal subgroups of $\an[7]$ one easily see that $\sn[7]$ does not contain a subgroup of such order. For the case of a subgroup of order $3^2\times 5$ then if such group exists it must be a subgroup of $\an[7]$ or $\sn[6]$. As in the previous case the group cannot be realized as a subgroup of $\sn[7]$. 
%Finally from the table which gives the order of the maximal subgroups of $\sn[6]$ follows that the group cannot be realised as a subgroup of $\sn[6]$. 
%Fot the tables see the Atlas of finite Group Representations, \url{http://brauer.maths.qmul.ac.uk/Atlas/v3/alt/A7/} section alternated group $\an[6]$ and alternated group $\an[7]$ \comj{Put this in the references?}.
\end{rem}

%%%%%%%%%%%%%%%%%%%% Bibliografia

\end{document}